\documentclass[reqno]{amsart}

\usepackage{etoolbox}
\usepackage{fullpage,hyperref}
\usepackage{enumitem}

\patchcmd{\abstract}{\null\vfil}{}{}{}

\theoremstyle{plain}
\newtheorem{theorem}{Theorem}[section]
\newtheorem{proposition}[theorem]{Proposition}
\newtheorem{corollary}[theorem]{Corollary}
\newtheorem{lemma}[theorem]{Lemma}

\theoremstyle{definition}
\newtheorem{definition}[theorem]{Definition}

\theoremstyle{remark}
\newtheorem{remark}[theorem]{Remark}

\def\R{\mathbb R}

\def\Z{\mathbb Z}

\def\P{\mathbb P}
\def\X{\mathbb X}
\def\Q{\mathbb Q}
\def\bN{\mathbf N}
\def\PP{\mathcal P}
\def\cB{\mathcal B}
\def\md{\mathrm{d}}
\def\been{\begin{enumerate}}
\def\enen{\end{enumerate}}

\def\su{\subseteq}
\def\f{\frac}
\def\ff{\infty}
\def\E{\mathbb E}
\def\bet{\begin{theorem}}
\def\ent{\end{theorem}}
\def\bepr{\begin{proposition}}
\def\enpr{\end{proposition}}

\def\vp{\varphi}

\def\EE{\mathcal E}
\def\one{\mathbf{1}}
\def\ms{\mathsf}
\def\Var{\ms{Var}}
\def\Cov{\ms{Cov}}
\def\zz{\boldsymbol{z}}
\def\d{{\rm d}}
\def\es{\varnothing}
\def\sm{\setminus}
\def\dist{\operatorname{dist}}

\def\bel{\begin{lemma}}
\def\enl{\end{lemma}}
\def\bep{\begin{proof}}
\def\enp{\end{proof}}

\def\le{\leqslant}

\def\mc{\mathcal}
\usepackage{amsmath,amsfonts,amssymb,amsthm,mathrsfs,tikz,pgfplots,bm}
\pgfplotsset{compat=1.18}

\usepackage[foot]{amsaddr}
\usepackage{todonotes}
\usetikzlibrary{arrows.meta,calc,fillbetween}

\allowdisplaybreaks[4]

\title{Functional central limit theorems for Laguerre tessellations}
\author{Chinmoy Bhattacharjee}
\address{Department of Mathematics, Universit\"at Hamburg, Germany}
\author{Anna Gusakova}
\address{Institut f\"ur Mathematische Stochastik, Universit\"at M\"unster, Germany}
\author{Christian Hirsch}
\address{Department of Mathematics, Aarhus University, Aarhus, Denmark}
\email{hirsch@math.au.dk}
\subjclass[2020]{Primary 60D05, 60F17; Secondary 60G55}
\keywords{Poisson--Laguerre tessellations, extreme generators, Gaussian approximation, covariance asymptotics, functional central limit theorem}

\usepackage{environ}

\newif\ifdraft
\drafttrue

\NewEnviron{draftnote}{%
  \ifdraft
    \par\begingroup\color{red}\BODY\par\endgroup
  \fi
}

\begin{document}

\begin{abstract}
We study empirical processes of observable extreme generators in stationary Poisson--Laguerre tessellations, with each generator weighted either by one or by the volume of its cell. Under moment assumptions that allow heavy-tailed marks, we prove functional central limit theorems as the observation window grows, on compact intervals where the mark law has a bounded density. We also obtain covariance asymptotics with a surface-order error and quantitative univariate and multivariate Gaussian approximation.
\end{abstract}

\maketitle
\thispagestyle{empty}
\tableofcontents

\frenchspacing

%
%
\section{Introduction}
\label{sec:intro}

Laguerre tessellations are weighted analogues of Voronoi tessellations in which each generator carries a weight that determines its strength in the competition for space. They have been used to model foam structures \cite{Lautensack2008FoamStructures} and represent three-dimensional grain structures \cite{LyckegaardLauridsenLudwigFondaPoulsen2011Laguerre}. In the Poisson--Laguerre model, the generators form a homogeneous Poisson process with independent marks whose common probability law we denote by $\mathbb F$. This model was studied in \cite{zuyev}, and a broader class of Poisson--Laguerre tessellations was developed in \cite{GiWL25}. Inference for the mark law must account for generators with empty cells, which leave no directly observable contribution to the tessellation. The question addressed here is how empirical statistics based on observable generators fluctuate as the observation window grows.

Furthermore,  nonparametric estimators for the mark distribution  were introduced in \cite{vit} and their strong consistency was proved. Their constructions use observable generators and cell geometry; one of them is based on generators contained in their own cells. We call these generators \emph{observable extreme generators} and study two associated empirical statistics. The first counts observable extreme generators in an observation window whose marks do not exceed a prescribed level, and the second sums their cell volumes. Varying the level gives an unweighted and a volume-weighted empirical process.

We allow general real-valued mark distributions under the moment assumptions specified in our main results, including distributions with heavy tails. For growing convex observation windows, our contributions are as follows.
\begin{enumerate}
\item {In Theorem \ref{thm:func}}, we prove functional central limit theorems on compact intervals where the mark law has a bounded density. After centering and division by the square root of the window volume, each empirical process converges in the Skorokhod topology to a centered Gaussian process.
\item {In Theorem \ref{thm:var}}, we identify the limiting covariance kernels and establish covariance asymptotics with a surface-order error. The kernels describe the dependence between Gaussian fluctuations at different levels.
\item {In Theorem \ref{thm:univ} and Theorem \ref{thm:mult}}, we establish quantitative univariate and multivariate normal approximation. Under the stated nondegeneracy assumptions, the univariate bounds decay as the inverse square root of the window volume in both Wasserstein and Kolmogorov distance.
\end{enumerate}

The closest results on Gaussian approximation for Laguerre
tessellations were established in
\cite{BhattacharjeeGusakova2025GaussianApproximation},
where quantitative central limit theorems are proved for counts
of generators with nonempty cells in the $\beta$-, $\beta'$-,
and Gaussian-Voronoi models. Here, the observables select generators contained in their own cells and retain their marks as an indexing parameter. The resulting process limits describe fluctuations across an interval of mark levels for both counting and cell-volume statistics.

The main geometric difficulty is that a generator with a sufficiently negative mark can influence the tessellation far from its spatial location. The lower tail of the mark law therefore enters the estimates of dependence, and exponential localization cannot be assumed. For univariate Gaussian approximation, we use the region-stabilization framework of  \cite{BM22}, in which each contribution is determined by the points in a specified, possibly unbounded region. The same geometric estimates enter the multivariate Poisson normal-approximation bounds of  \cite{TT24}. Cell-volume weights require additional control because a generator's contribution depends on its entire cell. Following \cite{BhattacharjeeGusakova2025GaussianApproximation}, we use cover times, which record when the growing cell has reached its final shape, and establish tail and moment estimates that permit localization of these contributions.

To obtain convergence of the empirical processes, we control increments over short intervals of mark levels through covariance and fourth-cumulant bounds. Functional limit theorems for geometric statistics have also been established for Ripley's $K$-function \cite{BiscioSvane2022RipleyK}, persistent Betti numbers on cylindrical networks \cite{KrebsHirsch2022PersistentBetti}, iteration-stable tessellations \cite{SchreiberThale2013STIT}, and sparse random geometric graphs in high dimensions \cite{BonnetHirschRosenWillhalm2023SparseRGG}. In the present setting, the increment estimates must accommodate both the dependence created by the mark distribution and the unbounded weights given by cell volumes. The localization and moment estimates provide this control and allow the Gaussian approximation results to be extended to the full empirical processes.

The rest of the article is structured as follows. Section~\ref{sec:mod} introduces the model and states the main results. Section~\ref{sec:out} establishes the geometric and moment estimates, including bounds for cover times. Sections~\ref{sec:varlb} and~\ref{sec:var} prove the variance lower bounds and covariance asymptotics, respectively. Section~\ref{sec:univ_proof} proves the univariate and multivariate Gaussian approximation results, and Section~\ref{sec:tight} completes the functional limit theorems by proving tightness.

%
%
\section{Model definition and main results}
\label{sec:mod}
We study statistical fluctuations in Poisson–Laguerre tessellations, a weighted generalization of Voronoi tessellations in which each point of a Poisson process generates a cell according to a space–time competition rule.
Our focus is on extreme points, meaning those generating points that are contained in their own Laguerre cell. We now introduce the geometric framework, the observation model, and the empirical processes studied in this paper. Throughout, for sequences $(a_n)_{n\geq 1}$ and $(b_n)_{n\geq 1}$ with
$b_n>0$, we write $a_n=O(b_n)$ as $n\to\infty$ if there exist constants
$C<\infty$ and $n_0\in\mathbb N$ such that
\(
|a_n|\leq Cb_n
\text{ for all } n\geq n_0.
\)
Unless stated otherwise, the constants $C$ and $n_0$ may depend on fixed
parameters such as $d$, $\mathbb F$, $z$, and $I$, but not on $n$.

%
%
\subsection{Geometric framework}
We study the model of Poisson-Laguerre tessellations pioneered in \cite{zuyev}.
Throughout this paper, we write $x = (v,h)\in\R^{d + 1}$ with $v\in \R^d$ and $h\in\R$ (and similarly for $x_i$). Moreover, we assume that $d \ge 2$. The case $d = 1$ could also be considered using our approach but some case distinction would be needed for the rates. However, as $d=1$ is less interesting from the application perspective, we do not treat it here, thereby allowing for a more streamlined presentation.

For $\X:=\R^d\times\R$ and a locally finite counting measure $\mu$ on $\X$, write $\operatorname{supp}(\mu):=\{x\in\X:\mu(\{x\})>0\}$. In what follows we will identify $\mu$ with its support and write simply $x\in\mu$ instead of $x\in\operatorname{supp}(\mu)$. We denote the restriction of $\mu$ to a Borel set $A$ by $\mu|_A$ {(or $\mu\cap A$, while identifying $\mu$ with its support)} and its point count by $\mu(A)$. The symbol $0$ denotes the origin in a Euclidean space and the zero measure when used as a configuration.

The Laguerre tessellation is a generalization of the Voronoi tessellation, which allows us to account for different strengths of generator points. More precisely for a locally finite counting measure $\vp=\sum_i\delta_{x_i}$ of generator locations $v_i$ together with arrival times $h_i$, we define 
\begin{align}
\label{eq:cell}
 C(x; \vp) := \big\{u \in \R^d: L(x, u) \le \inf_{x' \in \vp }L(x', u)\big\},
\end{align}
to be the \emph{Laguerre cell} associated with $x \in \R^{d + 1}$ where 
$$L(x, u) := \|v - u\|^2 + h \quad \text{for } x = (v,h)$$
is the Laguerre distance.  Here, the mark $h$ acts as a vertical shift, so that points with smaller $h$ have a competitive advantage.

%
%
\subsection{Extreme points}
In contrast to the Voronoi case, a generator may have an empty Laguerre cell. Even among generators with nonempty cells, the generator location need not be contained in its own cell. We call a generator \(x=(v,h)\) extreme if \(v\in C(x;\varphi)\), and write
\begin{align}
\label{eq:ext}
\EE(\vp) := \{x \in \vp : v \in C(x; \vp)\}
\end{align}
to denote the subset of generators that are contained in their own Laguerre cell. In other words, a point
\(x=(v,h)\) belongs to \( \EE(\vp)\) precisely when it wins the Laguerre competition at its own spatial location
\(v\). Equivalently, for $x\in\vp$,
\[
x\in \EE(\vp)
\quad\Longleftrightarrow\quad
\vp(D(x))=0,
\]
where
\begin{equation}\label{eq:D}
D(x) := \bigl\{(v',h') \in \mathbb{R}^{d+1} : h - h' > \|v - v'\|^2 \bigr\},
\end{equation}
is an open region below the standard downward paraboloid with apex $(v,h)$ (hypograph).
Thus, \(x\) is extreme exactly when there is no competing point below the downward paraboloid with apex
at \(x\); any such point would beat \(x\) at its own location \(v\). For a test point $x$ that is not assumed to belong to $\varphi$, we write $x\in\EE(\varphi+\delta_x)$ explicitly. This event is equivalent to $\varphi(D(x))=0$. The cell $C(x;\varphi)$ is defined for every $x\in\X$ and satisfies $C(x;\varphi+\delta_x)=C(x;\varphi)$.

%
%
\subsection{Poisson input}
In this paper, we analyze statistical properties of the \emph{Poisson-Laguerre tessellation}, i.e., the Laguerre tessellation built on a marked homogeneous Poisson point process $\PP=\sum_i\delta_{X_i}$ on $\X$, where $X_i=(V_i,H_i)$. Here, the locations $\{V_i\}_i$ form a unit-intensity homogeneous Poisson point process on $\R^d$ that is independently marked by real-valued i.i.d.\ arrival times $\{H_i\}_i\su \R$ with some typical mark distribution $\mathbb F$. We write $H\sim\mathbb F$ for a typical mark.
Equivalently, $\mathcal{P}$ is a Poisson point process in $\R^{d+1}$ with intensity measure $\mathbb{Q}:=\lambda\otimes \mathbb F$, where $\lambda$ denotes the Lebesgue measure on $\R^d$.
Let $\kappa_d$ be the volume of the unit ball in $\R^d$.
Henceforth, since the measure $\Q$ is stationary with respect to the spatial coordinate $v$, we will use the notation 
\begin{equation}\label{eq:Qdef}
\Q(h) := {\Q}(D(v,h)) = \kappa_d \E[(h-H)_+^{d/2}]
\end{equation}
with $D$ defined at \eqref{eq:D}. For each fixed \(x\), the
boundary \(\partial D(x)\) has \(\Q\)-measure zero. The set of pairs
\((x,x')\) for which \(L(x',v)=L(x,v)\) or
\(L(x,v')=L(x',v')\) has \(\Q\otimes\Q\)-measure zero. Consequently,
the difference between strict and nonstrict versions of the inequalities in \eqref{eq:D} is irrelevant for most of our arguments.

It is clear from the definition of Laguerre cells \eqref{eq:cell} that the properties of the mark distribution $\mathbb F$ have a major impact on geometrical properties of the resulting Laguerre tessellation. In what follows, we use the notation $x_-:=\max (0,-x)$ and $x_+=\max(0,x)$. Henceforth, we always assume $\E[H_-^{d/2}] < \ff$, thus guaranteeing that $\PP$ indeed generates a tessellation. When $\mathbb F$ is a probability measure on $(-\infty,0]$ this was shown in \cite[Theorem 4.1]{zuyev}. When $\mathbb F$ is a probability measure on $\R$ this can be deduced from \cite[Theorem 3.3]{GiWL25}. It should be noted that in \cite{GiWL25} it was assumed that $\mathbb F$ has a density with respect to Lebesgue measure on $\mathbb{R}$, but the proof of Theorem 3.3 does not rely on this fact and works in the same way for general $\mathbb F$. The condition of \cite[Theorem 3.3]{GiWL25} in the general case reads as
\begin{equation}\label{eq:ConditionH}
\E[(t-H)_+^{d/2}]=\int_{-\infty}^t(t-h)^{d/2}\mathbb F(\md h)<\infty\qquad\text{ for all }t\in\R,
\end{equation}
which is equivalent to $\E[H_-^{d/2}]<\infty$ for probability measures $\mathbb F$. Indeed, \eqref{eq:ConditionH} implies $\E[H_-^{d/2}]<\infty$ by taking $t=0$. On the other hand if $\E[H_-^{d/2}]<\infty$ then for any $t\leq 0$ we have $\E[(t-H)_+^{d/2}]\leq \E[H_-^{d/2}]<\infty$ and by applying the inequality $(a+b)^{d/2}\leq 2^{d/2-1}(a^{d/2}+b^{d/2})$ for $a,b>0$ for $t > 0$  we get 
\begin{align*}
\E[(t-H)_+^{d/2}]&=\E[(t-H)_+^{d/2}\one\{H\ge 0\}]+\E[(t-H)_+^{d/2}\one\{H<0\}]\leq t^{d/2}+2^{d/2-1}t^{d/2}+2^{d/2-1}\E[H_-^{d/2}]<\infty.
\end{align*}

\subsection{Empirical functionals}
In a statistical context, it is difficult to estimate \(\mathbb F\) directly from tessellation data observed in a window
\(W_n=nW\), where \(W\subseteq\mathbb R^d\) is a compact convex body containing  the origin $0$ in its interior.

The reason is that some generators are not
observable at all, being completely covered by other cells. However, \cite{vit} showed that the distribution
\(\mathbb F\) is uniquely determined by a suitable transformation that can be estimated from observable geometric
information. 
The idea is to restrict the estimation to the set \eqref{eq:ext} of extreme generators contained in their own cells. 
For \(z\in\mathbb R\), let $W_n^z := W_n \times (-\infty,z]$. We define
\begin{align}
\label{eq:gn}
	G_n(z) := \sum_{\substack{X_i \in W_n^z\cap\PP}}\one\{X_i \in \EE(\PP)\}
\end{align}
which counts the extreme generators in the observation window \(W_n\) whose marks do not exceed \(z\). While \(G_n(z)\) is the most direct empirical functional, it may be sensitive to observational noise. Indeed,
under the unweighted statistic, extreme points with very small Laguerre cells contribute just as much as those
with large cells, even though the former are typically more susceptible to geometric perturbations. Following
\cite{vit}, we therefore also consider the volume-weighted version
\begin{align}
\label{eq:gnv}
	G_n^{\ms V}(z) :=\sum_{\substack{X_i \in W_n^z\cap \PP}}\one\{X_i \in \EE(\PP)\}|C(X_i; \PP)|,
\end{align}
where $|A|=\lambda(A)$ denotes the Lebesgue measure of a set $A\subset \mathbb{R}^d$.
This weighted
functional emphasizes geometrically stable extreme points by assigning larger weight to generators with larger
Laguerre cells. Thus, \(G_n(z)\) records the number of observable extreme generators up to level \(z\),
whereas \(G_n^{\ms V}(z)\) weights these generators by the size of the cells they produce.

\subsection{Main results}
We state the results in the order in which they are used later: first a fixed-level one-dimensional Gaussian approximation, then covariance asymptotics, then multivariate Gaussian approximation, and finally a functional central limit theorem. Throughout this subsection, moment assumptions on $H_-$ ensure existence of the tessellation and finiteness of the relevant moments; stronger assumptions are imposed only to obtain quantitative rates. Unless stated otherwise, the constants in the $O$-notations may depend on fixed parameters such as $d$, $\mathbb F$, $z$, and~$I$.

We begin with the fixed-level regime. For $z \in \R$, define
\[
\hat G_n(z) := \f{G_n(z)-\E[G_n(z)]}{\sqrt{\Var(G_n(z))}},
\qquad
\hat G_n^{\ms V}(z) := \f{G_n^{\ms V}(z)-\E[G_n^{\ms V}(z)]}{\sqrt{\Var(G_n^{\ms V}(z))}}.
\]

{To compare the distributions of $\hat G_n(z)$ and $\hat G_n^{\ms V}(z)$ with that of a standard normal random variable, we provide non-asymptotic bounds on distances of the form 
\begin{equation}\label{eq:dH}
	d_{\mathcal{H}}(\mu,\nu) := \sup_{h \in {\mathcal{H}}}|\E h(Y)-\E h(X)|,
\end{equation}
where $\mathcal{H}$ is a large enough class  of real-valued functions on $\R$ and $X \sim \mu$ and $Y \sim \nu$ are real-valued random variables. With a slight abuse of notation, we will write $d_{\mathcal{H}}(X,Y)$ for $d_{\mathcal{H}}(\mu,\nu)$. The \textit{Wasserstein distance} $d_W$ is obtained from \eqref{eq:dH} by taking $\mathcal{H} = \{h: |h(x)-h(y)| \le |x-y|\}$ as the class of all Lipschitz-1 functions, while the \textit{Kolmogorov} distance $d_K$ corresponds to $\mathcal{H} = \{\one\{\cdot \le x\}, x \in \R\}$.

For our multivariate results, we use the $d_3$ distance defined for random vectors $X, Y \in \R^d$ with finite second moments ($\E \|X\|^2, \E \|Y\|^2 <\infty$) as
$$
d_3(X,Y) := \sup_{h \in \mathcal{H}_d^3} |\E h(Y)-\E h(X)|. 
$$
Here $\mathcal{H}_d^3$ denotes the set of all $\mathcal{C}^3$ functions $h: \R^d \to \R$ with $\|h^{(2)}\|_\infty$ and $\|h^{(3)}\|_\infty$ bounded by 1, where $\|h^{(k)}\|_\infty, k =2,3,$ denotes the supremum of the absolute values of
all $k$-th partial derivatives. }

The first result gives presumably optimal rates for Gaussian approximation at each fixed-level $z \in \R$. { }
%
%
\bet[Univariate normal approximation]
\label{thm:univ}
Let \(z\in\mathbb R\) be such that \(\mathbb F((-\infty, z])>0\).
Let $d_*=\max\{d_W, d_K\}$.
\begin{enumerate}
\item Assume that $\E [H_-^{3d}]<\infty$. 
Then, 
$d_*\big(\hat G_n(z), N(0,1)\big) \in O\big({|W_n|^{-1/2}}\big).$
\item Assume that $\E[H_-^{m_0}]<\infty$ for some $m_0>11d/2$. 
Assume additionally that $\mathbb F$ is not a Dirac measure.
Then,
$d_*\big(\hat G^{\ms V}_n(z), N(0,1)\big) \in O\big({|W_n|^{-1/2}}\big).$
\end{enumerate}
\ent

Next, we consider the distributional convergence at multiple locations $z \in \R$. While for the univariate result we have normalized by the standard deviation, for the multivariate and functional results it will be more convenient to normalize by $\sqrt{|W_n|}$ instead; see Remark~\ref{rem:scal}. Accordingly, for $z \in \R$ we define
\[
\hat G_{n,*}(z) := \f{G_n(z)-\E[G_n(z)]}{\sqrt{|W_n|}},
\qquad
\hat G_{n,*}^{\ms V}(z) := \f{G_n^{\ms V}(z)-\E[G_n^{\ms V}(z)]}{\sqrt{|W_n|}}.
\]

As a first step, we describe the covariance asymptotics. 
The limit covariance admits a decomposition into a diagonal contribution, an off-diagonal covariance term,
and two correction terms accounting for the fact that the joint event uses a common underlying configuration. 
To that end, writing $\R^{d, z}:= \R^d \times(-\infty, z]$, we introduce the covariance kernels
\begin{align*}
    k(z, z') &:= \int_{-\infty}^{z \wedge z'} \P((0, h) \in \EE(\PP+\delta_{(0, h)})) \mathbb F(\d h) \\
    &\quad+ \int_{-\infty}^z \int_{\R^{d, z'}}
    \Cov\big(\one\{(0, h) \in \EE(\PP+\delta_{(0, h)})\}, \one\{x' \in \EE(\PP+\delta_{x'})\}\big)
    \Q(\d x')\mathbb F(\d h) \\
    &\quad- \int_{-\ff}^z \int_{\R^{d, z'}}
    \E\big[\one\big\{\PP(D(0, h)\cup D(x'))=0\big\}\big]
    \one\{L(x',0) < L((0,h),0)\}
    \Q(\d x')\mathbb F(\d h) \\
    &\quad- \int_{-\ff}^{z'} \int_{\R^{d,z}}
    \E\big[\one\big\{\PP(D(x)\cup D(0,h'))=0\big\}\big]
    \one\{L(x,0) < L((0,h'),0)\}
    \Q(\d x)\mathbb F(\d h').
\end{align*}
where the region $D(x)\subseteq \R^{d+1}$ is defined at \eqref{eq:D}. Similarly, for the volume-weighted case,  denote
$$A(x; \vp) := \one\{x\in \mc E(\vp+\delta_x)\}|C(x; \vp)|$$ 
for locally finite $\vp$. Note that we add here the point $x$ to the configuration $\vp$ in the definition  for notational convenience.
To ease notation, we also write $A_x := A(x; \PP)$, and define 
\begin{align*}
    k^{\ms V}(z, z')&:= \int_{-\infty}^{z \wedge z'} \E[A_{(0, h)}^2] \mathbb F(\d h) + \int_{-\infty}^z \int_{\R^{d, z'}}  \Cov\big(A_{(0, h)}, A_{x'} \big)  \Q(\d x')\mathbb F(\d h)  \\\\
	&\quad+ \int_{-\infty}^z \int_{\R^{d, z'}} \E\big[A\big((0, h); \PP+\delta_{x'}\big) A\big(x'; \PP+\delta_{(0,h)}\big) - A_{(0, h)}A_{x'} \big]   \Q(\d x')\mathbb F(\d h). \\
\end{align*}
As stated in Theorem \ref{thm:var} below, the kernels $k$ and $k^{\ms V}$ are finite under suitable moment assumptions on $H_-$. 
Now, we can state the covariance asymptotics with bounds on the rate of convergence.
%
%

\bet[Covariance asymptotics] 
\label{thm:var}
Let $z, z' \in \R$.
\begin{enumerate}
    \item  Assume that $\E[H_-^{3d/2}]<\infty$. 
    Then, $k(z,z') < \infty$ and
$$\Big|\f1{|W_n|} \Cov(G_n(z), G_n(z')) - k(z, z')\Big| \in O\!\big(|W_n|^{-1/d}\big).$$
Moreover, for any $z$ with $\mathbb F((-\infty,z])>0$, we have $k(z, z) > 0$.
    \item  Assume that $\E[H_-^{m_0}]<\infty$ for some $m_0 > 11d/2$. Then, $k^{\ms V}(z,z') < \infty$ and
$$\Big|\f1{|W_n|} \Cov(G_n^{\ms V}(z), G_n^{\ms V}(z')) - k^{\ms V}(z, z')\Big| \in O\!\big(|W_n|^{-1/d}\big).$$
Assume additionally that $\mathbb F$ is not a Dirac measure. If $\mathbb F((-\infty,z]) > 0$, then  $ k^{\ms V}(z,z) > 0$.
\end{enumerate}
\ent
While Theorem \ref{thm:var} implies that the kernels $k$ and $k^{\ms V}$ are positive semi-definite, we leave criteria for positive definiteness to future research. Next, we show multivariate normal approximation with the asserted asymptotic covariance kernels.

With the covariance kernels in hand, we obtain multivariate normal approximation for vectors indexed by finitely many levels. Below, for $\zz = (z_1, \dots, z_m) \in \R^m$, we denote $\Sigma(\zz): = (k(z_i,z_j))_{1 \le i,j \le m}$ and $\Sigma^{\ms V}(\zz): = (k^{\ms V}(z_i,z_j))_{1 \le i,j \le m}$.

%
%
\bet[Multivariate normal approximation] 
\label{thm:mult}
Let $\zz = (z_1, \dots, z_m) \in \R^m$. 
\begin{enumerate}
    \item  Assume that $\E[H_-^{3d}]<\infty$. Then,
$$d_3\big(\{\hat G_{n, *}(z)\}_{z \in \zz}, N\big(\boldsymbol{0}, \Sigma(\zz)\big)\big) \in O\big({|W_n|^{-1/2}}\big) + \sum_{i,j=1}^m
\left|
k(z_i,z_j)
-
\frac{\Cov\big(G_n(z_i),G_n(z_j)\big)}{|W_n|}
\right|. $$
    \item  Assume that $\E[H_-^{m_0}]<\infty$ with $m_0>11d/2$. Then,
$$d_3\big(\{\hat G_{n, *}^{\ms V}(z)\}_{z \in \zz}, N\big(\boldsymbol{0}, \Sigma^{\ms V}(\zz)\big)\big) = O\big( {|W_n|^{-1/2}}\big) + \sum_{i,j=1}^m
\left|
k^{\ms V}(z_i,z_j)
-
\frac{\Cov\big(G_n^{\ms V}(z_i),G_n^{\ms V}(z_j)\big)}{|W_n|}
\right|.$$
\item In particular, under the respective assumptions in (1) and (2),
$$
\max\{d_3\big(\{\hat G_{n, *}(z)\}_{z \in \zz}, N\big(\boldsymbol{0}, \Sigma(\zz)\big)\big), d_3\big(\{\hat G_{n, *}^{\ms V}(z)\}_{z \in \zz}, N\big(\boldsymbol{0}, \Sigma^{\ms V}(\zz)\big)\big)\} \in O(|W_n|^{-1/d}).
$$
\end{enumerate}
\ent

As is evident from the third assertion, the bottleneck for the rate of convergence in Theorem \ref{thm:mult} is the covariance convergence.  Loosely speaking, the covariance term decays at surface order, which is typical for random spatial models. We refer to \cite[Theorem 2]{TT24} for another example involving random geometric graphs.

Finally, we upgrade the finite-dimensional convergence to a functional limit theorem.
%
\bet[Functional CLT] 
\label{thm:func}
Let $I\su \R$ be a compact interval. Assume that $\mathbb F$ admits a density $f$ on $I$ satisfying $\|f\|_\infty<\infty$.
\begin{enumerate}
    \item Assume that $\E[H_-^{3d}]<\infty$. Let $Z(\cdot)$ be a centered Gaussian process on $I$ with covariance function $k(\cdot,\cdot)$. Then, in the Skorokhod topology, $\hat G_{n,*}(\cdot)\Rightarrow Z(\cdot)$.
    \item Assume that $\E[H_-^{m_0}]<\infty$ for some $m_0 > 11d/2$. Let $Z^{\ms V}(\cdot)$ be a centered Gaussian process on $I$ with covariance function $k^{\ms V}(\cdot,\cdot)$. Then, in the Skorokhod topology, $\hat G_{n,*}^{\ms V}(\cdot)\Rightarrow Z^{\ms V}(\cdot)$.
\end{enumerate}
\ent

\begin{remark}[Moment assumptions]
The moment assumptions become stronger when passing from the unweighted to
the volume-weighted setting. For the unweighted Gaussian approximation and
functional limit theorem, a moment of order \(3d\) is sufficient, whereas
the weighted results require higher moments because the cell volume depends
on the geometry of the entire stabilized cell. We make no claim that the
moment exponents appearing here are sharp. They are the exponents needed to
close the moment, covariance, and tightness estimates used in our proofs.
In particular, they should be viewed as sufficient conditions for the
present approach rather than as intrinsic thresholds of the model. The
covariance asymptotics in Theorem~2.2 provide the main input for the
multivariate and functional results.
\end{remark}

\begin{remark}[Normalization]
    \label{rem:scal}
    As mentioned above, for the multivariate and functional case, it is more convenient to consider the normalization of $G_n(z) -\E[G_n(z)]$ by $\sqrt{|W_n|}$ instead of $\sqrt{\Var(G_n(z))}$. This is because otherwise, the denominator would be $z$-dependent, making the required increment bounds substantially more cumbersome. Conversely, in the univariate case it makes sense to normalize by the standard deviation. This is because this normalization avoids additional error terms coming from the convergence of the variance. Overall, we believe that this trade-off justifies the different normalizations. Since $\mathrm{Var}(G_n(z)) \sim |W_n|\,k(z,z)$ by Theorem~\ref{thm:var}, the two normalizations differ only by the factor $\sqrt{k(z,z)}$ at each fixed $z$.

\end{remark}

\def\le{\leqslant}
\def\es{\varnothing}
\def\e{\varepsilon}
\def\sm{\setminus}
\def\mc{\mathcal}
\def\su{\subseteq}
\def\f{\frac}
\def\tf{\tfrac}
\def\be{\beta}
\def\vp{\varphi}
\def\k{\kappa}
\def\Cov{\ms{Cov}}
%
%

\def\mc{\mathcal}
\def\ge{\geqslant}
\section{Preliminary results}
\label{sec:out}

First, we introduce some notation used frequently henceforth. 
For $v\in\R^d$ let $B_r(v):=\{u\in\R^d:\|u-v\|\le r\}$ be the closed ball of radius $r\ge0$, $B_r^\circ(v):=\{u\in\R^d:\|u-v\|<r\}$ be the open ball of radius $r>0$, and $\partial B_r(v)$ be the corresponding sphere. We write $\operatorname{dist}(A,B):=\inf\{\|a-b\|:a\in A,b\in B\}$ and $\operatorname{dist}(v,B):=\operatorname{dist}(\{v\},B)$ and $S^{d-1}:=\partial B_1(0)$. We denote $(d-1)$-dimensional Hausdorff measure by $\mathcal H^{d-1}$. For $\rho>0$, write $W_n^{\ominus\rho}:=\{v\in W_n:\operatorname{dist}(v,\partial W_n)>\rho\}$ for its interior erosion and $W_n^\complement:=\R^d\setminus W_n$ for its spatial complement. 

\subsection{Coverage time}
\label{ss:coverage}
In contrast to $G_n(z)$, studying the volume-weighted case is substantially more involved. This is because the score function associated with $G_n^{\ms V}(z)$ depends on the volume of the Laguerre cell, which is possibly influenced by points outside the region $D(x)$ defined in \eqref{eq:D}. To address this issue, we will use the \textit{cover time} introduced in \cite{BhattacharjeeGusakova2025GaussianApproximation}, which quantifies how far one needs to look outside $D(x)$ to determine the Laguerre cell of $x$. We briefly recall its definition and main properties below.

Let \(\varphi\) be a locally finite counting measure on
\(\mathbb R^d\times\mathbb R\), and let \(x=(v,h)\in\mathbb R^d\times\mathbb R\).
For \(t\in\mathbb R\), write
\[
\varphi_{\le t}:=\varphi\big|_{\mathbb R^d\times(-\infty,t]}.
\]
{We may interpret the Laguerre tessellation as the result of a crystal growth process, where each generator $x$ produces a crystal (cell) $C(x;\varphi)$ according to the following rule: at time $h$ the crystal starts its growth from $v$ isotropically with speed $1/(2\sqrt{t})$. Each point in space belongs to the crystal that reaches it first.} Then the cell of \(x\) by time \(t\) is described by the following set
\begin{equation}\label{eq:truncated-cell-corrected}
C_t(x;\varphi)
:=
\begin{cases}
\Big\{u\in\mathbb R^d:
L(x,u)\le t,\quad
L(x,u)\le \displaystyle\inf_{y\in\varphi_{\le t}}L(y,u)
\Big\},&t\ge h,\\[0.4em]
\varnothing,&t<h.
\end{cases}
\end{equation}
The constraint \(L(x,u)\le t\) is essential. Without it, restricting only the
competitors to marks at most \(t\) generally produces a set larger than the full
Laguerre cell and does not describe the crystal grown by time \(t\).

\begin{lemma}[Geometry of the truncated cell]\label{lem:truncated-cell-geometry}
For every \(t\ge h\),
\begin{equation}\label{eq:truncated-cell-identities}
\begin{aligned}
C_t(x;\varphi)
&=
\Big\{u\in B_{\sqrt{t-h}}(v):
L(x,u)\le \inf_{y\in\varphi_{\le t}}L(y,u)\Big\}
=C(x;\varphi)\cap B_{\sqrt{t-h}}(v).
\end{aligned}
\end{equation}
Consequently, \(C_t(x;\varphi)\subseteq C_{t'}(x;\varphi)\) whenever
\(t\le t'\), and
\(
\bigcup_{t\in\mathbb R}C_t(x;\varphi)=C(x;\varphi).
\)
\end{lemma}

\begin{proof}
The condition \(L(x,u)\le t\) is equivalent to
\(\|u-v\|\le\sqrt{t-h}\). Moreover, if \(y=(w,s)\) has \(s>t\), then
\[
L(y,u)\ge s>t\ge L(x,u).
\]
Thus, on \(\{L(x,u)\le t\}\), generators with mark larger than \(t\) cannot
beat \(x\) at $u$. Comparison with \(\varphi_{\le t}\) is therefore equivalent to
comparison with the full configuration \(\varphi\), which proves
\eqref{eq:truncated-cell-identities}. The monotonicity and exhaustion statements
follow immediately.
\end{proof}

The \emph{cover time} of \(x\) is defined by
\begin{equation}\label{eq:covertime}
T(x;\varphi)
:=
\begin{cases}
-\infty,&C(x;\varphi)=\varnothing,\\[0.2em]
\inf\big\{t\ge h:C_t(x;\varphi)=C(x;\varphi)\big\},
&C(x;\varphi)\ne\varnothing.
\end{cases}
\end{equation}
The infimum may equal \(+\infty\) for an unbounded cell. If the cell is nonempty,
then
\begin{equation}\label{eq:cover-time-radius-representation}
T(x;\varphi)
=
\sup_{u\in C(x;\varphi)}L(x,u)
=
h+\sup_{u\in C(x;\varphi)}\|u-v\|^2.
\end{equation}
In particular, for \(t\ge h\),
\begin{equation}\label{eq:cover-time-containment}
T(x;\varphi)\le t
\quad\Longleftrightarrow\quad
C(x;\varphi)\subseteq B_{\sqrt{t-h}}(v).
\end{equation}
Whenever \(T(x;\varphi)<\infty\),
\begin{equation}\label{eq:cell-at-cover-time}
C(x;\varphi)=C_{T(x;\varphi)}(x;\varphi),
\qquad
|C(x;\varphi)|
\le
\kappa_d\big(T(x;\varphi)-h\big)_+^{d/2}.
\end{equation}

\begin{proof}[Proof of \eqref{eq:cover-time-radius-representation}--\eqref{eq:cell-at-cover-time}]
By Lemma~\ref{lem:truncated-cell-geometry}, the equality
\(C_t(x;\varphi)=C(x;\varphi)\) is equivalent to
\(C(x;\varphi)\subseteq B_{\sqrt{t-h}}(v)\), which is in turn equivalent to
\(\sup_{u\in C(x;\varphi)}L(x,u)\le t\). Taking the smallest such \(t\)
gives \eqref{eq:cover-time-radius-representation} and
\eqref{eq:cover-time-containment}. The remaining assertions follow by putting
\(t=T(x;\varphi)\).
\end{proof}


\begin{lemma}[Boundary characterization of the cover time]
\label{lem:cover-time-boundary-characterization}
Assume that \(\varphi\) generates a Laguerre tessellation, that
\(x=(v,h)\in\varphi\), and that \(v\in\operatorname{int}C(x;\varphi)\).
Then, for every \(t\ge h\),
\begin{equation}\label{eq:cover-time-boundary-equivalence}
T(x;\varphi)>t
\quad\Longleftrightarrow\quad
\partial B_{\sqrt{t-h}}(v)
\not\subseteq
\bigcup_{y\in\operatorname{supp}(\varphi)\setminus\{x\}}C_t(y;\varphi).
\end{equation}
Hence, if \(T(x;\varphi)>t\), then there exists
\(w\in\partial B_{\sqrt{t-h}}(v)\) such that
\(
\varphi(D(w,t))=0,
\)
where $D$ is defined in \eqref{eq:D}.
\end{lemma}

\begin{proof}
Suppose first that \(T(x;\varphi)>t\). By
\eqref{eq:cover-time-containment}, there exists
\(u\in C(x;\varphi)\) with \(\|u-v\|>\sqrt{t-h}\). Since
\(v\in\operatorname{int}C(x;\varphi)\) and the cell is convex, the point
\[
w=v+\frac{\sqrt{t-h}}{\|u-v\|}(u-v)
\]
belongs to \(\operatorname{int}C(x;\varphi)\). Hence, \(x\) is the unique
minimizer of the Laguerre distance at \(w\), so \(w\) belongs to no other cell
at time \(t\). This proves the forward implication.

Conversely, suppose that a point
\(w\in\partial B_{\sqrt{t-h}}(v)\) belongs to no
\(C_t(y;\varphi)\), \(y\ne x\). Since the Laguerre cells cover
\(\mathbb R^d\), the point \(w\) must be won uniquely by \(x\), and hence
\(w\in\operatorname{int}C(x;\varphi)\). The cell therefore contains points
immediately beyond the sphere, so it is not contained in
\(B_{\sqrt{t-h}}(v)\). By \eqref{eq:cover-time-containment},
\(T(x;\varphi)>t\).
Finally, if \(w\) is not covered by any competing cell at time \(t\), then no
\(y\in\operatorname{supp}(\varphi)\setminus\{x\}\) satisfies \(L(y,w)\le t=L(x,w)\). 
\end{proof}

For the Poisson input used below, the spatial intensity is diffuse. Hence, for
fixed \(x=(v,h)\), on the event \(x\in\EE(\PP+\delta_x)\), the point \(v\)
lies in the interior of its cell almost surely. Thus,
Lemma~\ref{lem:cover-time-boundary-characterization} applies almost surely on
the event occurring in Proposition~\ref{prop:CoverTimetail}.

 The following proposition provides exponential tail bounds for the cover time, uniformly over the event that the point is extreme.

%
%
\begin{proposition}[Tail bound for the cover time]
  \label{prop:CoverTimetail}
    Let $x=(v,h)\in\R^d\times\R$. Then,
    \begin{itemize}
        \item[(i)] $  \sup_{t\le h} \P
        \big(x\in\EE(\PP+\delta_{x}), T(x;\PP)>t\big)\le   \exp(-\Q(h));$
    \item[(ii)] and for $t> h$, 
    \[
    \P\big(x\in\EE(\PP+\delta_{x}), T(x;\PP)>t\big)\le  C\inf_{0<a\le  \sqrt{t-h}}\Big\{(a^{-1}\sqrt{t-h})^{d-1}\exp\Big(-{1\over 2^d}\Q(t-4a^2)\Big)\Big\};
    \]
    \item[(iii)] in particular, for $t>h$,
    \[
    \P\big(x\in\EE(\PP+\delta_{x}), T(x;\PP)>t\big)\le  C\exp\Big(-{1\over 2^d}\Q\Big({h+t\over 2}\Big)\Big).
    \]
    \end{itemize}
    Here $C>0$ is a constant depending only on $d$.
\end{proposition}

Choosing $a$ proportional to $\sqrt{t-h}$ balances the polynomial prefactor and the exponential decay.

\begin{remark}[Estimate for $\Q(h)$]\label{rem:Qlb}
For any $x_0$ such that $\mathbb F((-\infty,x_0])>0$, and any $h\ge x_0$, we have
\begin{align*}
\Q(h)&\ge\kappa_d\int_{-\infty}^{x_0}((h-x_0)+(x_0-s))^{d/2}\mathbb F(\md s)\ge \kappa_d \mathbb F((-\infty,x_0])(h-x_0)^{d/2}.
\end{align*}
Note also that without loss of generality we may assume $x_0 \ge 0$. 
\end{remark}
\begin{proof}[Proof of Proposition \ref{prop:CoverTimetail}]
Without loss of generality, let \(v=0\) and write \(x=(0,h)\). For \(t\le h\),
\[
\{x\in\EE(\PP+\delta_{x}),\ T(x;\PP)>t\}\subseteq\{x\in\EE(\PP+\delta_{x})\},
\]
and therefore
\(
\P\big(x\in\EE(\PP+\delta_{x}),\ T(x;\PP)>t\big)
\le \P\big(x\in\EE(\PP+\delta_{x})\big)=e^{-\Q(h)}.
\)
This proves (i).

Now let \(t>h\). By Lemma~\ref{lem:cover-time-boundary-characterization},
\begin{equation}\label{eq:cover-time-to-empty-paraboloid}
\{x\in\EE(\PP+\delta_{x}),\ T(x;\PP)>t\}
\subseteq
\Big\{\exists w\in\partial B_{\sqrt{t-h}}(0):
\PP(D(w,t))=0\Big\}.
\end{equation}
Fix \(0<a\le\sqrt{t-h}\), and let \(\mathcal N_a\) be an \(a\)-net of
\(\partial B_{\sqrt{t-h}}(0)\) of cardinality at most $C\big(a^{-1}\sqrt{t-h}\big)^{d-1}$. For \(u\in\mathcal N_a\), define
\[
K(a,u):=\Big\{(v',s):s\le t-a^2,\quad
\|v'-u\|<\sqrt{t-s}-a\Big\}.
\]
If \(w\in B_{a}(u)\), then \(K(a,u)\subseteq D(w,t)\). Each spatial boundary has zero Lebesgue measure, so the strict inequality in the definition of $K(a,u)$ leaves its $\Q$-measure unchanged, even when $\mathbb F$ has atoms. Hence, the event on
the right-hand side of \eqref{eq:cover-time-to-empty-paraboloid} implies that
\(\PP(K(a,u))=0\) for at least one \(u\in\mathcal N_a\). Moreover,
\(
\Q(K(a,u))=\kappa_d\int_{-\infty}^{t-a^2}
\big(\sqrt{t-s}-a\big)^d\mathbb F(\md s).
\)
For \(s\le t-4a^2\), one has \(a\le\frac12\sqrt{t-s}\), and therefore
\[
\Q(K(a,u))\ge 2^{-d}\kappa_d
\int_{-\infty}^{t-4a^2}(t-s)^{d/2}\mathbb F(\md s)
\ge 2^{-d}\Q(t-4a^2).
\]
Consequently,
\(
\P\big(\PP(K(a,u))=0\big)
\le \exp\big(-2^{-d}\Q(t-4a^2)\big).
\)
The union bound now yields
\[
\P\big(x\in\EE(\PP+\delta_{x}),\ T(x;\PP)>t\big)
\le C\big(a^{-1}\sqrt{t-h}\big)^{d-1}
\exp\big(-2^{-d}\Q(t-4a^2)\big),
\]
which proves (ii). For (iii), choose
\(
a=\sqrt{t-h}/(2\sqrt2).
\)
Then, \(t-4a^2=(t+h)/2\), while the net-cardinality factor is bounded by a
constant depending only on \(d\). This proves (iii).
\end{proof}

The exponential tail bounds yield moment estimates for the excess cover time.
The following corollary records the estimates used later in the analysis of the volume-weighted score.
\begin{corollary}[Moment bound for the cover time]\label{lem:cover-time-moment-integrability}
Fix $z\in\R$ and assume $\mathbb F((-\infty,z])>0$.
Then for any $p>0$, there exists $C\in (0,\infty)$, depending only on $d,\mathbb F, z,p$, such that for all $h\le z$,
\[
\E\big[\one\{(0,h)\in\EE(\PP+\delta_{(0,h)})\}\,(T((0,h);\PP)-h)_+^{p}\big]
\le C\big(1+(z-h)_+^{p}\big).
\]
Moreover, under the moment assumption $\E[H_-^{p}]<\infty$, we have
\[
\int_{-\infty}^{z}\E\big[\one\{(0,h)\in\EE(\PP+\delta_{(0,h)})\}\,(T((0,h);\PP)-h)_+^{p}\big]\,\mathbb F(\d h)<\infty.
\]
\end{corollary}
\begin{proof}
Set
$
X := \one\{x\in\EE(\PP+\delta_{x})\}\,(T(x;\PP)-h)_+ ,
$
where $x=(0,h)$. Then, for all $s>0$,
\[
\mathbb P(X>s)
=\mathbb P\big(x\in\EE(\PP+\delta_{x}),\,T(x;\PP)>h+s\big).
\]
Using the representation,
$
\mathbb E[X^p]=p\int_0^\infty s^{p-1}\mathbb P(X>s)\,\mathrm ds,
$
we split the integral at $A:=2(z-h)_+$. First,
\[
p\int_0^A s^{p-1}\mathbb P(X>s)\,\mathrm ds
\le p\int_0^A s^{p-1}\,\mathrm ds
= A^p
\le C(z-h)_+^p .
\]
Next, let $s>A \ge 0$, so that $t:=h+s>h$. By Proposition~\ref{prop:CoverTimetail}(iii),
\[
\mathbb P(X>s)
\le C_0\exp\Big(-\tfrac{1}{2^d}\Q\big(\tfrac{h+t}{2}\big)\Big)
= C_0\exp\Big(-\tfrac{1}{2^d}\Q(h+s/2)\Big).
\]
Moreover,
$
\Q(u)\ge \kappa_d \mathbb F((-\infty,z])(u-z)_+^{d/2}.
$
For $s>A$ we have $h+s/2\ge z$, hence
$
\Q(h+s/2)\ge c (s-A)^{d/2}
$
for some $c=c(d,\mathbb F,z)>0$. Consequently,
$
\mathbb P(X>s)\le C_0\exp\big(-c(s-A)^{d/2}\big).
$
Therefore,
\[
\int_A^\infty s^{p-1}\mathbb P(X>s)\,\mathrm ds
\le C\int_0^\infty (u+A)^{p-1}e^{-cu^{d/2}}\,\mathrm du
\le C'(1+A^p),
\]
where the last inequality follows since $\int_0^\infty u^q e^{-cu^{d/2}}\,\mathrm du <\infty$
for all $q\ge0$. Thus,
$
\mathbb E[X^p]\le C\big(1+(z-h)_+^p\big).
$
Finally,
\(
\int_{-\infty}^z \mathbb E[X^p]\,\mathbb F(\mathrm dh)
\le
 C\big(1+\mathbb E[(z-H)_+^p]\big).
\)
Since $(z-H)_+\le |z|+H_-$, the assumption $\mathbb E[H_-^p]<\infty$ implies finiteness of
the right-hand side.
\end{proof}
In particular, this shows that the amount of configuration needed to determine the full cell is sufficiently
integrable under the stated moment condition on the mark distribution, which is the key input needed later for
the weighted process.
We note a variant that will be used directly in the univariate CLT proof in Section~\ref{sec:univ_proof}.
\begin{lemma}
	\label{lem:a2}
	Let $\mathbb F((-\infty,z])>0$ and fix $p>0$. Then for $x = (v,h)\in \R^{d+1}$ with $h\le z$,
	$$
	\E [\one\{x\in\EE(\PP+\delta_{x})\} |C(x; \PP)|^{4+p}] \le C \bigl(1+(z-h)^{(4+p)d/2}\bigr)
	$$
	for a constant $C \in (0,\infty)$ depending only on $z,\mathbb F,d,p$.
\end{lemma} 

\begin{proof}
On the event \(x\in\EE(\PP+\delta_x)\),
\eqref{eq:cell-at-cover-time} gives
\(
|C(x;\PP+\delta_x)|
\le \kappa_d\big(T(x;\PP+\delta_x)-h\big)^{d/2}.
\)
Therefore, by Corollary~\ref{lem:cover-time-moment-integrability},
\begin{align*}
\E\Big[\one\{x\in\EE(\PP+\delta_x)\}
|C(x;\PP+\delta_x)|^{4+p}\Big]
&\le \kappa_d^{4+p}
\E\Big[\one\{x\in\EE(\PP+\delta_x)\}
\big(T(x;\PP+\delta_x)-h\big)^{(4+p)d/2}\Big]\\
&\le C\big(1+(z-h)_+^{(4+p)d/2}\big).
\end{align*}
\end{proof}

%
%
    \section{Variance lower bound}
\label{sec:varlb}

    In this section, we establish the lower bounds for the variance of the functionals $G_n(z)$ and $G^{\ms V}_n(z)$.  More precisely we show the following propositions.

    \begin{proposition}\label{prop:VarianceBound}
For any $z\in\mathbb{R}$ with $\mathbb F((-\infty,z])> 0$ and any $n\in\mathbb{N}$ we have $\Var(G_n(z))\ge C(\mathbb F,z)|W_n|$, where $C(\mathbb F,z)\in (0,\infty)$.
\end{proposition}

\begin{proposition}\label{prop:VarianceBoundVolume}
Assume that $\mathbb F$ is not a Dirac measure. Then, for any $z$ with $\mathbb F((-\infty,z])> 0$ and $n>n_0(\mathbb F,z,W)$ we have $\Var(G^{\ms V}_n(z))\ge C(\mathbb F,z, W)|W_n|$ with some $C(\mathbb F,z,W)\in (0,\infty)$.
\end{proposition}

\subsection{Proof of Proposition \ref{prop:VarianceBound}} In this proof we use the notation $G_n(z;\PP)$ in order to stress the dependence of $G_n(z)$ on $\PP$ and the fact that $G_n(z;\PP)$ is a Poisson functional. We will use the following first-chaos projection inequality for Poisson functionals \cite{LPbook}
\begin{equation}\label{eq:VarianceFockSpace}
\Var(G_n(z))\ge\int_{\R^{d+1}}(\E[D_xG_n(z;\PP)])^2\Q(\d x),
\end{equation}
where
\[
D_xG_n(z;\PP):=G_n(z;\PP+\delta_x)-G_n(z;\PP),
\]
is the add-one cost operator.

Further we note that adding a point $x$ to the process $\PP$ influences the value of $G_n(z)$ in two ways. First, if $x$ is itself an extreme point, which is contained in $W_n^z$, then the value of $G_n(z)$ increases by $1$. At the same time $x$ may destroy extremality of some already existing points, which will decrease the value of $G_n(z)$. Since $x\in \EE(\PP+\delta_x)$ is equivalent to $\PP(D(x))=0$ and $x$ destroys the extremality of a point $X_i\in \EE(\PP)$ if and only if $x\in D(X_i)$ we arrive at
\begin{align*}
 D_xG_n(z;\PP)&= \one\{x\in W_n^z\}\one \{\PP(D(x))=0\}-\sum_{X_i\in\PP\cap W_n^z}\one\{X_i\in\EE(\PP)\}\one\{x\in D(X_i)\}.
\end{align*}
Writing $x=(v,h)$ and applying Mecke's formula we get
\begin{equation}\label{eq:ExpDx}
\E[D_xG_n(z;\PP)]=\one\{x\in W_n^z\}e^{-\Q(h)}-\int_{-\infty}^ze^{-\Q(t)}\int_{W_n}\one\{\|v-w\|^2<t-h\}\d w \mathbb F(\d t),
\end{equation}
where $\Q(h)$ is defined by \eqref{eq:Qdef} and $\Q(h)<\infty$ for any $h\in\R$ due to \eqref{eq:ConditionH}.

Let now $s=\inf\{t\leq z\colon \mathbb F((t,z])=0\}$. Since $\mathbb F((-\infty,z])>0$ we have that $s$ is finite and for any $\varepsilon>0$ it holds $\mathbb F((s-\varepsilon,s])>0$. Using $\mathbb F((s,z])=0$ for $h\in (s-\varepsilon,s]$ and $v\in\R^d$ we obtain
\(
    \int_{-\infty}^ze^{-\Q(t)}\int_{W_n}\one\{\|v-w\|^2<t-h\}\d w \mathbb F(\d t)\leq  \kappa_d\int_{h}^s(t-h)^{d/2}\mathbb F(\d t)\leq  \kappa_d\varepsilon^{d/2},
\)
and, moreover, $e^{-\Q(h)}\leq e^{\Q(s)}$, since $\Q$ is increasing. Thus, choosing $\varepsilon$ sufficiently small, such that $\kappa_d\varepsilon^{d/2}\leq {1\over 2}e^{-\Q(s)}$, we note that
\eqref{eq:ExpDx} gives for $h\in (s-\varepsilon,s]$ and $v\in W_n$, that 
\(
\E[D_xG_n(z;\PP)]\ge {1\over 2}e^{-\Q(s)}.
\)
 Combining this with \eqref{eq:VarianceFockSpace} we finally obtain
for some $C(\mathbb F,z)>0$ that
\begin{align*}
    \Var(G_n(z))\ge {1\over 4}e^{-2\Q(s-\varepsilon)}\mathbb F((s-\varepsilon,s])|W_n|=C(\mathbb F,z)|W_n|.
\end{align*}

\subsection{Proof of Proposition \ref{prop:VarianceBoundVolume}}

We obtain the variance lower bound by conditioning on sigma algebra $\mathcal{G}$, which is constructed so that only the central
generators in suitable local configurations remain random.
A local lemma isolates their cells and gives a uniform positive conditional
variance; we deduce the proposition from that lemma before proving it.
For the conditional variance bound, we first reduce to a deterministic
estimate for spatial motion at a fixed height and then prove that estimate
by a radial formula and continuity.
Recall the interior erosion $W_n^{\ominus\rho}$ defined in Section~\ref{sec:mod}.
These interior windows allow the local configurations to fit inside $W_n$.

  \begin{figure}
      
  \centering

  \begin{tikzpicture}[x=0.8cm,y=0.8cm,>=Latex]

\def\d{0.4}
\def\xL{-3}
\def\xR{-3+44*\d+2*\d}
\def\yTop{49*\d*\d}
\def\tmax{7.2*\d} 

\draw[black,line width=0.45pt,dashed] (\xL,\yTop-22*\d*\d) -- (\xR,\yTop-22*\d*\d);
\node[black,anchor=west] at (\xR,\yTop-22*\d*\d+0.1) {$b$};
\draw[black,line width=0.45pt,dashed] (\xL,\yTop-23*\d*\d) -- (\xR,\yTop-23*\d*\d);
\node[black,anchor=west] at (\xR,\yTop-23*\d*\d-0.02) {$b-\varepsilon\delta^2$};

\draw[red,line width=0.45pt,dashed] (\xL,\yTop-48*\d*\d) -- (\xR,\yTop-48*\d*\d);
\node[red,anchor=west] at (\xR,\yTop-48*\d*\d+0.1) {$a$};

\draw[red,line width=0.45pt,dashed] (\xL,\yTop-49*\d*\d) -- (\xR,\yTop-49*\d*\d);
\node[red,anchor=west] at (\xR,\yTop-49*\d*\d-0.02) {$a-\delta^2$};

\draw[black,thick] (\xL,\yTop) -- (\xR,\yTop);

\begin{scope}[blue!75!black,line width=1.4pt]

  \draw (\xL,\yTop) -- (\xL+11*\d,\yTop);

  \draw plot[domain=0:\tmax,samples=80] ({\xL+13*\d-\x}, {\yTop-\x*\x});
  \draw (\xL+13*\d,\yTop) -- (\xL+33*\d,\yTop);
  \draw plot[domain=0:\tmax,samples=80] ({\xL+33*\d+\x}, {\yTop-\x*\x});

    \draw (\xL+35*\d,\yTop) -- (\xL+46*\d,\yTop);

  \foreach \x in {\xL+11*\d,\xL+13*\d,\xL+33*\d,\xL+35*\d} {
    \draw[line width=1.4pt] (\x,\yTop-0.05) -- (\x,\yTop+0.12);
  }
\end{scope}

\begin{scope}[blue!75!black, dashed, line width=1.2pt]

  \draw plot[domain=0:\tmax,samples=80] ({\xL+11*\d+\x},  {\yTop-\x*\x});
    \draw plot[domain=0:\tmax,samples=80] ({\xL+35*\d-\x}, {\yTop-\x*\x});
\end{scope}

\fill[black] (\xL+\d,\yTop) circle (1.4pt);

\fill[black] (\xL+23*\d,\yTop) circle (1.4pt);

\fill[black] (\xL+45*\d,\yTop) circle (1.4pt);

\node[anchor=west] at (\xR,\yTop+0.05) {$a+48\delta^2$};

\draw[gray!70,dashed,line width=0.35pt] (\xL+19.5*\d,\yTop) -- (\xL+19.5*\d,\yTop-50*\d*\d);
\draw[gray!70,dashed,line width=0.35pt] (\xL+26.5*\d,\yTop) -- (\xL+26.5*\d,\yTop-50*\d*\d);

\draw[blue!75!black,<->,thick] (\xL+13*\d,\yTop+0.7) -- node[above] {$10\delta$} (\xL+23*\d,\yTop+0.7);
\node at (\xL+23*\d,\yTop+0.3) {$y$};

\draw[red,thick] (\xL+22*\d,\yTop-49*\d*\d) rectangle (\xL+24*\d,\yTop-48*\d*\d);
\node[red] at (\xL+22*\d+0.4,\yTop-51*\d*\d) {$K_0(y)$};

\def\boxY{\yTop - 23*\d*\d}
\def\boxW{0.4*\d}
\def\boxH{\d*\d}

\draw[black,thick] (\xL+19.5*\d-\boxW/2,\boxY) rectangle ++(\boxW,\boxH);
\node at (\xL+19.5*\d+0.5*\boxW,\yTop - 25*\d*\d) {$K_1(y)$};

\fill[black] (\xL+19.5*\d,\yTop) circle (1.4pt);
\node[anchor=south] at (\xL+19.5*\d,\yTop+0.08) {$w_1(y)$};

\draw[black,thick] (\xL+26.5*\d-\boxW/2,\boxY) rectangle ++(\boxW,\boxH);
\node at (\xL+26.5*\d+0.5*\boxW,\yTop - 25*\d*\d) {$K_2(y)$};
\fill[black] (\xL+26.5*\d,\yTop) circle (1.4pt);
\node[anchor=south] at (\xL+26.5*\d,\yTop+0.08) {$w_2(y)$};

\draw[black,->,thick] (\xL+23*\d,0.95) -- node[below] {$3.5\delta$} (\xL+26.5*\d,0.95);
\fill[black] (\xL+23*\d,0.95) circle (1.1pt);

\end{tikzpicture}

\caption{Illustration of the construction around point $y\in (22\delta)\Z^d\cap W_n^{\ominus11\delta}$ in dimension $2$.}\label{fig:Construction3}
\end{figure}
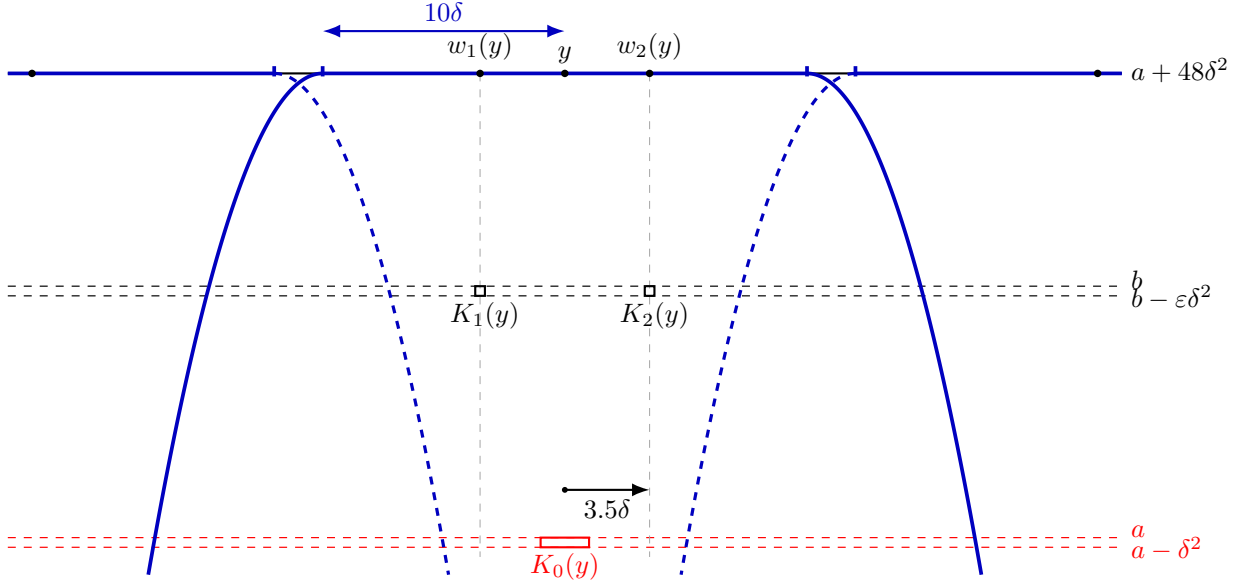
  
The local configuration is illustrated in Figure~\ref{fig:Construction3}.
Since $\mathbb F$ is not a Dirac measure and $\mathbb F((-\infty,z])>0$, choose $a<b$ with
$a\le z$ such that $\mathbb F((a-t,a])>0$ and $\mathbb F((b-t,b])>0$ for every
$t\in(0,(b-a)/2)$.
Set $\delta:=\sqrt{(b-a)/26}$ and
$\Lambda_n:=(22\delta)\Z^d\cap W_n^{\ominus11\delta}$, which is nonempty for
sufficiently large $n$.
For $y\in\Lambda_n$, define
\[
K(y):=\{(v,h)\in\R^{d+1}:h\le a+48\delta^2,
\ \|v-y\|\le10\delta+\sqrt{a+48\delta^2-h}\}.
\]
Choose distinct points $u_1,\ldots,u_k\in\partial B_{3.5}(0)$ such that
the balls $B^\circ_{4.65}(349u_i/350)$ cover $B_{8}(0)$.
Such a finite collection exists by compactness: for each
$\theta\in S^{d-1}$, the ball centered at $3.49\theta$ with
radius $4.65$ contains $8\theta$ in its interior.
Each ball also contains the origin, so a finite cover of the sphere
extends to the ball by convexity.
Put $w_i(y):=y+\delta u_i$ and
$w_i':=349w_i(y)/350+y/350$.
For a fixed $y$, we have
\begin{equation}\label{eq:Condition1}
L :=\{v\in\R^d:5\delta\le\|v-y\|\le8\delta\} \subset B_{8\delta}(y)
\subseteq\bigcup_{i=1}^kB^\circ_{4.65\delta}(w_i').
\end{equation}
The collection $(u_i)_{i=1}^k$ is fixed independently of $y$ and $n$.
For sufficiently small $q\in(0,0.01)$ and $\varepsilon\in(0,1]$, define
\(K_0(y):=B_{\delta}(y)\times[a-\delta^2,a]\subseteq K(y)\), and  for \(1\le i\le k\),
\[
\begin{aligned}
K_i(y)
:=B_{q\delta}(w_i(y))\times[b-\varepsilon\delta^2,b]
=B_{q\delta}(w_i(y))
\times[a+(26-\varepsilon)\delta^2,a+26\delta^2]
\subseteq K(y).
\end{aligned}
\]
Choose $q$ small enough that these boxes are pairwise disjoint.
The choice of $b$ ensures
$\mathbb F([b-\varepsilon\delta^2,b])>0$ for every such $\varepsilon$.
Let
\[
\mathcal{A}_y
:=
\Big\{
\PP\big(K_i(y)\big)=1 \text{ for all } i=0,\ldots,k
\Big\}
\;\cap\;
\Big\{
\PP\big(K(y)\setminus\bigcup_{i=0}^{k}K_i(y)\big)=0
\Big\}.
\]
Since the boxes $K_i(y)$, $0\le i\le k$, are pairwise disjoint, independence of Poisson counts gives
    \begin{equation}\label{eq:18.01.26}
        \P(\mathcal{A}_y)=\exp\big(-\Q(K(y))\big)\prod_{i=0}^{k}\Q(K_i(y))=\tilde c_1(z,\mathbb F)>0,
    \end{equation}
    where $\Q(K_i(y))\in (0,\infty)$ for all $0\le  i\le  k$ as follows from the choice of $a$, $b$, $\delta$, $q$ and $\varepsilon$, and 
\begin{align*}
\Q(K(y))&\le 10^d\cdot 2^{d-1}\kappa_d\delta^d+2^{d-1}\kappa_d\E[(a+48\delta^2-H)_+^{d/2}]<\infty,
\end{align*}
as follows from \eqref{eq:ConditionH}.
To construct conditioning sigma algebra $\mathcal G$, put
$U:=(\bigcup_{y\in\Lambda_n}K_0(y))^c$ and define
\[
\mathcal H:=\sigma(\PP|_U,(\PP(K_0(y)))_{y\in\Lambda_n}).
\]
For $y\ne y'\in\Lambda_n$, we have
$K(y)\cap K_0(y')=\varnothing$: at central heights, the spatial radius
of $K(y)$ is at most $17\delta$, whereas every point of $K_0(y')$
has spatial distance at least $21\delta$ from $y$.
Thus, the events $\mathcal A_y$ are $\mathcal H$-measurable.
Define $\mathcal G$ by additionally revealing the configuration inside
$K_0(y)$ for those $y$ for which $\mathcal A_y$ fails.

On $\mathcal A_y$, write $(V_i(y),H_i(y))$ for the unique point of
$\PP$ in $K_i(y)$, $0\le i\le k$, and suppress $y$ when it is fixed.
Set $X_y:=|C((V_0,H_0);\PP)|$ on $\mathcal A_y$ and $X_y:=0$ otherwise.
Let $c_2(\mathbb F,z)>0$ denote a constant to be chosen below.

\begin{figure}
    \centering
\begin{tikzpicture}[x=0.8cm,y=0.8cm,>=Latex]

\def\d{0.4}
\def\xL{-3}
\def\xR{-3+34*\d}
\def\yTop{49*\d*\d}
\def\tmax{7.2*\d} 

\def\boxY{\yTop - 17*\d*\d}
\def\boxW{0.4*\d}
\def\boxH{\d*\d}


\pgfmathsetmacro{\aZero}{\xL+16*\d+0.7*\d}
\pgfmathsetmacro{\aOne}{\xL + 13.5*\d-\boxW/2 + 0.5*\boxW}
\pgfmathsetmacro{\aTwo}{\xL + 20.5*\d-\boxW/2 + 0.5*\boxW}
\pgfmathsetmacro{\bApexZero}{\yTop-49*\d*\d+0.5*\d*\d}
\pgfmathsetmacro{\bApexOne}{\boxY + 0.5*\boxH}
\pgfmathsetmacro{\bApexTwo}{\boxY + 0.2*\boxH}

\pgfmathsetmacro{\xIntLeft}{(\aZero*\aZero-\aOne*\aOne+\bApexZero-\bApexOne)/(2*(\aZero-\aOne))}
\pgfmathsetmacro{\yIntLeft}{(\xIntLeft-\aZero)^2 + \bApexZero}

\pgfmathsetmacro{\xIntRight}{(\aZero*\aZero-\aTwo*\aTwo+\bApexZero-\bApexTwo)/(2*(\aZero-\aTwo))}
\pgfmathsetmacro{\yIntRight}{(\xIntRight-\aZero)^2 + \bApexZero}

\begin{scope}
  \fill[blue!30]
    plot[domain=\aZero-7*\d:\aZero+7*\d,samples=250,variable=\xx]
      ({\xx}, {(\xx-\aZero)*(\xx-\aZero) + \bApexZero})
    -- cycle;
    
  \draw[blue!80!black,line width=0.5pt]
    plot[domain=\aZero-7*\d:\aZero+7*\d,samples=250,variable=\xx]
      ({\xx}, {(\xx-\aZero)*(\xx-\aZero) + \bApexZero});
    \end{scope}

\begin{scope}
  \fill[orange!30]
    plot[domain=\aOne-4.2*\d:\aOne+4.2*\d,samples=250,variable=\xx]
      ({\xx}, {(\xx-\aOne)*(\xx-\aOne) + \bApexOne})
    -- cycle;

  \fill[orange!30]
    plot[domain=\aTwo
    +4.2*\d:\aTwo-4.2*\d,samples=250,variable=\xx]
      ({\xx}, {(\xx-\aTwo)*(\xx-\aTwo) + \bApexTwo})
    -- cycle;

  \draw[orange!80!black,line width=0.5pt]
    plot[domain=\aOne-4.2*\d:\aOne+4.2*\d,samples=250,variable=\xx]
      ({\xx}, {(\xx-\aOne)*(\xx-\aOne) + \bApexOne});

  \draw[orange!80!black,line width=0.5pt]
    plot[domain=\aTwo-4.2*\d:\aTwo+4.2*\d,samples=250,variable=\xx]
      ({\xx}, {(\xx-\aTwo)*(\xx-\aTwo) + \bApexTwo});
\end{scope}

\draw[gray!80,densely dashed] (\xIntLeft,-10*\d*\d) -- (\xIntLeft,\yIntLeft);
\draw[gray!80,densely dashed] (\xIntRight,-10*\d*\d) -- (\xIntRight,\yIntRight);

\fill[black] (\xIntLeft,\yIntLeft) circle (1.2pt);
\fill[black] (\xIntRight,\yIntRight) circle (1.2pt);

\draw[blue!75!black,<->,thick] (\xIntLeft,-9*\d*\d) -- node[above] {$C((V_0,H_0);\PP)$} (\xIntRight,-9*\d*\d);

\node[black] at (\aZero,\bApexZero+3.2*\d*\d) {{\tiny $(V_0,H_0)$}};
\fill[blue] (\aZero,\bApexZero) circle (1.4pt);
\node[black] at (\aOne,\bApexOne+3.2*\d*\d) {{\tiny $(V_1,H_1)$}};
\fill[orange] (\aOne,\bApexOne) circle (1.4pt);
\node[black] at (\aTwo,\bApexTwo+3.2*\d*\d) {{\tiny$(V_2,H_2)$}};
\fill[orange] (\aTwo,\bApexTwo) circle (1.4pt);

\draw[black,thick] (\xL,\yTop) -- (\xR,\yTop);

\begin{scope}[blue!75!black,line width=1.4pt]
  \draw plot[domain=0:\tmax,samples=80] ({\xL+7*\d-\x}, {\yTop-\x*\x});
  \draw (\xL+7*\d,\yTop) -- (\xL+27*\d,\yTop);
  \draw plot[domain=0:\tmax,samples=80] ({\xL+27*\d+\x}, {\yTop-\x*\x});

  \foreach \x in {\xL+7*\d,\xL+27*\d} {
    \draw[line width=1.4pt] (\x,\yTop-0.05) -- (\x,\yTop+0.12);
  }
\end{scope}

\fill[black] (\xL+17*\d,\yTop) circle (1.4pt);
\node at (\xL+17*\d,\yTop+0.3) {$y$};

\draw[red,thick] (\xL+16*\d,\yTop-49*\d*\d) rectangle (\xL+18*\d,\yTop-48*\d*\d);
\node[red] at (\xL+16*\d+0.4,\yTop-51*\d*\d) {$K_0(y)$};

\draw[black,thick] (\xL+13.5*\d-\boxW/2,\boxY) rectangle ++(\boxW,\boxH);
\node at (\xL+13.5*\d+0.5*\boxW,\yTop - 19*\d*\d) {$K_1(y)$};

\fill[black] (\xL+13.5*\d,\yTop) circle (1.4pt);
\node[anchor=south] at (\xL+13.5*\d,\yTop+0.08) {$w_1(y)$};

\draw[black,thick] (\xL+20.5*\d-\boxW/2,\boxY) rectangle ++(\boxW,\boxH);
\node at (\xL+20.5*\d+0.5*\boxW,\yTop - 19*\d*\d) {$K_2(y)$};
\fill[black] (\xL+20.5*\d,\yTop) circle (1.4pt);
\node[anchor=south] at (\xL+20.5*\d,\yTop+0.08) {$w_2(y)$};

\end{tikzpicture}

\caption{Illustration of the cell $C((V_0,H_0);\PP)$. We have $(V_0,H_0)\in \mathcal{E}(\PP)$, while $(V_1,H_1), (V_2,H_2)\not\in \mathcal{E}(\PP)$ and, hence, the corresponding cells do not count for $G_n^{\rm V}(z)$.}\label{fig:Construction4}
\end{figure}
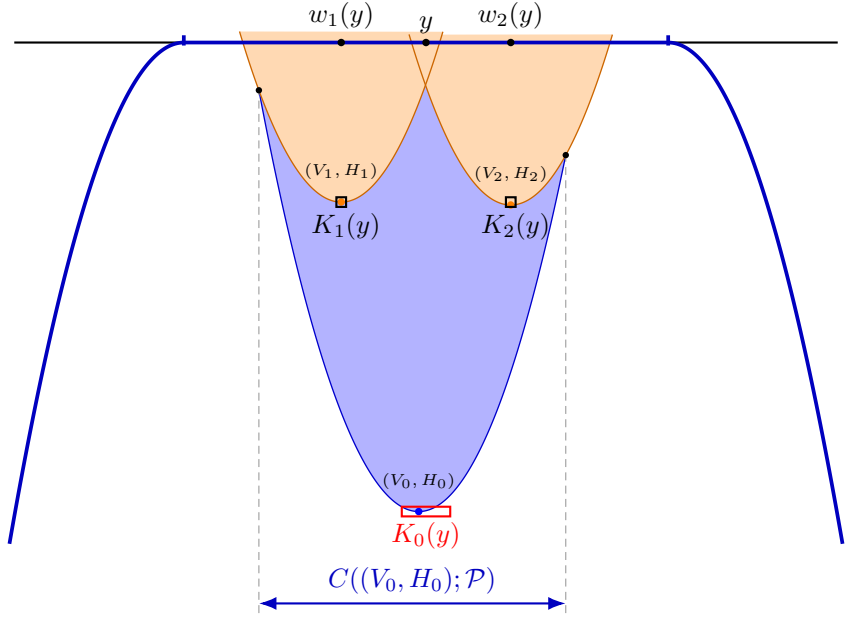

The following lemma separates the geometric isolation of the central cell
from its conditional variance.
The assertions also control the other terms in the weighted
functional when a central generator is varied.

\begin{lemma}\label{lem:local_volume_configuration}
For sufficiently small $q$ and $\varepsilon$, there is a choice of
$c_2(\mathbb F,z)$ such that the following assertions hold almost surely.
\begin{itemize}
\item[(a)] On $\mathcal A_y$, $(V_0,H_0)$ is extreme and
$(V_i,H_i)$, $1\le i\le k$, are not extreme.
\item[(b)] On $\mathcal A_y$, the central cell is contained in
$B_{8\delta}(y)$ and is determined by its own generator and the
surrounding generators $(V_i,H_i)$, $1\le i\le k$.
Its only possible neighbors are $C((V_i,H_i);\PP)$, $1\le i\le k$.
Assertions (a) and (b) remain valid when the central generator is
replaced by any point of $K_0(y)$, with all other generators fixed.
\item[(c)] For every $y\in\Lambda_n$,
\[
\one\{\mathcal A_y\}\Var(X_y\mid\mathcal G)
\ge\one\{\mathcal A_y\}c_2(\mathbb F,z).
\]
\item[(d)] For distinct $y,y'\in\Lambda_n$, the variables $X_y$ and
$X_{y'}$ are independent conditionally on $\mathcal G$.
\end{itemize}
\end{lemma}

As the proof of Lemma \ref{lem:local_volume_configuration} is technical, we defer it to Appendix \ref{apx:varlb}. We now show how to conclude the proof of Proposition 
\ref{prop:VarianceBoundVolume}.

\begin{proof}[Proof of Proposition~\ref{prop:VarianceBoundVolume}]
If the variance is infinite, the lower bound is immediate, so assume
$G_n^{\ms V}(z)$ is square integrable.
Choose $q$ and $\varepsilon$ as in
Lemma~\ref{lem:local_volume_configuration}, and put
$Z_y:=X_y\one\{\mathcal A_y\}$ and
$R:=G_n^{\ms V}(z)-\sum_{y\in\Lambda_n}Z_y$.
Thus,
\begin{equation}\label{eq:Decomposition}
G_n^{\ms V}(z)=R+\sum_{y\in\Lambda_n}Z_y.
\end{equation}

We first verify that $R$ is $\mathcal G$-measurable.
On $\mathcal A_y$, the covering in \eqref{eq:Condition1} and
$\|V_i-w_i'\|<0.02\delta$ imply that the surrounding generators
attain Laguerre distance less than $a+48\delta^2$ throughout
$B_{8\delta}(y)$.
Every generator outside $K(y)$ has Laguerre distance greater than
$a+48\delta^2$ there, by the definition of $K(y)$.
Consequently, only the central generator and its fixed surrounding
generators can win inside this ball.
By Lemma~\ref{lem:local_volume_configuration}(a)--(b), moving the
central generator within $K_0(y)$ affects no cells outside the ball,
and the surrounding generators remain nonextreme.
The central generator lies in $W_n^z$, so its entire contribution is exactly $X_y$.
Subtracting the central contributions therefore leaves a quantity
determined by $\mathcal G$.

The events $\mathcal A_y$ are $\mathcal G$-measurable, and
Lemma~\ref{lem:local_volume_configuration}(d) makes the $Z_y$
pairwise conditionally independent.
The conditional variance decomposition gives
\begin{equation}\label{eq:TotalVariance}
\Var(G_n^{\ms V}(z))
=\E[\Var(G_n^{\ms V}(z)\mid\mathcal G)]
+\Var(\E[G_n^{\ms V}(z)\mid\mathcal G])
\ge\E[\Var(G_n^{\ms V}(z)\mid\mathcal G)].
\end{equation}
Using \eqref{eq:Decomposition}, part (c), and \eqref{eq:18.01.26},
we obtain
\begin{equation}\label{eq:VarianceViaDecomposition}
\begin{aligned}
\Var(G_n^{\ms V}(z))
\ge\sum_{y\in\Lambda_n}\E[\Var(Z_y\mid\mathcal G)]
\ge c_2(\mathbb F,z)\sum_{y\in\Lambda_n}\P(\mathcal A_y)
=c_2(\mathbb F,z)\tilde c_1(z,\mathbb F)|\Lambda_n|.
\end{aligned}
\end{equation}
For sufficiently large $n$, the lattice count satisfies
$|\Lambda_n|\ge c_3(W,z,\mathbb F)|W_n|$.
Thus, the claimed bound holds with
$C(\mathbb F,z,W):=c_2(\mathbb F,z)\tilde c_1(z,\mathbb F)c_3(W,z,\mathbb F)>0$.
\end{proof}

%
%
\section{Covariance asymptotics: proof of Theorem~\ref{thm:var}}
\label{sec:var}

This section identifies the limiting covariance kernels of \(G_n(z)\) and
\(G_n^{\ms V}(z)\), which are used in the multivariate and functional limit
theorems of Sections~\ref{sec:mult}--\ref{sec:tight_V}. 
In each case, Mecke's formula decomposes the covariance into
a diagonal term and a term involving a sum over pairs of
distinct points. The asymptotics of the diagonal term follow
from stationarity.
For the second term, we bound the stationary kernel and its first spatial
moment. A common translation estimate then shows that the part crossing the
boundary of \(W_n\) has surface order growth. For each statistic, we first
derive the covariance asymptotics from the stated kernel bounds and then
prove those bounds using the translation estimate. By symmetry, it suffices to consider \(z\le z'\).

\subsection{Unweighted case}
\label{ss:cov}
We first treat the extremality indicators at levels \(z\le z'\).
The Mecke formula gives
\begin{align*}
\Cov\big(G_n(z), G_n(z')\big) =& \int_{W_n^z}\P\big(x \in \EE(\PP+\delta_{x})\big) \Q(dx) \\
	&+\E[\sum_{\substack{ i \ne j: X_i \in W_n^z, X_j \in W_n^{z'} }}\one\{\{X_i, X_j\} \su \EE(\PP)\}] -\E[G_n(z)]\E[G_n(z')] .
\end{align*}
By stationarity, the first summand equals $|W_n|\int_{-\ff}^z\P\big((0, h) \in \EE(\PP+\delta_{(0, h)})\big) \mathbb F(dh)$. Hence, it remains to bound the second. This will be the task in the remainder of this section.

We recall from Section \ref{sec:mod} that $X_i \in \EE(\PP)$ is equivalent to $\PP(D(X_i))=0$.
For the term involving two distinct points, both points are inserted into the
same configuration \(\PP+\delta_x+\delta_{x'}\), whereas the product of the
expectations uses two separate insertions. The difference consists of the
covariance of the two extremality indicators and two corrections for direct
interaction: \(x\) may destroy the extremality of \(x'\) at \(v'\), and
\(x'\) may destroy the extremality of \(x\) at \(v\). More precisely,
\begin{align*}
	&\E[\sum_{\substack{i\ne j: X_i \in W_n^z, X_j \in W_n^{z'}}}&&\one\{\{X_i, X_j\} \su \EE(\PP)\}] - \E[G_n(z)]\E[G_n(z')]\\
&\quad=\int_{W_n^z}\int_{W_n^{z'}} &&\E[\one\{\PP(D(x))=0 \}(1 - \one\{L(x, v) > L(x', v) \}) \\
&&&\times \one\{\PP(D(x'))=0 \} (1 - \one\{L(x', v') > L(x, v') \})]\Q(dx) \Q(dx')\\
&&&\quad\quad-\E[G_n(z)]\E[G_n(z')]\\
&\quad= \int_{W_n^z}\int_{W_n^{z'}}&&\Cov\big(\one\{\PP(D(x))=0 \}, \one\{\PP(D(x'))=0 \}\big)\Q(dx) \Q(dx')\\
	&\qquad- \int_{W_n^z}\int_{W_n^{z'}}&&\E\big(\one\{\PP(D(x))=\PP(D(x'))=0 \}\big) \one\{L(x', v) < L(x, v)\}\Q(dx) \Q(dx')\\
	&\qquad- \int_{W_n^z}\int_{W_n^{z'}}&&\E\big(\one\{\PP(D(x))=\PP(D(x'))=0 \}\big) \one\{L(x, v') < L(x', v')\}\Q(dx) \Q(dx').
\end{align*}
The first integral is the covariance of the two extremality indicators. The
remaining two integrals account for direct interaction between the inserted
points. Two propositions establish absolute integrability of these terms
and show that replacing the second spatial window by \(\mathbb R^d\) creates
an error of surface order. Write \(x=(v,h)\), \(x'=(v',h')\), and
\(W_n^{\complement,z'}:=(\mathbb R^d\setminus W_n)\times(-\infty,z']\),
and let \(m_0>0\) denote the order of the assumed mark moment.
The first proposition controls the regions in which one inserted generator
can destroy the extremality of the other.

\begin{proposition}[Interaction between the inserted points]\label{lem:t1}
Assume that $\E[H_-^{m_0}]<\infty$ for some
$m_0\ge(d+1)/2$. Then the following statements hold.
\medskip

(i) (\emph{Absolute integrability})
\begin{align*}
	\int_{-\ff}^z\int_{\R^{d,z'}}  \one\{L(x',0)<L((0,h),0)\}\Q(dx')\mathbb F(dh) < \ff;
\end{align*}

(ii) (\emph{Decay})
As $n\to\infty$,
\[
\int_{W_n^z}\int_{W_n^{\complement,z'}}
\one\{L(x',v)<L(x,v)\}\,\Q(dx')\Q(dx)
=O(|W_n|^{1-1/d}).
\]
\end{proposition}

The second proposition treats the covariance of the void indicators.
Their determining paraboloids overlap, and the intensity of that overlap
controls both the stationary integral and its boundary contribution.
\begin{proposition}[Covariance term]\label{lem:t2}
Assume that $\E\![H_-^{\,m_0}]<\infty$ for some $m_0\ge 3d/2$. Then the following assertions hold.
\begin{enumerate}
  \item  (\emph{Absolute integrability})
\[
\!\!\int_{-\infty}^{z}\!\!\int_{\R^{d,z'}}\!\!
|\Cov\big(\one\{\PP(D(0, h))=0\},\one\{\PP(D(x'))=0\}\big)|\,
\Q(d x')\mathbb F(dh)\;<\;\infty.
\]
\item (\emph{Decay}) As $n\to\infty$,
\[
\int_{W_n^z}\!\!\int_{W_n^{\complement,z'}}
|\Cov\big(\one\{\PP(D(x))=0\},\one\{\PP(D(x'))=0\}\big)|\,
\Q(dx')\,\Q(dx)
= O\!\Big(|W_n|^{\,1-1/d}\Big).
\]
\end{enumerate}
\end{proposition}

\begin{proof}[Proof of Theorem~\ref{thm:var}(1)]
The assumption $\E[H_-^{3d/2}]<\infty$ implies both moment
hypotheses in Propositions~\ref{lem:t1} and~\ref{lem:t2}. The covariance expansion above
consists of its diagonal term and three off-diagonal integrals. The
expectation, multiplying each direct-interaction indicator, is at most 1,
so Proposition~\ref{lem:t1}, also with the two points interchanged, proves
absolute integrability of both contributions.
Proposition~\ref{lem:t2} gives the same conclusion for the covariance of the void
indicators. Replacing the second spatial window by \(\mathbb R^d\)
therefore changes the sum by \(O(|W_n|^{1-1/d})\). Spatial stationarity
identifies the resulting expression, including the exact diagonal term,
as \(|W_n|k(z,z')\). Dividing by \(|W_n|\) proves the asserted covariance
asymptotics. For the positivity assertion, assume additionally that \(\mathbb F((-\infty,z])>0\).
The hypotheses of Proposition~\ref{prop:VarianceBound} are then satisfied,
so it gives
\(k(z,z)=\lim_{n\to\infty}|W_n|^{-1}\Var(G_n(z))\ge C(\mathbb F,z)>0\).
\end{proof}

\subsubsection{Proof of Proposition~\ref{lem:t1}}
\label{sssec:t2}
The indicator for direct interaction is increasing in the mark $h$ of the
generator whose extremality is tested. Replacing $h$ by \(z\)
therefore bounds the interaction region by a ball with a random radius.
We first prove the translation estimate that converts its first spatial
moment into a boundary bound, then prove the proposition by radial
integration. 

For \(u\in\mathbb R^d\), put
\(b_n(u):=|W_n|-|W_n\cap(W_n-u)|\). Let
\(K:\mathbb R^d\to[0,\infty)\) be measurable with
\(\int_{\mathbb R^d}(1+\|u\|)K(u)\,du<\infty\).
The following lemma converts this integrability into a boundary estimate
that applies to each covariance contribution.

\begin{lemma}[Translation bound for the observation window]
\label{lem:window-translation}
There exists \(C_W<\infty\), depending only on \(d\) and \(W\), such that
\(0\le b_n(u)\le C_Wn^{d-1}\|u\|\) for every \(n\ge1\) and
\(u\in\mathbb R^d\). Consequently,
\[
 \int_{W_n}\int_{\mathbb R^d\setminus W_n}K(v'-v)\,dv'\,dv
 =\int_{\mathbb R^d}b_n(u)K(u)\,du
 \le C_Wn^{d-1}\int_{\mathbb R^d}\|u\|K(u)\,du.
\]
\end{lemma}

\begin{proof}
For \(u=0\), the bound for \(b_n(u)\) is immediate. Otherwise, put
\(e=u/\|u\|\) and slice \(W_n\) along lines parallel to \(e\).
Convexity makes every nonempty slice an interval. If its length is
\(\ell\), translating it by \(u\) removes length
\(\min\{\ell,\|u\|\}\) from its intersection with the original interval.
Thus, \(b_n(u)\) is at most \(\|u\|\) times the
\((d-1)\)-dimensional volume of the orthogonal projection of \(W_n\)
onto \(e^\perp\). Since \(W\) is compact it is contained in a fixed ball and, hence, the volume of this projection
is bounded by \(C_Wn^{d-1}\), uniformly in \(e\). Finally, substitute
\(u=v'-v\) in the boundary integral. For fixed \(u\), the set of admissible
\(v\)'s is \(W_n\setminus(W_n-u)\), whose volume is \(b_n(u)\).
Tonelli's theorem and the bound just proved give the assertion.
\end{proof}

\begin{proof}[Proof of Proposition~\ref{lem:t1}]
Put \(R=\sqrt{(z-H)_+}\) and
\(j_z(u)=\P(\|u\|\le R)\), for \(u\in\mathbb R^d\).
For \(h\le z\), integrating the direct-interaction indicator over the
second mark, and enlarging its range to \(\mathbb R\), gives at most
\(j_z(v'-v)\). Tonelli's theorem and radial integration yield
\(
 \int_{\mathbb R^d}j_z(u)\,du=\kappa_d\E[R^d],
 \) and \(
 \int_{\mathbb R^d}\|u\|j_z(u)\,du
 =\frac{d\kappa_d}{d+1}\E[R^{d+1}].
\)
Both expectations are finite under the stated moment assumption. The first
identity proves absolute integrability, since the remaining mark integral
has mass at most one. Lemma~\ref{lem:window-translation}, applied to
\(K=j_z\), bounds the boundary contribution by
\(C_Wn^{d-1}\E[R^{d+1}]\). Since \(|W_n|=n^d|W|\), this is
\(O(|W_n|^{1-1/d})\).
\end{proof}

\subsubsection{Proof of Proposition~\ref{lem:t2}}
\label{sssec:t1}
The covariance of two Poisson void indicators is controlled by the intensity
of the intersection of their determining regions. Raising both generator
marks to a common upper level gives a spatial kernel independent of the two
marks. We compute its integral and first spatial moment, then apply the
translation estimate. For \(\ell\in\mathbb R\) and
\(u\in\mathbb R^d\), define
\(K_\ell(u):=\Q(D(0,\ell)\cap D(u,\ell))\).

Put \(\bar z=z\vee z'\), and let \(x=(v,h)\), \(x'=(v',h')\),
with \(h\le z\) and \(h'\le z'\). The Poisson void probabilities give
\begin{equation}\label{eq:extremality-overlap-bound}
\begin{aligned}
 \Cov\big(\one\{\PP(D(x))=0\},
                    \one\{\PP(D(x'))=0\}\big)
 &=e^{-\Q(D(x)\cup D(x'))}
      \big(1-e^{-\Q(D(x)\cap D(x'))}\big)
 \le K_{\bar z}(v'-v).
\end{aligned}
\end{equation}
For \(R=\sqrt{(\bar z-H)_+}\), let \(B_R\) be the ball of radius
\(R\) centered at the origin. The horizontal sections of the paraboloids
show that \(K_{\bar z}(u)=\E[|B_R\cap(B_R+u)|]\). Hence, Tonelli's
theorem gives
\begin{equation}\label{eq:overlap-kernel-moments}
\begin{aligned}
 \int_{\mathbb R^d}K_{\bar z}(u)\,du
 &=\kappa_d^2\E[R^{2d}],\\
 \int_{\mathbb R^d}\|u\|K_{\bar z}(u)\,du
 &=\E\int_{B_R}\int_{B_R}\|v-w\|\,dv\,dw
 \le2\kappa_d^2\E[R^{2d+1}].
\end{aligned}
\end{equation}
The moment assumption implies finiteness of both expectations, because
\((2d+1)/2\le3d/2\). The first identity and
\eqref{eq:extremality-overlap-bound} prove absolute integrability after
integrating the two marks. Lemma~\ref{lem:window-translation}, applied to
\(K=K_{\bar z}\), bounds the boundary integral by
\(C_Wn^{d-1}\E[R^{2d+1}]\). This proves that the boundary contribution
has surface order, together with the asserted normalized bound.

\subsection{Volume-weighted covariance}
\label{ss:covv}

We now consider the statistic weighted by cell volume. We again expand the
covariance by the Mecke formula, but the term involving two distinct points
requires control of the random cells. We first state a bound for the
stationary kernel, derive the covariance asymptotics from it, and then prove
the bound by localization and resampling. Throughout this subsection, write
\(A_x=A(x;\PP)=\one\{x\in\EE(\PP+\delta_x)\}\,|C(x;\PP)|\).
For \(B\subset\mathbb R^d\), we use the shorthand
\(B^z=B\times(-\infty,z]\). The contribution from two inserted points
is described by the kernel
\[
 \Delta^{\ms V}(x,x')
 :=\E[A(x;\PP+\delta_{x'})A(x';\PP+\delta_x)]
   -\E[A_x]\E[A_{x'}].
\]
By definition, \(G_n^{\ms V}(z)=\sum_{X_i\in W_n^z}A_{X_i}\).
The Mecke formula then gives the decomposition
\[
\begin{aligned}
 \Cov\big(G_n^{\ms V}(z),G_n^{\ms V}(z')\big)
 &=D_n^{\ms V}(z,z')+R_n^{\ms V}(z,z'),\\
 D_n^{\ms V}(z,z')&:=\int_{W_n^{z\wedge z'}}\E[A_x^2]\,\Q(dx),\\
 R_n^{\ms V}(z,z')&:=\iint_{W_n^z\times W_n^{z'}}
                  \Delta^{\ms V}(x,x')\,\Q(dx')\,\Q(dx).
\end{aligned}
\]
To estimate the term involving two distinct points, we decompose
\(
        \Delta^{\ms V}(x,x')=T_1^{\ms V}(x,x')-T_2^{\ms V}(x,x'),
\)
where
\(
        T_1^{\ms V}(x,x'):=\Cov(A_x,A_{x'})
\)
and
\[
        T_2^{\ms V}(x,x')
        :=
        \E[
        A_xA_{x'}
        -
        A(x;\PP+\delta_{x'})A(x';\PP+\delta_x)
        ].
\]
Here \(T_1^{\ms V}\) is the covariance of the scores defined by separate
insertions, while \(T_2^{\ms V}\) corrects for the interaction between the two
inserted points.
For \(x=(v,h)\) and \(t\in\mathbb R\), define
\begin{align}
    \label{eq:dvxt}
D^{\ms V}(x,t)
:=
\{
(w,s)\in\X:
s\le t,\
\|w-v\|\le \sqrt{(t-h)_+}+\sqrt{(t-s)_+}
\}.
\end{align}
We write \(T_x:=T(x;\PP)\) and
\(
        D_x^{\ms V}:=D^{\ms V}(x,T_x).
\)
On the event \(x\in \EE(\PP+\delta_x)\), the score \(A_x\) is determined
by the configuration inside \(D_x^{\ms V}\).

Assume throughout this subsection that
\(\E[H_-^{m_0}]<\infty\) for some \(m_0>11d/2\). Choose
\(\rho>d+1\) sufficiently close to \(d+1\). Every mark moment used
below has order at most \(m_0\). Constants may depend on \(d,\mathbb F,z,z'\),
\(m_0\), and the auxiliary exponents, but not on \(n,x,x'\). For \(x=(v,h)\) and
\(x'=(v',h')\), write \(r=\|v-v'\|\). For each
\(\ell\in\mathbb R\), define
\(M_\ell(h)=1+(\ell-h)_+^{m_0}\).

The following proposition bounds the stationary kernel by a product of
integrable mark factors and a spatial kernel with finite first moment.
It supplies both the absolute integrability and the boundary estimate
needed for the covariance asymptotics.
\begin{proposition}[Bounds for overlap and boundary terms]
\label{lem:weighted-lower-tail-boundary}
Assume that \(\E[H_-^{m_0}]<\infty\) for some \(m_0>11d/2\).  Then, for
fixed \(z,z'\in\mathbb R\), there exists \(C<\infty\) such that, for
\(h\le z\) and \(h'\le z'\),
\[
        |\Delta^{\ms V}(x,x')|
        \le
        C\,M_z(h)M_{z'}(h')\,(1+r)^{-\rho}.
\]
In particular,
\[
\frac1{|W_n|}
\int_{W_n^z}
\int_{(\mathbb R^d\setminus W_n)^{z'}}
|\Delta^{\ms V}(x,x')|\,\Q(dx')\,\Q(dx)
=
O(|W_n|^{-1/d}),
\]
and the stationary kernel is absolutely integrable:
\[
\int_{-\infty}^{z}
\int_{\mathbb R^d\times(-\infty,z')}
|\Delta^{\ms V}((0,h),(v',h'))|\,dv'\,\mathbb F(dh')\,\mathbb F(dh)<\infty .
\]
\end{proposition}

\begin{proof}[Proof of Theorem~\ref{thm:var}(2)]
If \(\mathbb F((-\infty,z])=0\) or \(\mathbb F((-\infty,z'])=0\), the corresponding statistic vanishes
almost surely and every term in its covariance kernel is zero. We may
therefore assume \(\mathbb F((-\infty,z]),\mathbb F((-\infty,z'])>0\). Put \(u=z\wedge z'\).
The moment assumption is exactly that of
Proposition~\ref{lem:weighted-lower-tail-boundary}. By Lyapunov's
inequality and Lemma~\ref{lem:a2} with \(p=1\),
\(\E[A_{(0,h)}^2]\le\E[A_{(0,h)}^5]^{2/5}
\le C(1+(u-h)_+^d)\) for \(h\le u\). Hence, the diagonal
mark integral is finite. Spatial stationarity gives
\(D_n^{\ms V}(z,z')=|W_n|\int_{-\infty}^{u}\E[A_{(0,h)}^2]\mathbb F(dh)\)
without a boundary error.

Proposition~\ref{lem:weighted-lower-tail-boundary} makes the stationary
integral of \(\Delta^{\ms V}\) absolutely convergent. It also shows that
replacing the second spatial window in \(R_n^{\ms V}(z,z')\) by
\(\mathbb R^d\) changes that term by \(O(|W_n|^{1-1/d})\).
Stationarity then identifies the sum of the diagonal and extended terms
as \(|W_n|k^{\ms V}(z,z')\). Thus,
\[
 |W_n|^{-1}\Cov\big(G_n^{\ms V}(z),G_n^{\ms V}(z')\big)
 =k^{\ms V}(z,z')+O(|W_n|^{-1/d}).
\]
For the positivity assertion, assume additionally that \(\mathbb F\) is not a
Dirac measure and that \(\mathbb F((-\infty,z])>0\). Then
Proposition~\ref{prop:VarianceBoundVolume} gives
\(k^{\ms V}(z,z)=\lim_{n\to\infty}|W_n|^{-1}\Var(G_n^{\ms V}(z))
\ge C(\mathbb F,z)>0\) whenever \(\mathbb F((-\infty,z])>0\).
\end{proof}

\subsubsection{Proof of Proposition~\ref{lem:weighted-lower-tail-boundary}}
We compare each score with a version determined by a fixed region and
resample the Poisson process in that region. The first auxiliary lemma
controls the truncation error and expresses the remaining covariance
through the event that resampling changes the second score. The second
lemma bounds this event and the interaction between inserted generators
by estimating the corresponding geometric overlaps. We state these inputs,
prove the proposition from them, and defer their proofs to
Appendix~\ref{apx:out}.

For \(x=(v,h)\) and \(t\ge h\), define
\begin{equation}\label{eq:deterministic-local-score}
 A_x^{[t]}
 :=
 \one\{\PP(D(x))=0\}\,|C_t(x;\PP)|.
\end{equation}
Recall that \(C_t\) is the truncated cell from
\eqref{eq:truncated-cell-corrected}.
For \(t>h\), put \(B=D^{\ms V}(x,t)\). Given a locally finite
configuration \(\mu\) and an additional point \(y\in\X\), let
\(\mathscr D_x(\mu)=D^{\ms V}(x,T(x;\mu))\) if \(A(x;\mu)>0\), and
let \(\mathscr D_x(\mu)=\varnothing\) otherwise. Let \(\PP'\) be an
independent copy of \(\PP\), and define
\(\PP_x^{[t]}=\PP|_{B^{\complement}}+\PP'|_B\) and
\(A_{x'}^{(x,t)}=A(x';\PP_x^{[t]})\). We write
\(\EE_{x'}=\{A_{x'}>0\}\) and use \(D_{x'}^{\ms V}\) on this event.
The parameters \(p,a>0\) below denote a moment order and a power in the
truncation bound, respectively. The first lemma supplies the localization
and resampling identities used in the covariance estimate.

\begin{lemma}[Localization and resampling on a fixed region]
\label{lem:weighted-deterministic-resampling}
For \(t>h\), the following assertions hold.
\begin{enumerate}
\item[\emph{(i)}]
The score \(A_x^{[t]}\) is measurable with respect to
\(\sigma(\PP|_B)\)  with
\(A_x^{[t]}=A_x\) on \(\{T_x\le t\}\).
For every \(p,a>0\), every fixed \(z\in\mathbb R\), and every \(h\le z\),
\begin{equation}\label{eq:deterministic-truncation-error}
 \E[|A_x-A_x^{[t]}|^p]
 \le
 C(t-h)^{-a}
 \big(1+(z-h)_+^{pd/2+a}\big).
\end{equation}

\item[\emph{(ii)}]
For every locally finite \(\mu\) and every
\(y\notin\mathscr D_x(\mu)\),
\(
 A(x;\mu+\delta_y)=A(x;\mu).
\)
When \(A(x;\mu)>0\), the cell and cover time are unchanged as well.

\item[\emph{(iii)}]
If \(\Q(B)<\infty\), then \(\PP_x^{[t]}\) has the same law
as \(\PP\), and
\begin{equation}\label{eq:deterministic-resampling-identity}
 \Cov(A_x^{[t]},A_{x'})
 =\E[
 (A_x^{[t]}-\E[A_x^{[t]}])
 \{A_{x'}-A_{x'}^{(x,t)}\}
 ].
\end{equation}
Furthermore,
\begin{equation}\label{eq:deterministic-resampling-change-probability}
 \P\big(A_{x'}\ne A_{x'}^{(x,t)}\big)
 \le
 2\,\E[
 \one\{\EE_{x'}\}\,
 \Q(B\cap D_{x'}^{\ms V})
 ].
\end{equation}
\end{enumerate}
\end{lemma}

The next lemma provides the two geometric estimates used in the proof. The
first controls the overlap between the fixed resampling region and the region
determining the second score. The second controls the event that one inserted
point lies in the region determining the other score. For \(r\ge1\),
put \(t_r=h+r^2/256\) and \(B_r=D^{\ms V}(x,t_r)\). The exponents
\(q,m>0\) in this lemma specify the required mark moments.

\begin{lemma}[Overlap and interaction between inserted points]
\label{lem:weighted-deterministic-overlap}
If \(\rho>d+1\), \(q>\rho+d/2\), and
\(\E[H_-^q]<\infty\), then, for \(h\le z\),
\(h'\le z'\), and \(r\ge1\),
\begin{equation}\label{eq:deterministic-overlap-final}
 \E\big[
 \one\{\EE_{x'}\}\,
 \Q(B_r\cap D_{x'}^{\ms V})
 \big]
 \le
 C\big(1+(z'-h')_+^q\big)(1+r)^{-2\rho}.
\end{equation}
Moreover, if \(m>\rho\) and \(\E[H_-^m]<\infty\), then
\begin{equation}\label{eq:point-hit-final}
 \P\big(A_x>0,\ x'\in D_x^{\ms V}\big)
 \le
 C\big(1+(z-h)_+^m\big)
       \big(1+(z'-h')_+^m\big)
       (1+r)^{-2\rho}.
\end{equation}
The analogous estimate with \(x\) and \(x'\) interchanged also holds.
\end{lemma}

\begin{proof}[Proof of Proposition~\ref{lem:weighted-lower-tail-boundary}]
With \(\rho>d+1\) as fixed above, choose \(m\) and \(q\) so that
\(\rho<m<m_0\) and \(d+\rho<q<m_0\).
These choices are possible under \(m_0>11d/2\). In particular,
\(q>\rho+d/2\), so the overlap estimate applies, while \(m>\rho\)
allows us to use the estimate for mutual interaction.
Every polynomial factor below has degree strictly smaller than
\(m_0\) and is therefore bounded by a constant times the corresponding
factor \(M_z(h)\) or \(M_{z'}(h')\).
For \(r\le1\), the desired estimate follows directly from Cauchy's inequality
and Lemma~\ref{lem:a2}, after increasing the constant.  We therefore assume \(r>1\)
and put \(t_r=h+r^2/256\).

We first treat \(T_1^{\ms V}(x,x')=\Cov(A_x,A_{x'})\). Adding and
subtracting \(A_x^{[t_r]}\) splits this covariance into
\(\Cov(A_x-A_x^{[t_r]},A_{x'})\) and
\(\Cov(A_x^{[t_r]},A_{x'})\), which we bound separately.
By \eqref{eq:deterministic-truncation-error} with \(p=2\),
\(a=\rho\), and \(t_r-h=r^2/256\), we have
\(
 \|A_x-A_x^{[t_r]}\|_2
 \le
 C(1+r)^{-\rho}
 \big(1+(z-h)_+^{(d+\rho)/2}\big).
\)
Moreover, Lyapunov's inequality and Lemma~\ref{lem:a2} with \(p=1\) give
\(
 \|A_{x'}\|_2
 \le
 C\big(1+(z'-h')_+^{d/2}\big).
\)
Consequently,
\begin{equation}\label{eq:T1-truncation-part}
\begin{aligned}
 |\Cov(A_x-A_x^{[t_r]},A_{x'})|
 &\le
 C(1+r)^{-\rho}
 \big(1+(z-h)_+^{(d+\rho)/2}\big)
 \big(1+(z'-h')_+^{d/2}\big).
\end{aligned}
\end{equation}
For the local part, first note that \(\Q(B_r)<\infty\). Indeed, 
\(
 \Q(B_r)\le C(1+r^d+\E[H_-^{d/2}])<\infty.
\)
Thus, Lemma~\ref{lem:weighted-deterministic-resampling}(iii) applies. Put
\(U_x:=A_x^{[t_r]}-\E[A_x^{[t_r]}]\) and
\(E_{x'}:=\{A_{x'}\ne A_{x'}^{(x,t_r)}\}\). The resampling identity bounds the absolute covariance by
\(\E[|U_x|(A_{x'}+A_{x'}^{(x,t_r)})\one\{E_{x'}\}]\),
since the scores are nonnegative. H\"older's inequality with exponents
\(4,4,2\) therefore yields
\(
 |\Cov(A_x^{[t_r]},A_{x'})|
 \le
 \|U_x\|_4
 (\|A_{x'}\|_4+\|A_{x'}^{(x,t_r)}\|_4)
 \P(E_{x'})^{1/2}.
\)
Here \(\|U_x\|_4\le2\|A_x\|_4\), because
\(0\le A_x^{[t_r]}\le A_x\), while
\(A_{x'}^{(x,t_r)}\) has the same distribution as \(A_{x'}\). Therefore,
by Lemma~\ref{lem:a2}, \eqref{eq:deterministic-resampling-change-probability}, and
\eqref{eq:deterministic-overlap-final},
\begin{equation}\label{eq:T1-local-part}
 |\Cov(A_x^{[t_r]},A_{x'})|
 \le
 C(1+r)^{-\rho}
 \big(1+(z-h)_+^{d/2}\big)
 \big(1+(z'-h')_+^{(d+q)/2}\big).
\end{equation}
We next treat \(T_2^{\ms V}\). Write
\(
 A_x^{x'}:=A(x;\PP+\delta_{x'}),\) and \(
 A_{x'}^x:=A(x';\PP+\delta_x),
\)
and let \(E_x=\{A_x>0,\ x'\in D_x^{\ms V}\}\) and
\(E_{x'}=\{A_{x'}>0,\ x\in D_{x'}^{\ms V}\}\).
Insertion can only shrink a cell and destroy extremality, so
\(0\le A_x^{x'}\le A_x\) and \(0\le A_{x'}^x\le A_{x'}\). By
Lemma~\ref{lem:weighted-deterministic-resampling}(ii),
\(A_x-A_x^{x'}\) vanishes outside \(E_x\), and
\(A_{x'}-A_{x'}^x\) vanishes outside \(E_{x'}\). The nonnegative difference
\(A_xA_{x'}-A_x^{x'}A_{x'}^x\) equals the sum of
\((A_x-A_x^{x'})A_{x'}\) and
\(A_x^{x'}(A_{x'}-A_{x'}^x)\). It is therefore bounded by
\(A_xA_{x'}(\one\{E_x\}+\one\{E_{x'}\})\).
Applying H\"older's inequality with exponents \(4,4,2\) to the two
events bounds \(T_2^{\ms V}\) by \(\|A_x\|_4\|A_{x'}\|_4\)
times \(\P(E_x)^{1/2}+\P(E_{x'})^{1/2}\).
Lemma~\ref{lem:a2} and the two interaction estimates in
\eqref{eq:point-hit-final} therefore imply
\(
 |T_2^{\ms V}(x,x')|
 \le C M_z(h)M_{z'}(h')(1+r)^{-\rho}.
\)
The last bound, together with \eqref{eq:T1-truncation-part} and
\eqref{eq:T1-local-part}, proves the asserted pointwise bound for
\(\Delta^{\ms V}=T_1^{\ms V}-T_2^{\ms V}\).
The mark factors are integrable under \(\E[H_-^{m_0}]<\infty\).
Moreover, \(\rho>d+1\) implies
\(\int_{\mathbb R^d}(1+\|u\|)(1+\|u\|)^{-\rho}\,du<\infty\).
Integrating the pointwise bound over the two marks and applying
Lemma~\ref{lem:window-translation} with
\(K(u)=(1+\|u\|)^{-\rho}\) therefore gives
\[
 \int_{W_n^z}\int_{(\mathbb R^d\setminus W_n)^{z'}}
 |\Delta^{\ms V}(x,x')|\,\Q(dx')\,\Q(dx)
 \le Cn^{d-1}
       \int_{\mathbb R^d}\|u\|(1+\|u\|)^{-\rho}\,du
 \le C|W_n|^{1-1/d}.
\]
Dividing by \(|W_n|\) proves the normalized boundary bound. Integrability
of \(K\), together with the same mark bounds, proves absolute integrability
of the stationary kernel.
\end{proof}

%
%
\section{Proofs of Theorems \ref{thm:univ} and \ref{thm:mult} --univariate and multivariate CLT} 
\label{sec:univ_proof}

\subsection{Proof of Theorem \ref{thm:univ} via region stabilization}
\label{sec:univ_region}
To prove Theorem \ref{thm:univ}, we interpret the Poisson functional \(\hat G_n(z)\), \(z\in\mathbb R\), as a sum
of region-stabilizing score functions in the sense of \cite{BM22}. Under appropriate assumptions, the
corresponding Poisson functional satisfies a quantitative central limit theorem.

The usual approach via exponentially decaying stabilization radii is not available in the present setting.
Indeed, because of the unboundedness of \(D(v,h)\), and since adding a new point inside \(D(v,h)\) may turn
\((v,h)\) from an extreme point into a non-extreme point, the associated stabilization radius is almost surely
infinite. We therefore work instead in a region-stabilization framework. To apply the theorem from
\cite{BM22}, we first specify the general setting and notation.

Let \(\lambda\) denote the Lebesgue measure on \(\mathbb R^{d}\), and let \(\Q \equiv \Q_{\mathbb F} = \lambda \otimes
\mathbb F\). Let \(\PP\) be a Poisson point process on \(\mathbb R^{d+1}\) with intensity measure \(\Q\). Naturally, the
possible configurations of \(\PP\) lie in the set \(\bN\) consisting of all simple and locally finite counting
measures on \(\mathbb R^{d+1}\), equipped with a suitable \(\sigma\)-algebra \(\mathcal N\).

Consider the Poisson functional of the form 
\begin{equation}\label{eq:scoresum}
	H(\PP)= \sum_{x=(v,h) \in \PP} \xi(x;\PP),
\end{equation}
represented as a sum of \textit{local contributions} $\xi: \R^{d+1} \times \bN \to \R$, the so-called measurable score functions, from each point $x \in \PP$.

For the application to Theorem~\ref{thm:univ}, the abstract score \(\xi\) will be instantiated by the concrete
score functions associated with \(\hat G_n(z)\) and \(\hat G_n^V(z)\). In the unweighted case, the relevant
stabilization region is the paraboloid \(D(x) \cup \{x\}\), while in the weighted case one must additionally control the
cover time introduced in Section~3. The assumptions below are standard in the theory of region-stabilizing
Poisson functionals: \((A0)\) and \((A1)\) encode locality, while \((A2)\) is the corresponding uniform moment
condition required for quantitative normal approximation. Thus, the main task is to verify these assumptions
for the Laguerre score functions and then invoke the quantitative CLT from \cite{BM22}.

We will write $0$ for the empty point process, while for $\vp_1, \vp_2 \in \bN$, we write $\vp_1 \le \vp_2$ when $\vp_2$ dominates $\vp_1$. Assume the following:
\begin{enumerate}
\item[(A0)] If $\xi(x; \vp_1)=\xi(x; \vp_2)$ for some $\vp_1,\vp_2 \in \mathbf{N}$
with $0\neq \vp_1\le \vp_2$, then
\begin{equation}
	\label{eq:ximon}
	\xi(x; \vp_1)=\xi(x; \vp') \quad 
	\text{for all} \, \vp'\in\mathbf{N} \; \text{ with } \; \vp_1\le \vp'\le \vp_2.
\end{equation}
	\item[(A1)] \textit{Stabilization region:} There exists a map $R:\big\{(x;\vp)\in\R^{d+1}\times \mathbf{N}:x \in\vp\big\} \to \cB$, where $\cB$ is the Borel $\sigma$-algebra on $\R^{d+1}$, such that
	\begin{enumerate}
		\item[(A1.1)] 
		\begin{equation}
			\label{eq:1}
			\{\vp \in \mathbf{N} : y \in R(x;\vp+\delta_x)\}\in \mathscr{N} \quad \forall \, x,y\in\R^{d+1},
		\end{equation}
		 and
		\begin{equation}
			\label{eq:2}
			\P\big(y \in R(x; \PP+\delta_x)\big)\, \text{ and } \,
			\P\big(\{y_1, y_2\}\su  R(x; \PP+\delta_x)\big)
		\end{equation}
		are Lebesgue measurable functions of $(x;y) \in (\R^{d+1})^2$ and
		$(x;y_1,y_2) \in (\R^{d+1})^3$, respectively.
		\item[(A1.2)]
		$$
		R(x;\vp_1) \supseteq R(x;\vp_2),
		\quad \vp_1 \le \vp_2,\; \forall x \in\vp_1,
		$$ 
		\item[(A1.3)] for all $\vp\in\mathbf{N}$, $x\in\vp$ and $y \notin R(x;\vp)$, 
        $$\vp_{R(x;\vp)} \neq 0 \implies \vp_{R(x;\vp +\delta_y)} \neq 0.
        $$
		\item[(A1.4)] for all $\vp\in\mathbf{N}$ and $x\in\vp$,
		\begin{displaymath}
			\xi\big(x;\vp\big)
			=\xi\big(x;\vp_{R(x;\vp)}\big).
		\end{displaymath}
	\end{enumerate}
\end{enumerate}
If (A1) holds with a nontrivial function
$R$, then we call the function $\xi$ a
\emph{region-stabilizing} score function with region of stabilization given by $R$.

\begin{enumerate}
	\item[(A2)] \textit{$L^{4+p}$-norm:} Assume that there exists a $p >0$ and a measurable function $M_p:\X \to (0,\infty)$
	such that, for all $\vp\in\mathbf{N}$ with $\vp(\R^{d+1}) \le 6+\lceil p/2 \rceil$,
	\begin{displaymath}
		\Big\|\xi \big(x; \PP+\delta_x+\vp\big)\Big\|_{4+p}
		\le M_p(x), \; \forall x \in \R^{d+1}.
	\end{displaymath}
\end{enumerate}

Let $r: \R^{d+1} \times \R^{d+1} \to [0,\infty]$ be a measurable function
with
\begin{equation}
	\label{eq:Rs}
	\P\big(y \in R(x; \PP +\delta_x)\big)	\le  e^{-r(x;y)}, \quad x, y \in \R^{d+1}.
\end{equation}
For
$y_1,y_2 \in \R^{d+1}$, denote
\begin{equation}
	\label{eq:g2s}
	q(y_1,y_2):=\int_{\R^{d+1}} \P\big\{\{y_1,y_2\} 
	\su R\big(x; \PP +\delta_x\big)\big\} \;\Q(dx).
\end{equation}
For $p >0$ as in (A2) and $\zeta:=\frac{p}
    {(8+2p)(4+\lceil p/2\rceil)}$, let
\begin{align}
	\label{eq:g}
	g(y) :=&\int_{\R^{d+1}} e^{-\zeta r(x; y)} \;\Q(dx), \quad h(y) = \int_{\R^{d+1}} M_p(x)^{4+p/2}e^{-\zeta r(x; y)} \;\Q(dx), \text{ and }\nonumber\\
	&S(y) := \max(M_p(y)^2,M_p(y)^4)+ \tilde h(y)\big(1+g(y)^4\big), \quad y\in\R^{d+1},
\end{align}
where $\tilde h(y) = \max (h(y)^{2/(4+p/2)}, h(y)^{4/(4+p/2)})$.
For
$\alpha>0$, let
\begin{equation}
	\label{eq:fa}
	f_\alpha(y):=f_\alpha^{(1)}(y)+f_\alpha^{(2)}(y)+f_\alpha^{(3)}(y),
	\quad y \in\R^{d+1},
\end{equation}
where 
\begin{align}
	\label{eq:fal}
	f_\alpha^{(1)}(y)&:=\int_{\R^{d+1}} S(x) e^{- \alpha r(x;y)}
	\;\Q(dx), \notag \\
	f_\alpha^{(2)} (y)&:=\int_{\R^{d+1}} S(x) e^{- \alpha r(y,x)} 
	\;\Q(dx), \nonumber\\
	f_\alpha^{(3)}(y)&:=\int_{\R^{d+1}} S(x) q(x;y)^\alpha \;\Q(dx).
\end{align}
Finally, define the function
\begin{equation}
	\label{eq:p}
	\kappa(x):= \P\big(\xi(x; \PP+\delta_x) \ne 0\big),\quad
    x\in\R^{d+1}. 
\end{equation}

For an integrable function $f : \R^{d+1} \to \R$, denote
\(\Q f:= \Q(f) :=\int_{\R^{d+1}} f(x)\, \Q(dx).\) 
We will use the following quantitative bound for Gaussian approximation from \cite{BM22}.

\begin{theorem}[\cite{BM22}, Theorem 2.1]\label{thm:KolBd}
	Assume that $\xi$ satisfies conditions (A0) -- (A2) and
	let $H$ be as in \eqref{eq:scoresum}. Then, for $p$ as in (A2) and
	$\tau:=p /(32+4 p)$, there exists a constant $C \in (0,\infty)$ depending only on $p$ such that
    $$
    \Var(H) \le C \; \Q ((\kappa+g)^{2\tau}S).
    $$
    Moreover,
	\begin{align*}
		d_{W}\big(\frac{H-\E H}{\sqrt{\operatorname{Var} H}},  N\big) 
		&\le C \Bigg[\frac{\sqrt{\Q f_\tau^2}}{\operatorname{Var} H}
		+\frac{\Q ((\kappa+g)^{2\tau}S)}{(\operatorname{Var} H)^{3/2}}\Bigg],
	\end{align*}
	and
	\begin{multline*}
		d_{K}\big(\frac{H-\E H}{\sqrt{\operatorname{Var} H}},
		N\big) 
		\le C \Bigg[\frac{\sqrt{\Q f_\tau^2}
			+ \sqrt{\Q f_{2\tau}}}{\operatorname{Var} H}
		+\frac{\sqrt{\Q ((\kappa+g)^{2\tau}S)}}{\operatorname{Var} H}
		\\
		+\frac{\Q ((\kappa+g)^{2\tau} S)}{(\operatorname{Var} H)^{3/2}} +\frac{(\Q ((\kappa+g)^{2\tau} S))^{5/4}
			+ (\Q ((\kappa+g)^{2\tau} S))^{3/2}}{(\operatorname{Var} H)^2}\Bigg],
	\end{multline*}
	where $N$ is a standard normal random variable and
	$C \in (0,\infty)$ is a constant depending only on $p$.
\end{theorem} 

{In \cite[Theorem 2.1]{BM22}, the parameter $p$ in (A2) is taken to be in the interval $(0, 1]$. However, the arguments
for its proof remain valid for any choice of $p > 0$, as stated in Theorem \ref{thm:KolBd} above. The only
changes occur in the proof of \cite[Lemma 5.5]{BM22}: the expansion of
the $(4+p/2)$-moment now contains terms involving up to
$4+\lceil p/2\rceil$ distinct points, which leads to the updated choice of
\(
    \zeta
\) in the definition preceding \eqref{eq:g}; see Lemma \ref{lem:BM55ext} in Appendix \ref{apx:univ} for further details.
In this paper, especially in the proof of Lemma \ref{lem:fmombdV} below in the volume weighted case, 
we do make use of larger values of $p$ to minimize our moment assumption in Theorem \ref{thm:univ}(2).}

For Theorem \ref{thm:mult}, we use the multivariate quantitative CLT in \cite[Theorem 1]{TT24} along with our region of stabilization and associated bounds.

\smallskip

\noindent \underline{\em Unweighted case.} We will use Theorem \ref{thm:KolBd} to prove the result. Throughout this section, we fix some $z \in \R$. Recall that throughout this paper we write $x_i = (v_i,h_i) \in \R^{d+1}$.
Note that for $z \in \R$, the function $G_n(z)$ can be expressed as a sum of score functions as in \eqref{eq:scoresum} with 
\begin{equation}\label{eq:xiz}
\xi_n^{(z)} (x; \vp): = \one \big\{x \in W_n^z \cap  \EE(\vp)\big\}, \qquad x \in \R^{d+1},
\end{equation}
where  $\vp \in \mathbf{N}$ with $x \in \vp$, while we define $\xi_n^{(z)}(x; 0)=0$. Note that $\xi_n^{(z)}$ trivially satisfies \eqref{eq:ximon}. For $\vp \in \bN$ and $x \in \vp$, we can define the region of stabilization for $\xi_n^{(z)}$ as
\begin{equation}\label{eq:RoS}
	R_n^{(z)}(x; \vp): = \begin{cases}
		D(x) \cup \{x\},\,& v \in W_n, h \le z,\, D(x) \cap \vp=\emptyset\\
		\emptyset,\, & \text{otherwise},
	\end{cases}
\end{equation}
where  we recall that 
$D(v,h):=\big\{(v',h') \in \R^{d+1}: h-h' > \|v - v'\|^2\big\}.
$
 This necessitates the use of Theorem \ref{thm:KolBd}. It is straightforward to see that $R_n^{(z)}$ satisfies conditions (A1.1) - (A1.4) in Theorem \ref{thm:KolBd}. Also, condition (A2) is satisfied for any $p>0$ and $M_p \equiv 1$. 
 To prove the assertion, we will now derive suitable bounds on the terms appearing in Theorem \ref{thm:KolBd}. More precisely, in Section \ref{ssec:prelim_bounds}, we will derive suitable bounds on the functions $r_n$, $g^{(z)}$, $f_\alpha^{(i)}$, $i=1,2,3$ appearing in Theorem \ref{thm:KolBd} for appropriate choices of $\alpha \in (0,1)$ and $\zeta>0$. The assertion then follows using Theorem \ref{thm:KolBd} and the variance lower bound in Proposition \ref{prop:VarianceBound}. 

%
%
\subsubsection{Preliminary bounds}
\label{ssec:prelim_bounds}
We begin by recalling that
$
\Q(h_1) = {\Q}(D(x_1)) = \kappa_d \int_{-\infty}^{h_1} (h_1-h)^{d/2} \, F(dh),
$
and $\Q(h_1) \le \Q(h_2)$ for $h_1 \le h_2$.
Further, set \(r^{(z)}(x,x)=0\) for $x \in \R^{d+1}$. For $x_1 \neq x_2$, we have
\begin{align*}
	&\P(x_2 \in R_n^{(z)}(x_1, \PP + \delta_{x_1})) 
	\le \one\big\{x_1 \in W_n^z, x_2 \in D(x_1)\big\} \, \P\big( \PP\cap D(x_1)=\es \big)\\
	&\le \one\big\{x_1 \in W_n^z, x_2 \in D(x_1)\big\} \exp\big(-{\Q}(h_1)\big).
\end{align*}
Thus, for $x_1 \neq x_2$, we can take
\begin{equation}\label{eq:rz}
r^{(z)}(x_1,x_2) \equiv r_n(x_1,x_2) = \begin{cases}
		{\Q}(h_1),\,& v_1 \in W_n, h_1 \le z, x_2 \in D(x_1)\\
		\infty,\, & \text{otherwise}.
	\end{cases}
\end{equation}

The following result provides the necessary moment bounds to obtain the asserted bound in Theorem \ref{thm:univ}(1) via Theorem \ref{thm:KolBd}.

\begin{lemma}\label{lem:fmombd} For $\alpha \in (0,1)$, $f_\alpha^{(j)}$, $j=1,2,3$ as in \eqref{eq:fal}, and $\kappa, g, S$ as in \eqref{eq:p} and \eqref{eq:g} corresponding to the score in \eqref{eq:xiz},
$$
\mathbb{Q}(f_{\alpha}^{(j)} )^i \le C|W_n|, \qquad i=1,2 \text{ and } j=1,2,3,
$$
and
\(
\mathbb{Q}(\kappa^\alpha S), \mathbb{Q}(g^\alpha S) \le C |W_n|
\)
for some constant $C \in (0,\infty)$ independent of $n$.
\end{lemma}
\begin{proof}
    See Appendix \ref{apx:univprf}.
\end{proof}

\begin{proof}[Proof of Theorem \ref{thm:univ}(1)]
    The result follows immediately from an application of Theorem \ref{thm:KolBd} upon using the bounds obtained in Lemma \ref{lem:fmombd} and the Variance lower bound in Proposition \ref{prop:VarianceBound}.
\end{proof}

%
%
\noindent \underline{\em Volume weighted case.}
\label{subsec:univ_proof}
As in Section \ref{sec:univ_proof}, we fix some $z \in \R$. The function $G_n^{ms V}(z)$ can be expressed as a sum of score functions as in \eqref{eq:scoresum} with 
\begin{equation}\label{eq:xizV}
	\xi_n^{(z, \ms V)} (x; \vp): = \one \big\{x \in W_n^z \cap \EE(\vp )\big\} |C(x;\vp)|, \qquad \vp \in \mathbf{N}, x \in \vp,
\end{equation}
while we define $\xi_n^{(z, \ms V)}(x; 0)=0$. As before, $\xi_n^{(z, \ms V)}$ trivially satisfies \eqref{eq:ximon}. For the volume weighted sum, we need to use an updated region of stabilization compared to the one defined at \eqref{eq:RoS}.  For $\vp \in \bN$ and $x=(v,h) \in \vp$,  letting $T=T(x;\vp)$ be the cover time defined at \eqref{eq:covertime}, define the region of stabilization for $\xi_n^{(z, \ms V)}$ as
\begin{equation}
	R_n^{(z, \ms V)}(x; \vp): = \begin{cases}
		D^{\ms V}(x;  \vp)\,& \text{if} \quad T(x;  \vp)>h, x \in \EE(\vp)\cap W_n^z,\\
		\emptyset,\, & \text{otherwise},
	\end{cases}
\end{equation}
where for $(v,h)$ with $T>h$, using the notation \eqref{eq:dvxt}, we set
$$D^{\ms V}(x;  \vp):= D^{\ms V}(x;  T(x;\vp)).$$
Recall that, to ease notation, we write $T(x, \mathcal{P})$ instead of  $T(x, \mathcal{P}+ \delta_{x})$.
It is straightforward to see that $R_n$ satisfies conditions (A1.1) - (A1.4) in Theorem \ref{thm:KolBd}. To verify condition (A2), we can invoke Lemma \ref{lem:a2}, confirming (A2) for any $p>0$.

Thus, one can take $M_p(x) \asymp 1+ h_-^{d/2}$ for any $p>0$. To prove the assertion for Theorem \ref{thm:univ}(2), we will now bound the terms in Theorem \ref{thm:KolBd}. The assertion will then follow from Theorem \ref{thm:KolBd} and the variance lower bound in Proposition \ref{prop:VarianceBoundVolume}.

%
%
\subsubsection{Preliminary bounds}Write $x_1,x_2,x_3 \in \R^{d+1}$ as $x_1 = (v_1,h_1), x_2 = (v_2,h_2), x_3=(v_3,h_3)$. We first bound $\P(x_1 \in R_n^{(z,\ms V)}(x_2, \mathcal{P}))$. Note that, in this case we must have $x_2 \in \EE(\mathcal{P}+ \delta_{x_2})$ and $ \emptyset \neq R_n^{(z,\ms V)}(x_2, \mathcal{P}+ \delta_{x_2})) \ni x_1$, which yields
$$
\|v_1-v_2\|\leq 2(\sqrt{T(x_2, \mathcal{P})-h_1}\vee \sqrt{T(x_2, \mathcal{P})-h_2})=2\sqrt{T(x_2, \mathcal{P})-(h_1 \wedge h_2)}.
$$ 
This in turn implies that 
$
T(x_2, \mathcal{P}) \ge \|v_1-v_2\|^2/4+(h_1 \wedge h_2).
$
Since we trivially also have $T(x_2, \mathcal{P}) > h_1 \vee h_2$ in this case, we obtain for $x_1,x_2 \in \R^{d+1}$ that
\begin{align}\label{eq:tailbd}
	&\P(x_2 \in R_n^{(z,\ms V)}(x_1, \mathcal{P}))\nonumber\\
	&\le \one(x_1 \in W_n^z)\P\big(x_1 \in \EE(\mathcal{P}), T(x_1, \mathcal{P})>h_1 \vee h_2, 
	T(x_1, \mathcal{P})>\frac{\|v_1 - v_2\|^2}{4}+ {h_1\wedge h_2}\big)\nonumber\\
	&\le \one(x_1 \in W_n^z)\P\big(x_1 \in \EE(\mathcal{P}), T(x_1, \mathcal{P})>h_1 \vee h_2, 
	T(x_1, \mathcal{P}) >\frac{\|v_1 - v_2\|^2}{8}+ \frac{h_1}{2}+\frac{h_2}{2}\big)\nonumber\\
	&\le C \one(x_1 \in W_n^z) \exp \big(-\frac{1}{2^d}\Q\big(\frac{h_1 + h_1\vee h_2 \vee \big(\frac{\|v_1 - v_2\|^2}{8}+ \frac{h_1}{2}+\frac{h_2}{2}\big)}{2}\big)\big)\nonumber\\
	& \le C \one(x_1 \in W_n^z) \exp \big(-\frac{1}{2^d}\Q\big(\frac{\|v_1 - v_2\|^2}{16}+ \frac{3h_1}{4}+\frac{h_2}{4}\big)\big) = : e^{-r^{(z, \ms V)}(x_1;x_2)},
\end{align}
where the penultimate step is due to Proposition \ref{prop:CoverTimetail} and where we write $h_{1, -}$ instead of $(h_{1})_-$.
\begin{lemma}\label{lem:Ia}
    Let $\E[H_-^{md}]<\infty$ for some $m>0$. Then for any $\alpha, a,b,c>0, t \in[0, (m-1/2)d]$ and $x_2=(v_2,h_2)$,
    $$
    I_\alpha(x_2) := \int_{\R^{d+1}} (1+h_{1,-}^t) \exp\big(-\alpha \Q\big(a\|v_1 - v_2\|^2+ bh_1+ch_2\big)\big) \mathbb{Q}(dx_1) \in O(h_{2,-}^{d/2}).
    $$
\end{lemma}
\begin{proof}
    See Appendix \ref{apx:univprfV}.
\end{proof}

Let $g^{(z, \ms V)}, h^{(z, \ms V)}, S^{(z, \ms V)}, \kappa^{(z, \ms V)}$ and $f_\alpha^{(i)}, i \in \{1,2,3\}$ be defined as in \eqref{eq:g}, \eqref{eq:p} and \eqref{eq:fa}, respectively. In particular, noting 
\begin{align*}
	g^{(z, \ms V)}(x_2) :=\int_{\R^{d+1}} e^{-\zeta r^{(z,\ms V)}(x_1;x_2)} \;\Q(dx_1) &\le \int_{W_n^z} \exp\big(-\frac{\zeta}{2^d} \Q\big(\frac{\|v_1 - v_2\|^2}{16}+ \frac{3h_1}{4}+\frac{h_2}{4}\big)\big) \mathbb{Q}(dx_1),
\end{align*}
we have from Lemma \ref{lem:Ia} with $t=0, \alpha=\zeta/2^d, a=1/16, b=3/4, c=1/4$ that $g^{(z, \ms V)}(x_2)   \in O(h_{2,-}^{d/2})$. Similarly, since
\begin{align*}
	h^{(z, \ms V)}(x_2) &\le \int_{W_n^z} (1+h_{1,-}^{(4+p/2)d/2})\exp\big(-\frac{\zeta}{2^d} \Q\big(\frac{\|v_1 - v_2\|^2}{16}+ \frac{3h_1}{4}+\frac{h_2}{4}\big)\big) \mathbb{Q}(dx_1),
\end{align*}
Lemma \ref{lem:Ia} yields $h^{(z, \ms V)}(x_2)   \in O(h_{2,-}^{d/2})$ for any $p \le 4m-10$ given $\E H_-^{md}<\infty$. The following lemma records the moment bounds necessary to apply Theorem \ref{thm:KolBd}. Recall that by our assumption, we have
$\E H_-^{m d}<\infty$ for some $m>11/2$. For the remainder of this section, we fix \(p=12\), which satisfies \(12=p<4m-10\) for $m>11/2$; consequently, we have
\(
    \tau=p/(32+4p)=3/20.
\)

\begin{lemma}\label{lem:fmombdV} For $\alpha \in (0,1)$,  we have
\(
\mathbb{Q}(f_{\alpha}^{(j)} )^i \le C|W_n|, \; i,j=1,2,
\)
while
\begin{equation}\label{eq:f3bds}
\mathbb{Q}(f_{2\tau}^{(3)}), \Q\big((f_\tau^{(3)})^2\big) \le C|W_n|,
\end{equation}
for some constant $C \in (0,\infty)$ independent of $n$.
Additionally,
$$
\mathbb{Q}((\kappa^{(z,\ms V)})^{2\tau} S^{(z,\ms V)}), \mathbb{Q}((g^{(z,\ms V)})^{2\tau} S^{(z,\ms V)}) \le C' |W_n|
$$
for some constant $C' \in (0,\infty)$ independent of $n$.
\end{lemma}
\begin{proof}
    See Appendix \ref{apx:univprfV}.
\end{proof}

\begin{proof}[Proof of Theorem \ref{thm:univ}(2)]
    An application of Theorem \ref{thm:KolBd} yields the result upon using Lemma \ref{lem:fmombdV} and Proposition \ref{prop:VarianceBoundVolume}.
\end{proof}

%
%

\subsection{Proof of Theorem \ref{thm:mult} -- multivariate CLT}
\label{sec:mult}
In this section, we prove the multivariate normal approximation in Theorem \ref{thm:mult} for the vector \((\hat G_{n,*}(z_i))_{i \in [m]}\) in \(d_3\)-distance. The proof combines the multivariate Malliavin--Stein bound from \cite[Theorem 1]{TT24} with the moment and stabilization estimates established in Section \ref{sec:univ_proof}.

Fix \(z_1,\ldots,z_m \in [a,b]\). For each \(i \in [m]\), let \(\xi_{n,i}\) denote either the score \(\xi_n^{(z_i)}\) from \eqref{eq:xiz} or the score \(\xi_n^{(z_i,\ms V)}\) from \eqref{eq:xizV}. Likewise, let \(r^{(z_i)}, q^{(z_i)}, S^{(z_i)}\) denote the corresponding stabilization quantities introduced in \eqref{eq:RoS}, \eqref{eq:rz} and \eqref{eq:fa}, and define \(f_{\alpha,i}\) as in \eqref{eq:fa}, but with \(r,q,S\) replaced by \(r^{(z_i)}, q^{(z_i)}, S^{(z_i)}\). We apply \cite[Theorem 1]{TT24} with \((\mathbb X,\lambda)=(\R^{d+1},\Q)\) and with
\[
F = (F_1,\ldots,F_m), \qquad F_i := \hat G_{n,*}(z_i), \qquad i \in [m].
\]
By construction, \(\E F_i = 0\) for all \(i \in [m]\).

For completeness, recall that for a Poisson functional \(F(\mathcal P)\) and \(x,y \in \R^{d+1}\), the add-one and iterated add-one operators are given by
\[
D_x F(\mathcal P) := F(\mathcal P+\delta_x)-F(\mathcal P), \qquad D^{(2)}_{x,y}F(\mathcal P) := D_x(D_yF(\mathcal P)).
\]
We say that $F \in \operatorname{dom} D$ if $\E [F^2] <\infty$ and $\int_{\R^{d+1}} \E (D_x F)^2 \Q(d x)<\infty$.
\begin{theorem}[\cite{TT24}, Theorem 1]\label{thm:multclt}
Let $\Q = \lambda \otimes \mathbb{F}$ be a \(\sigma\)-finite measure on $\R^{d+1}$ and \(\mathcal P\) be a Poisson process on $\R^{d+1}$ with intensity measure $\Q$ as in Theorem \ref{thm:mult}. 
For $m \ge 1$, let $F=(F_1,\ldots,F_m)$ be an $\R^m$-valued random vector such that \(F_i \in \operatorname{dom} D\) and \(\E F_i=0\) for all \(i \in [m]\). Let \(\Sigma=(\sigma_{ij})_{i,j \in [m]}\) be a symmetric positive semidefinite matrix and let \(X \sim \mathcal N(0,\Sigma)\). Then for all \(q_1,q_2 \in [1,2]\),
\[
d_3(F,X) \le \frac{1}{2}\big(\zeta_1+\zeta_2^{(q_1)}+\zeta_3^{(q_1)}\big)+\zeta_4^{(q_2)},
\]
where
\begin{align}
\zeta_1 &:= \sum_{i,j=1}^m \big|\sigma_{ij}-\Cov(F_i,F_j)\big|, \label{eq:zetas_mult_1}\\
\zeta_2^{(q_1)}
&:= 2^{2/q_1-1}\sum_{i,j=1}^m
\Bigg(
\int_{\R^{d+1}}
\Bigg(
\int_{\R^{d+1}}
\E\big[|D^{(2)}_{x,y}F_i|^{2q_1}\big]^{1/(2q_1)}
\E\big[|D^{(2)}_{x,y}F_j|^{2q_1}\big]^{1/(2q_1)}
\,\Q(\md x)
\Bigg)^{q_1}
\Q(\md y)
\Bigg)^{1/q_1}, \label{eq:zetas_mult_2}\\
\zeta_3^{(q_1)}
&:= 2^{2/q_1}\sum_{i,j=1}^m
\Bigg(
\int_{\R^{d+1}}
\Bigg(
\int_{\R^{d+1}}
\E\big[|D_xF_i|^{2q_1}\big]^{1/(2q_1)}
\E\big[|D^{(2)}_{x,y}F_j|^{2q_1}\big]^{1/(2q_1)}
\,\Q(\md x)
\Bigg)^{q_1}
\Q(\md y)
\Bigg)^{1/q_1}, \label{eq:zetas_mult_3}\\
\zeta_4^{(q_2)}
&:= m^{q_2-1}\sum_{i,j=1}^m
\int_{\R^{d+1}}
\E\big[|D_xF_i|^{q_2+1}\big]^{1/(q_2+1)}
\E\big[|D_xF_j|^{q_2+1}\big]^{1-1/(q_2+1)}
\,\Q(\md x). \label{eq:zetas_mult_4}
\end{align}
\end{theorem}

We take \(q_1=q_2=2\).  The following result provides estimates on the above terms for both our models.

\begin{lemma}\label{lem:zetabds}
    For both the unweighted and volume-weighted models, the terms $\zeta_1, \zeta_2^{(2)}, \zeta_3^{(2)}$ and $\zeta_4^{(2)}$ satisfy
    $$
    \zeta_i^{(2)} = O(|W_n|^{-1/2}) \quad i=2,3,4 \quad \text{and} \quad \zeta_1 = O(|W_n|^{-1/d}),
    $$
    where the constants are independent of $n$.
\end{lemma}
\begin{proof}
    See Appendix \ref{apx:mult}.
\end{proof}

\begin{proof}[Proof of Theorem \ref{thm:mult}]
Combining the bounds on $\zeta_1, \zeta_2^{(2)}, \zeta_3^{(2)}, \zeta_4^{(2)}$ obtained in Lemma \ref{lem:zetabds} immediately yields the assertions (1)--(3).
\end{proof}

\def\diam{\ms{diam}}
\def\De{\Delta}
\section{Tightness: proof of Theorem~\ref{thm:func}}
\label{sec:tight}

\subsection{Unweighted case}
\label{sec:func_proof}

In this section, we prove Theorem~\ref{thm:func}(1). Let
\(J=[a,b]\subset\mathbb R\) be the compact interval from the theorem. If
\(a=b\), the assertion reduces to the one-dimensional convergence from
Theorem~\ref{thm:mult}. Hence, assume \(a<b\). 
For \(u\in[0,1]\), define
\(
 \phi(u):=a+(b-a)u,
\)
\(
 Y_n(u):=\hat G_{n,*}(\phi(u)),
 \) and \(
 Y(u):=Z(\phi(u)).
\)
Composition with the increasing homeomorphism \(\phi\) is a homeomorphism
between \(D(J)\) and \(D([0,1])\), both equipped with the Skorokhod
\(J_1\)-topology. The affine reparametrization is used only when applying
the criterion for tightness in the path space. All geometric, localization, variance, and
cumulant estimates below are formulated for the original mark interval
\(J\).

For \(s<t\) in \(J\), write
\(
 I:=(s,t],\) \( |I|:=t-s,\) and \(
 \mathbb F(I)=\mathbb F((s,t]),
\)
and define
\[
 G_n(I):=G_n(t)-G_n(s)
 =
 \sum_{X_i\in\PP|_{W_n\times(s,t]}}
 \one\{X_i\in\EE(\PP)\}.
\]
Furthermore,
\(
 \tilde G_n(I)
 :=
 G_n(I)-\E[G_n(I)]
 =
 \tilde G_n(t)-\tilde G_n(s),
 \) and \(
 \tilde G_n(z):=G_n(z)-\E[G_n(z)].
\)
Since \(\mathbb F\) has a bounded density on \(J\), we have
\(
 \mathbb F(I)\le \|f\|_\infty |I|.
\)
In particular, \(\mathbb F(\{t\})=0\) for every fixed \(t\in J\), so the convention
\(I=(s,t]\) is consistent with the càdlàg version of the process. All
constants below may depend on \(J\) and \(\|f\|_\infty\), but are uniform in
\(n\) and \(I\subseteq J\).
Put
\(
 a_n:=|W_n|^{-2/3}.
\)
The estimate needed for intervals above the cutoff is
\begin{equation}
\label{eq:chent_gn}
\E\big[\tilde G_n(I)^4\big]
 \le C|W_n|^2|I|^{5/4},
 \qquad I\subseteq J,\quad |I|\ge a_n.
\end{equation}
Writing \(c_4(\cdot)\) for the fourth cumulant, we use
\(
 \E\big[\tilde G_n(I)^4\big]
 =3\Var(G_n(I))^2+c_4(G_n(I)).
\)
Thus, the proof reduces to bounds for the variance and cumulant,
followed by Davydov's reduction for intervals shorter than \(a_n\).

We record how \eqref{eq:chent_gn} matches Billingsley's criterion.
Theorem~13.5 in \cite{billingsley} applies to \(Y_n\) in \(D([0,1])\) with
the Skorokhod \(J_1\)-topology. Besides finite-dimensional convergence at
continuity points of the limit and the endpoint condition
\(
 Y(1)-Y(1-\delta)\xrightarrow[\delta\downarrow0]{\mathcal D}0,
\)
it requires a nondecreasing continuous control function \(R\) and constants
\(\beta_B\ge0\), \(\alpha_B>1/2\) such that, for
\(0\le r\le s\le t\le1\) and \(\lambda>0\),
\[
 \P\!\left(
 |Y_n(s)-Y_n(r)|\wedge|Y_n(t)-Y_n(s)|\ge\lambda
 \right)
 \le
 \lambda^{-4\beta_B}\big(R(t)-R(r)\big)^{2\alpha_B}.
\]
For adjacent original mark intervals, Cauchy's inequality and
\eqref{eq:chent_gn} give
\[
\begin{aligned}
\E\Big[
 |\hat G_{n,*}(s)-\hat G_{n,*}(r)|^2
 |\hat G_{n,*}(t)-\hat G_{n,*}(s)|^2
 \Big]\le
 C(s-r)^{5/8}(t-s)^{5/8}
 \le C(t-r)^{5/4}.
\end{aligned}
\]
After the affine reparametrization, Markov's inequality gives Billingsley's
bound with
\(
 \beta_B=1,\) \( \alpha_B=\frac58\),
and  \(R(u)=C_0u\)
for a sufficiently large constant \(C_0\). Since \eqref{eq:chent_gn} is
proved only above the cutoff \(a_n\), shorter increments are handled by
Davydov's theorem.

\begin{proposition}[Davydov reduction]
\label{pr:dav}
It is enough to establish \eqref{eq:chent_gn} for intervals
\(I=(s,t]\subseteq J\) with \(|I|\ge a_n\), together with
\(
	 \sup_{\substack{I\subseteq J\\ |I|\le a_n}}
	{\E[G_n(I)]} \in o(\sqrt{|W_n|})
\).
\end{proposition}
The proof is given in Section~\ref{ss:dav}, where the hypotheses of
\cite[Theorem~2]{davydov} are checked explicitly.

\begin{proposition}[Variance bound for \(\tilde G_n\)]
\label{pr:var}
Assume \(\E[H_-^{d(1+2\delta)}]<\infty\) for some
\(\delta\ge1/2\), and assume that \(\mathbb F\) admits a bounded density on \(J\).
Then
\(
 \sup_{n\ge1}\sup_{I\subseteq J}
 |W_n|^{-1}|I|^{-1}\Var(\tilde G_n(I))<\infty.
\)
\end{proposition}
The proof is given in Section~\ref{ss:var}.

\begin{proposition}[Cumulant bound]
\label{pr:cum}
Assume \(\E[H_-^{3d}]<\infty\). Then
\(
 \sup_{n\ge1}\sup_{I\subseteq J}
 |W_n|^{-1}|c_4(G_n(I))|<\infty.
\)
\end{proposition}
The proof is given in Section~\ref{ss:cum}.

\begin{proof}[Proof of \eqref{eq:chent_gn}]
Let \(I\subseteq J\) satisfy \(|I|\ge a_n\). The identity relating the fourth moment and the fourth cumulant gives
\(
 \E[\tilde G_n(I)^4]
 =3\Var(G_n(I))^2+c_4(G_n(I)),
\)
and therefore
\(
 \E[\tilde G_n(I)^4]
 \le3\Var(G_n(I))^2+|c_4(G_n(I))|.
\)
By Proposition~\ref{pr:var}, we have
\(
 3\Var(G_n(I))^2
 \le C|W_n|^2|I|^2
 \le C_J|W_n|^2|I|^{5/4},
\)
because \(|I|^{3/4}\le |J|^{3/4}\). By Proposition~\ref{pr:cum}, we have
\(
 |c_4(G_n(I))|\le C|W_n|.
\)
The cutoff gives
\(
 |W_n|^2|I|^{5/4}
 \ge |W_n|^2\bigl(|W_n|^{-2/3}\bigr)^{5/4}
 =|W_n|^{7/6}.
\)
Thus, \(|W_n|\le |W_n|^2|I|^{5/4}\) for all sufficiently large \(n\). This
proves \eqref{eq:chent_gn}.
\end{proof}

\subsubsection{Davydov's criterion}
\label{ss:dav}
We prove Proposition~\ref{pr:dav}. For \(t\in J\), define
\[
 N_n(t):=G_n(t)-G_n(a)=G_n((a,t]),
 \qquad
 X_n(t):=\frac{N_n(t)-\E[N_n(t)]}{\sqrt{|W_n|}}.
\]
By Theorem~\ref{thm:mult}, the finite marginals of \(X_n\)
converge to those of
\(
 X(t):=Z(t)-Z(a).
\)
We apply \cite[Theorem~2]{davydov} with
\(
 \alpha_D:=4,\) and \( \beta_D:=\frac54,
\)
and cutoff \(a_n=|W_n|^{-2/3}\).

If \(|t-u|\ge a_n\), then \eqref{eq:chent_gn} gives
\(
 \E|X_n(t)-X_n(u)|^4\le C|t-u|^{5/4}.
\)
We next obtain a uniform bound on the fourth moment at a fixed parameter. If \(t-a\ge a_n\),
then
\(
 \E|X_n(t)|^4\le C|t-a|^{5/4}\le C_J.
\)
If \(0\le t-a<a_n\), monotonicity of \(N_n\) gives
\(
 |N_n(t)-\E[N_n(t)]|
 \le N_n(a+a_n)+\E[N_n(a+a_n)],
\)
where \(a+a_n\) is replaced by \(b\) when it exceeds \(b\). Writing
\(\mu_n:=\E[N_n((a+a_n)\wedge b)]\), we obtain
\[
\begin{aligned}
 \E|X_n(t)|^4
 \le C\E|X_n((a+a_n)\wedge b)|^4
 +C({\mu_n}/{\sqrt{|W_n|}})^4
 \le C_Ja_n^{5/4}+C|W_n|^2a_n^4\le C_J.
\end{aligned}
\]
Here Lemma~\ref{lem:exp} gives \(\mu_n\le C|W_n|a_n\).
Finally, write
\[
 X_n=X_n^\circ-X_n^*:=
\frac{N_n(t)}{\sqrt{|W_n|}}
-
\frac{\E[N_n(t)]}{\sqrt{|W_n|}}.
\]
Both components are nondecreasing. On the grid
\(
 t_k:=(a+ka_n)\wedge b,
\)
Lemma~\ref{lem:exp} gives
\[
\begin{aligned}
 \max_k\bigl(X_n^*(t_{k+1})-X_n^*(t_k)\bigr)
 =
 \max_k\frac{\E[G_n((t_k,t_{k+1}])]}{\sqrt{|W_n|}}
 \le Ca_n\sqrt{|W_n|}
 =C|W_n|^{-1/6}\longrightarrow0.
\end{aligned}
\]
This is deterministic and hence \(o_{\P}(1)\). All hypotheses of
Davydov's theorem are satisfied, and therefore
\(
 X_n\Rightarrow Z(\,\cdot\,)-Z(a)
 \text{ in }D(J).
\)
This proves Proposition~\ref{pr:dav}.

\begin{lemma}[Expectation of level increments]
\label{lem:exp}
Let \(I=(s,t]\subseteq J\). Then, for every \(n\ge1\),
\[
 \E[G_n(I)]
 =
 \int_{W_n\times(s,t]}
 \P\bigl(x\in\EE(\PP+\delta_x)\bigr)\,\Q(dx)
 \le |W_n|\mathbb F((s,t])
 \le |W_n|\|f\|_\infty|I|.
\]
\end{lemma}

\begin{proof}
This follows from the Mecke formula and the bound of the displayed
probability by one.
\end{proof}

\subsubsection{Variance estimate}
\label{ss:var}
The covariance expansion from Section~\ref{ss:cov} applies with both marks
restricted to the increment interval. Its diagonal contribution is bounded
by the expected number of candidate generators, and its direct-interaction
corrections are nonpositive. The remaining term is controlled by the
integrable overlap kernel from Section~\ref{sssec:t1}.

\begin{proof}[Proof of Proposition~\ref{pr:var}]
Fix \(I=(s,t]\subseteq J=[a,b]\), and recall the kernel \(K_b\) defined
in Section~\ref{sssec:t1}. The moment assumption in the proposition implies
\(\E[H_-^{3d/2}]<\infty\), so the covariance bounds from that section
apply. Restrict both mark integrals in the Mecke expansion of
Section~\ref{ss:cov} to \(I\), bound the diagonal indicator by one, and
discard the two nonpositive direct-interaction terms. By
\eqref{eq:extremality-overlap-bound},
\[
 \Var(G_n(I))
 \le |W_n|\mathbb F(I)
   +\mathbb F(I)^2\int_{W_n}\int_{W_n}K_b(v'-v)\,dv'\,dv.
\]
Equation~\eqref{eq:overlap-kernel-moments}, with \(\bar z=b\), gives
\(\int_{\mathbb R^d}K_b(u)\,du
=\kappa_d^2\E[(b-H)_+^d]<\infty\).
Enlarging the second spatial integration to \(\mathbb R^d\) therefore
yields
\(\Var(G_n(I))\le C|W_n|(\mathbb F(I)+\mathbb F(I)^2)
\le C|W_n||I|\), since \(\mathbb F(I)\le1\) and
\(\mathbb F(I)\le\|f\|_\infty|I|\). Centering leaves the variance
unchanged, which proves the proposition.
\end{proof}

\subsubsection{Cumulant bound}
\label{ss:cum}
We prove Proposition~\ref{pr:cum}. Let \(I=(s,t]\subseteq J\). Let \(Q_1,\ldots,Q_{k_n}\) be disjoint unit
cubes with centers \(z_1,\ldots,z_{k_n}\in\mathbb Z^d\) such that
\(W_n\subseteq\bigcup_iQ_i\) and \(k_n\le C_W|W_n|\). Put 
\(G_n(I)=\sum_{i=1}^{k_n}\Delta_i(I)\) where
\[
 \Delta_i(I):=
 \sum_{X_j\in\PP|_{(Q_i\cap W_n)\times I}}
 \one\{X_j\in\EE(\PP)\}.
\]
For brevity, write \(\Delta_i:=\Delta_i(I)\).
The cumulant proof uses a separation radius for labeled quadruples, a lattice counting estimate, and an algebraic decomposition of the fourth cumulant. Given lattice centers
\(z_1,\ldots,z_k\in\mathbb Z^d\) and a labeled quadruple
\(\mathbf i=(i_1,i_2,i_3,i_4)\in\{1,\ldots,k\}^4\), define
\[
\rho(\mathbf i):=
\max_{\varnothing\neq S\subsetneq\{1,2,3,4\}}
\operatorname{dist}\big(
 \{z_{i_a}:a\in S\},
 \{z_{i_b}:b\in S^{\complement}\}
\big).
\]

\begin{lemma}[Lattice shell bound]
\label{lem:lattice-shell-estimate}
Let \(z\in\mathbb Z^d\). Then,
\(
 \#\{z'\in\mathbb Z^d:m\le\|z-z'\|<m+1\}
	\in O(m^{d-1}).
\)
\end{lemma}

\begin{proof}
Place the unit cube centered at each lattice point \(z'\) satisfying
\(m\le\|z-z'\|<m+1\). These cubes are disjoint and lie in an annulus of
uniformly bounded thickness. Its volume is at most \(C(1+m)^{d-1}\).
\end{proof}

\begin{lemma}[Shell counting for labeled quadruples]
\label{lem:labelled-quadruple-shell-count}
Let \(z_1,\ldots,z_k\in\mathbb Z^d\) with \(k\le C_W|W_n|\). For every
integer \(m\ge0\),
\[
 \#\{\mathbf i\in\{1,\ldots,k\}^4:
 \rho(\mathbf i)\in[m,m+1)\}
 \le C|W_n|(1+m)^{3d-1}.
\]
\end{lemma}

\begin{proof}
Let \(m\ge1\), and let \(\mathbf i=(i_1,i_2,i_3,i_4)\) satisfy
\(\rho(\mathbf i)\in[m,m+1)\). Consider the graph on
\(\{1,2,3,4\}\) in which \(a\) and \(b\) are adjacent when
\(
 \|z_{i_a}-z_{i_b}\|<m+1.
\)
The inequality \(\rho(\mathbf i)<m+1\) is equivalent to connectivity of this
graph: every nontrivial cut has a crossing edge shorter than \(m+1\).
Choose a cut \(G,H\) attaining \(\rho(\mathbf i)\ge m\). Now, every crossing edge has length at least \(m\),
while connectivity supplies one of length less than \(m+1\). Any edge of a
connected graph can be extended to a spanning tree. Thus, the
quadruple admits a labeled spanning tree with one distinguished edge of
length in \([m,m+1)\) and its other two edges shorter than \(m+1\).

There are only finitely many labeled trees and choices of distinguished
edge. Choose one endpoint of the distinguished edge in \(O(k)\) ways, its
other endpoint in \(O((1+m)^{d-1})\) ways by
Lemma~\ref{lem:lattice-shell-estimate}, and each remaining labeled vertex in
\(O((1+m)^d)\) ways. Hence,
\[
 \#\{\mathbf i:\rho(\mathbf i)\in[m,m+1)\}
 \le Ck(1+m)^{d-1}(1+m)^{2d}
 \le C|W_n|(1+m)^{3d-1}.
\]
For \(m=0\), connectivity at threshold \(1\) forces all four lattice
centers to coincide, giving \(O(k)=O(|W_n|)\) labeled quadruples.
\end{proof}

\begin{lemma}[Cumulant decomposition across a partition]
\label{lem:fourth-semicluster-split}
Let \(K=\{1,2,3,4\}\), let \(G,H\) be a nontrivial partition of \(K\),
and put
\(
 X_A:=\prod_{a\in A}X_a.
\)
Then \(c_4(X_1,X_2,X_3,X_4)\) is a finite linear combination, with
coefficients depending only on \(K\), of terms
\(
 \Cov(X_S,X_T)
 \prod_{\ell=1}^{r}\E[X_{K_\ell}],
\)
where \(\varnothing\neq S\subseteq G\),
\(\varnothing\neq T\subseteq H\), and
\(K_1,\ldots,K_r\) are pairwise disjoint subsets of
\(K\setminus(S\cup T)\).
\end{lemma}

\begin{proof}
Start from the moment-cumulant formula
\[
 c_4(X_1,X_2,X_3,X_4)
 =
 \sum_{\pi\in\Pi(K)}
 (-1)^{|\pi|-1}(|\pi|-1)!
 \prod_{B\in\pi}\E[X_B].
\]
For each \(B\) meeting \(G\) and \(H\), write
\(
 \E[X_B]
 =
 \Cov(X_{B\cap G},X_{B\cap H})
 +
 \E[X_{B\cap G}]\E[X_{B\cap H}].
\)
Expand the finite sum. Terms with no covariance factor cancel because the
joint cumulant vanishes under factorization across the two nonempty groups.
Every remaining term contains a covariance factor across the split; expanding
any additional covariance factors as differences of moments yields the
stated representation.
\end{proof}

\begin{lemma}[Covariance bound for products of block scores]
\label{covBoundLem}
Assume \(\E[H_-^{3d}]<\infty\). Let
\(\alpha=(i_1,\ldots,i_r)\) and \(\beta=(j_1,\ldots,j_s)\), where
\(1\le r,s\le3\), and put
\(X_\alpha:=\prod_{a=1}^r\Delta_{i_a},\) and \( X_\beta:=\prod_{b=1}^s\Delta_{j_b}.\)
Uniformly in \(n\), \(I\subseteq J\), and the multi-indices,
\[
 |\Cov(X_\alpha,X_\beta)|
 \in O\bigl(\dist(\alpha,\beta)^{-3d-1}\bigr),
\]
where
\(
 \dist(\alpha,\beta):=
 \operatorname{dist}\bigl(
 \{z_{i_a}:1\le a\le r\},
 \{z_{j_b}:1\le b\le s\}
 \bigr).
\)
\end{lemma}
\begin{proof}
See Appendix~\ref{apx:func}.
\end{proof}

\begin{proof}[Proof of Proposition~\ref{pr:cum}]
By multilinearity, we have 
\(
 c_4(G_n(I))
 =\sum_{i_1,i_2,i_3,i_4=1}^{k_n}
 c_4(\Delta_{i_1},\Delta_{i_2},\Delta_{i_3},\Delta_{i_4}).
\)
This ordered sum over block indices includes all diagonal patterns automatically.
Tuples with one, two, three, or four distinct block indices are all present.
After expanding the block sums into sums over Poisson points, the equality partitions have types
\(
 4,3+1,2+2,2+1+1,1+1+1+1,
\)
with respectively \(1,4,3,6,1\) set partitions. Thus, all terms in which some Poisson points coincide are already contained in the moments and cumulants of the block counts.

Fix \(\mathbf i=(i_1,i_2,i_3,i_4)\), choose a nonempty proper subset \(G\subsetneq K:=\{1,2,3,4\}\) attaining \(\rho(\mathbf i)\), and put
\(H:=G^{\complement}\). Lemma~\ref{lem:fourth-semicluster-split} expresses the joint
cumulant as a finite sum of covariance factors across this split multiplied
by moment factors. All moment factors are uniformly bounded, and
Lemma~\ref{covBoundLem} gives
\[
 \bigl|c_4(\Delta_{i_1},\Delta_{i_2},\Delta_{i_3},\Delta_{i_4})\bigr|
 \le C(1+\rho(\mathbf i))^{-3d-1}.
\]
The covariance bound is uniform in \(I\subseteq J\). The only moment assumption on the negative part of the mark used here is
\(\E[(b-H)_+^{3d}]<\infty\), which follows from
\(\E[H_-^{3d}]<\infty\).
For \(m\ge1\), Lemma~\ref{lem:labelled-quadruple-shell-count} gives
\[
\begin{aligned}
&\sum_{\mathbf i:\rho(\mathbf i)\in[m,m+1)}
 \bigl|c_4(\Delta_{i_1},\Delta_{i_2},\Delta_{i_3},\Delta_{i_4})\bigr|
\le
 C|W_n|(1+m)^{3d-1}(1+m)^{-3d-1}
 =C|W_n|(1+m)^{-2}.
\end{aligned}
\]
The \(m=0\) contribution is \(O(|W_n|)\), and
\(\sum_{m\ge1}(1+m)^{-2}<\infty\). Hence,
\(
 \sup_{I \subseteq J}|c_4(G_n(I))|\in O (|W_n|).
\)
\end{proof}

\begin{proof}[Completion of the proof of Theorem~\ref{thm:func}(1)]
By Theorem~\ref{thm:mult}, the finite-dimensional distributions of
\(\hat G_{n,*}\) converge to those of the centered Gaussian process \(Z\)
with covariance kernel \(k\). Proposition~\ref{pr:dav} gives
\[
 \frac{\tilde G_n(t)-\tilde G_n(a)}{\sqrt{|W_n|}}
 \Rightarrow Z(t)-Z(a)
 \qquad\text{in }D(J).
\]
The scalar sequence \(\hat G_{n,*}(a)\) is tight by its one-dimensional
convergence. Hence, the corresponding pairs are tight in
\(\mathbb R\times D(J)\), and their finite-dimensional limits identify every
subsequential limit as \((Z(a),Z(\,\cdot\,)-Z(a))\). Since addition is
continuous from \(\mathbb R\times D(J)\) to \(D(J)\) in the \(J_1\)-topology,
\(
 \hat G_{n,*}\Rightarrow Z\qquad\text{in }D(J).
\)
For \(s<t\) in \(J\), covariance convergence and
Proposition~\ref{pr:var} give
\[
 \E[(Z(t)-Z(s))^2]
 =\lim_{n\to\infty}
 \frac{\Var(G_n((s,t]))}{|W_n|}
 \le C|t-s|.
\]
Since \(Z\) is Gaussian,
\(
 \E[(Z(t)-Z(s))^4]\le C|t-s|^2.
\)
Kolmogorov's continuity criterion gives a continuous modification of \(Z\).
After the affine reparametrization, the endpoint increment satisfies
\(Y(1)-Y(1-\delta)\to0\) in \(L^2\). This is Billingsley's endpoint
condition.
\end{proof}

\subsection{Tightness for the volume-weighted empirical process}
\label{sec:tight_V}

In this section, we prove Theorem~\ref{thm:func}(2). Let
\(J=[a,b]\subset\mathbb R\) be the compact parameter interval from the
theorem. The increasing affine map from \(J\) to \([0,1]\) is used only when
applying the criterion for tightness in the path space. All geometric estimates for the Laguerre tessellation, together with the localization, variance, and cumulant estimates, are formulated on the original mark interval \(J\).
For  \(I:=(s,t]\subseteq J\), put
\(
 \mathbb F(I)=\mathbb F((s,t]),\) and \(W_n^I:=W_n\times I.
\)
Define
\[
 G_n^{\ms V}(I)
 :=
 G_n^{\ms V}(t)-G_n^{\ms V}(s)
 =
 \sum_{X_i\in\PP|_{W_n^I}}A(X_i;\PP),
\]
where we recall
\(
 A(x;\mu)
 :=
 \one\{x\in\EE(\mu+\delta_x)\}|C(x;\mu)|,
 \) and \(
 A_x:=A(x;\PP).
\)
The cell and the score are computed in the full configuration. Moreover,
\(
 \widetilde G_n^{\ms V}(I)
 :=
 G_n^{\ms V}(I)-\E[G_n^{\ms V}(I)].
\)
Since \(\mathbb F\) has a bounded density on \(J\),
\(
 \mathbb F(I)\le \|f\|_\infty |I|.
\)
The target estimate is
\begin{equation}
\label{eq:chent_gnV}
 \E\big[\widetilde G_n^{\ms V}(I)^4\big]
 \le C|W_n|^2|I|^{5/4},
 \qquad
 |I|\ge |W_n|^{-2/3}.
\end{equation}
The remaining intervals are handled by Proposition~\ref{pr:davydovV}. The
fourth-moment identity is
\[
 \E\big[\widetilde G_n^{\ms V}(I)^4\big]
 =
 3\Var(G_n^{\ms V}(I))^2+c_4(G_n^{\ms V}(I)).
\]
For \(x=(v,h)\) and \(T\in\mathbb R\), recall
\[
 D^{\ms V}(x,T)
 :=
 \left\{
 (u,g)\in\X:
 g\le T,\
 \|u-v\|\le
 \sqrt{(T-h)_+}+\sqrt{(T-g)_+}
 \right\}.
\]

The proof follows the argument used in the unweighted case, but two additional
points require attention. The score is unbounded, and its value is not known
until the Laguerre cell is determined. Proposition~\ref{pr:varV} obtains
the increment variance bound from the covariance kernel already estimated
in Section~\ref{ss:covv}. Proposition~\ref{pr:cumV} controls the fourth
cumulant by approximating products of block scores with functionals of the
Poisson process in bounded spatial regions. Proposition~\ref{pr:momentV} then
combines these estimates, while Proposition~\ref{pr:davydovV} handles
intervals below the cutoff.

Assume throughout this subsection that \(\mathbb F\) has a bounded density on \(J\)
and that \(\E[H_-^{m_0}]<\infty\) for some \(m_0>11d/2\). Choose
\[
Q_*\ge\frac{10d}{2d-1},\qquad q_*:=3Q_*,\qquad p:=Q_*,\qquad
 a_0:=\frac{Q_*}{Q_*-1},\qquad b_0:=Q_*.
\]
Then \(5d(1/a_0-1/p)=5d(1-2/Q_*)\ge3d+1\). Choose
\(\gamma\) with \(3d<\gamma<m_0\). Corollary~\ref{lem:cover-time-moment-integrability}
provides the required moment of order \(\gamma\) for a point inserted through
the Mecke formula, uniformly for marks in \(J\). The largest moment required
of the negative part of an auxiliary mark is \(3d\), which is available under
the standing assumption. Constants may depend on \(J,d,W,\mathbb F\) and on these
fixed exponents, but not on \(n\) or on \(I\subseteq J\).

The argument is reduced to the following four propositions.

\begin{proposition}[Weighted Davydov reduction]
\label{pr:davydovV}
To prove tightness in \(D(J)\), it suffices to establish the fourth-moment bound for intervals \(I=(s,t]\subseteq J\) with \(|I|\ge|W_n|^{-2/3}\).
\end{proposition}

\begin{proposition}[Variance bound for weighted increments]
\label{pr:varV}
Under the standing assumptions of Section~\ref{sec:tight_V}, we have
\(
\sup_{n\ge1}\sup_{I=(s,t]\subseteq J}
\Var(G_n^{\ms V}(I))/(|W_n||I|)< \infty.
\)
\end{proposition}

\begin{proposition}[Fourth cumulant bound for weighted increments]
\label{pr:cumV}
Under the standing assumptions of Section~\ref{sec:tight_V}, we have
\(
\sup_{I=(s,t]\subseteq J} |c_4(G_n^{\ms V}(I))|\in O(|W_n|).
\)
\end{proposition}

\begin{proposition}[Fourth-moment bound for weighted increments]
\label{pr:momentV}
Under the standing assumptions of Section~\ref{sec:tight_V}, there is a constant \(C<\infty\) such that, for every \(n\ge1\) and every
interval \(I=(s,t]\subseteq J\) with \(|I|\ge |W_n|^{-2/3}\), we have
\(
\E\big[\widetilde G_n^{\ms V}(I)^4\big]\le C|W_n|^2|I|^{5/4}.
\)
\end{proposition}

We first derive the fourth-moment estimate and the functional convergence from the three central bounds. Their proofs occupy the remainder of the section.

\begin{proof}[Proof of Proposition~\ref{pr:momentV}]
For \(I=(s,t]\subseteq J\), we have
\(
 \E[\widetilde G_n^{\ms V}(I)^4]
 =
 3\Var(G_n^{\ms V}(I))^2+c_4(G_n^{\ms V}(I)).
\)
Hence, by Propositions~\ref{pr:varV} and~\ref{pr:cumV},
\(
 \E[\widetilde G_n^{\ms V}(I)^4]
 \le
 C|W_n|^2|I|^2+C|W_n|.
\)
Since \(I\subseteq J\), we have
\(
 |I|^2\le |J|^{3/4}|I|^{5/4}.
\)
Moreover, if \(|I|\ge |W_n|^{-2/3}\), then
\(
 |W_n|^2|I|^{5/4}
 \ge
 |W_n|^{7/6}
 \ge
 |W_n|
\)
for all sufficiently large \(n\); the finitely many remaining \(n\)'s are
absorbed into the constant. Therefore,
\(
 \E[\widetilde G_n^{\ms V}(I)^4]
 \le
 C_J|W_n|^2|I|^{5/4}.
\)
\end{proof}

\begin{proof}[Proof of Theorem~\ref{thm:func}(2)]
Theorem~\ref{thm:mult} gives convergence of the finite-dimensional
distributions. Tightness of the centered weighted process in \(D(J)\)
follows from Propositions~\ref{pr:davydovV} and~\ref{pr:momentV}.
Let \(Z^{\ms V}\) be the centered Gaussian process with covariance kernel \(k^{\ms V}\).
For \(s<t\) in \(J\),
\[
 \E[(Z^{\ms V}(t)-Z^{\ms V}(s))^2]
 =
 \lim_{n\to\infty}
 |W_n|^{-1}\Var(G_n^{\ms V}((s,t]))
 \le C|t-s|.
\]
Since \(Z^{\ms V}\) is Gaussian,
\(
 \E[(Z^{\ms V}(t)-Z^{\ms V}(s))^4]\le C|t-s|^2.
\)
Kolmogorov's continuity criterion therefore gives a continuous modification
of \(Z^{\ms V}\). Tightness and finite-dimensional convergence imply
\(
 \widehat G_{n,*}^{\ms V}\Rightarrow Z^{\ms V}
 \text{in }D(J)
\)
with the Skorokhod \(J_1\)-topology. No positivity or nondegeneracy assumption is needed.
\end{proof}

The remainder of the section proves the three central propositions. The
variance proof restricts the existing covariance expansion to the increment
interval and uses the uniform cell moment bound on \(J\). The cumulant proof uses cube
envelopes and the cover times of contributing points to control covariance
between separated products of block scores.

\subsection{Proof of Proposition~\ref{pr:davydovV}: the weighted Davydov reduction}
\label{ss:davV}

The condition for intervals below the cutoff follows from the same argument based on the first moment as in the unweighted case, provided that the typical cell volume has a uniformly bounded first moment.
\bel[Expectation of weighted level increments]
\label{lem:expV}
Let \(I=(s,t]\subseteq J\). Then, for every \(n\ge1\),
\[
 \E[G_n^{\ms V}(I)]
 =
 |W_n|\int_I\E[A_{(0,h)}]\,\mathbb F(dh)
 \le
 C|W_n||I|.
\]
\enl

\begin{proof}
By Mecke's formula,
\(
 \E[G_n^{\ms V}(I)]
 =
 \int_{W_n\times I}
 \E[A_{(v,h)}]\,dv\,\mathbb F(dh)
 =
 |W_n|\int_I\E[A_{(0,h)}]\,\mathbb F(dh).
\)
Lemma~\ref{lem:a2} and Lyapunov's inequality give
\(
 \sup_{h\in J}\E[A_{(0,h)}]<\infty.
\)
Therefore,
\(
 \E[G_n^{\ms V}(I)]
 \le
 C|W_n|\mathbb F(I)
 \le
 C\|f\|_\infty|W_n||I|.
\)
\end{proof}

\begin{proof}[Proof of Proposition~\ref{pr:davydovV}]
Put \(a_n:=|W_n|^{-2/3}\) and \(N_n^{\ms V}(t):=G_n^{\ms V}(t)-G_n^{\ms V}(a)\), and
\[
 X_n^{\ms V}(t)
 :=
 \frac{N_n^{\ms V}(t)-\E[N_n^{\ms V}(t)]}{\sqrt{|W_n|}},
 \qquad t\in J.
\]
Both \(N_n^{\ms V}\) and \(t\mapsto\E[N_n^{\ms V}(t)]\) are nondecreasing.

For increments of length at least \(a_n\),
Proposition~\ref{pr:momentV} gives
\(
 \E[|X_n^{\ms V}(t)-X_n^{\ms V}(s)|^4]
 \le C|t-s|^{5/4}.
\)
If \(t-a\ge a_n\), the same bound applied to \((a,t]\) gives a uniform bound on the fourth moment. If \(t-a<a_n\), put
\(u_n:=(a+a_n)\wedge b\). Monotonicity gives
\(
 |N_n^{\ms V}(t)-\E N_n^{\ms V}(t)|
 \le
 N_n^{\ms V}(u_n)+\E N_n^{\ms V}(u_n).
\)
Proposition~\ref{pr:momentV} and Lemma~\ref{lem:expV}, with the finitely many
values for which \(a_n>|J|\) absorbed into the constant, then give
\(
 \sup_{n,t\in J}\E|X_n^{\ms V}(t)|^4<\infty.
\)
On the grid \(t_k:=(a+ka_n)\wedge b\), Lemma~\ref{lem:expV} yields
\[
	{{|W_n|^{-1/2}}}
 \max_k
 {\E[G_n^{\ms V}((t_k,t_{k+1}])]}
 \le
 Ca_n\sqrt{|W_n|}
 =
 C|W_n|^{-1/6}\longrightarrow0.
\]
Thus, all hypotheses of Davydov's theorem \cite[Theorem~2]{davydov} are satisfied, and
tightness reduces to the moment estimate on \(|I|\ge a_n\).
\end{proof}

\subsection{Proof of Proposition~\ref{pr:varV}: variance bound}
\label{ss:var_inc_vol}

We use the covariance expansion from Section~\ref{ss:covv} with both marks
restricted to the increment interval. On the compact interval \(J\), the
cell moment bound controls the diagonal term uniformly. The mark factors
in Proposition~\ref{lem:weighted-lower-tail-boundary} are also uniformly bounded,
so its spatial kernel controls the off-diagonal term with a constant
independent of \(n\) and \(I\).

\begin{proof}[Proof of Proposition~\ref{pr:varV}]
If \(\mathbb F((-\infty,b])=0\), all increments vanish almost surely, so assume
\(\mathbb F((-\infty,b])>0\). The standing condition \(m_0>11d/2\) implies the moment
hypothesis of
Proposition~\ref{lem:weighted-lower-tail-boundary}. With both upper levels set
equal to \(b\), its factors \(M_b(h)\) are uniformly bounded for
\(h\in J=[a,b]\). Moreover, \eqref{eq:cell-at-cover-time} and
Corollary~\ref{lem:cover-time-moment-integrability}, applied with exponent
\(d\), imply a uniform second-moment bound. Thus, for
\(x=(v,h)\), \(x'=(v',h')\), with \(h,h'\in J\), and the
\(\rho>d+1\) from that lemma, \(|\Delta^{\ms V}(x,x')|\le C(1+\|v-v'\|)^{-\rho}\) and 
\(
 \sup_{h\in J}\E[A_{(0,h)}^2]\le C,
\)
For \(I=(s,t]\subseteq J\), restricting both mark integrals in the
Mecke expansion of Section~\ref{ss:covv} to \(I\) gives
\[
\begin{aligned}
 \Var(G_n^{\ms V}(I))
 &=|W_n|\int_I\E[A_{(0,h)}^2]\,\mathbb F(dh)
   +\iint_{(W_n\times I)^2}\Delta^{\ms V}(x,x')\,\Q(dx')\,\Q(dx)\\
 &\le C|W_n|\mathbb F(I)
   +C\mathbb F(I)^2\int_{W_n}\int_{W_n}
        (1+\|v-v'\|)^{-\rho}\,dv'\,dv.
\end{aligned}
\]
The spatial kernel is integrable because \(\rho>d\). Enlarging the
second spatial integration to \(\mathbb R^d\) therefore gives
\(\Var(G_n^{\ms V}(I))\le C|W_n|(\mathbb F(I)+\mathbb F(I)^2)
\le C|W_n||I|\). The last inequality uses
\(\mathbb F(I)\le1\) and \(\mathbb F(I)\le\|f\|_\infty|I|\),
and all constants are uniform in \(n\) and \(I\).
\end{proof}

\subsection{Cube envelopes and cover times of contributing points}
\label{subsec:cube_envelopes}

The cumulant proof only requires localization for points that actually
contribute to the volume-weighted statistic. We therefore build all cube
envelopes from extreme points. This is important because the score
\(A_x=\one\{x\in\EE(\PP+\delta_x)\}|C(x;\PP)|\) vanishes for non-extreme
points.
\begin{definition}[Cube envelopes]
Let \(Q_z:=z+[-1/2,1/2]^d\), \(z\in\mathbb Z^d\). For \(T\in\mathbb R\),
define
\[
 \widehat D_{z,J}^{\ms V}(T)
 :=
 \left\{
 (v',h')\in\X:
 h'\le T,\
 \sqrt{(T-a)_+}+\sqrt{(T-h')_+}
 \ge
 (\|v'-z\|-\sqrt d/2)_+
 \right\}.
\]
\end{definition}

\begin{lemma}[Envelope domination]
\label{lem:env}
Let \(z\in\mathbb Z^d\), \(x=(v,h)\in Q_z\times J\), and
\(T\in\mathbb R\). Then
\(
 D^{\ms V}(x,T)\subseteq\widehat D_{z,J}^{\ms V}(T).
\)
Thus, for any locally finite
\(\mathcal X\subseteq Q_z\times J\), we have
\(
 \bigcup_{x\in\mathcal X}D^{\ms V}(x,T_x)
 \subseteq
 \widehat D_{z,J}^{\ms V}\!\big(\sup_{x\in\mathcal X}T_x\big).
\)
\end{lemma}

\begin{proof}
If \((v',h')\in D^{\ms V}(x,T)\), then
\(
 \|v'-v\|
 \le
 \sqrt{(T-h)_+}+\sqrt{(T-h')_+}
 \le
 \sqrt{(T-a)_+}+\sqrt{(T-h')_+}.
\)
Since \(\|v-z\|\le\sqrt d/2\), we have
\(
 (\|v'-z\|-\sqrt d/2)_+\le\|v'-v\|,
\)
which proves the inclusion.
\end{proof}
\begin{definition}[Maximum cover time among contributing points]
For \(z\in\mathbb Z^d\), define
\[
 T_z^e
 :=
 \sup\left\{
 T(x;\PP):
 x\in\PP|_{Q_z\times J},\
 x\in\EE(\PP)
 \right\},
\]
with \(\sup\varnothing:=-\infty\).
\end{definition}

\begin{lemma}
\label{lem:extreme-cube-time-moment}
Assume \(\E[H_-^p]<\infty\) for some \(p>0\), and let \(q<p\).
Then
\(
 \sup_{z\in\mathbb Z^d}
 \E[(T_z^e-a)_+^q]<\infty.
\)
\end{lemma}

\begin{proof}
Apply Lemma~\ref{lem:cube_time_tail_under_E} with \(I=J\). It gives
\(
 \sup_{z\in\mathbb Z^d}
 \P(T_z^e-a\ge t)
 \le Ct^{-p},
 \qquad t\ge1.
\)
Integrating the tail then gives
\(
 \E[(T_z^e-a)_+^q]
 =
 q\int_0^\infty
 t^{q-1}\P(T_z^e-a\ge t)\,dt
 \le C_q.
\)
\end{proof}

\subsection{Proof of Proposition~\ref{pr:cumV}: fourth cumulant bound}
\label{ss:cumV}

We prove the bound for the fourth cumulant by decomposing the statistic into unit
cubes. Cover \(W_n\) by disjoint cubes \(Q_1,\ldots,Q_{k_n}\), with centers
\(z_1,\ldots,z_{k_n}\in\mathbb Z^d\), such that \(k_n\le C_W|W_n|\). For
\(Q_i^I:=(Q_i\cap W_n)\times I\), define
\(
\Delta_i^{\ms V}(I):=\sum_j\one\{X_j\in Q_i^I\}A_{X_j}.
\)
Then \(G_n^{\ms V}(I)=\sum_{i=1}^{k_n}\Delta_i^{\ms V}(I)\). The proof has two inputs.
First, products of at most three centered block scores have moments of every
fixed order. Second, conditioning on the Poisson process in a bounded
neighborhood of the relevant cubes gives an approximation whose error decays
polynomially with the radius. The latter estimate turns separation of the
cubes into decay of covariance.

\medskip

We first record the required moment bound.
\begin{lemma}[Moments of weighted block scores]
\label{lem:weighted-block-product-moments}
For every fixed \(q<\infty\), we have
\(
 \sup_{n,i,I}\|\Delta_i^{\ms V}(I)\|_q<\infty,
\)
where the supremum is over \(I=(s,t]\subseteq J\). Consequently, products
of at most three centered weighted block variables have uniformly bounded
\(L^q\)-norms for every fixed \(q<\infty\).
\end{lemma}

\begin{proof}
Let
\(
 N_i(I):=\PP((Q_i\cap W_n)\times I).
\)
For \(q\ge1\), we have
\(
 |\Delta_i^{\ms V}(I)|^q
 \le
 N_i(I)^{q-1}
 \sum_{x\in\PP|_{Q_i^I}}A_x^q.
\)
Mecke's formula gives
\[
\begin{aligned}
 \E|\Delta_i^{\ms V}(I)|^q
 &\le
 \int_{Q_i^I}
 \E\Big[
 (N_i(I)+1)^{q-1}
 A(x;\PP+\delta_x)^q
 \Big]\,\Q(dx)
 \le
 C_q\int_I
 \E[A_{(0,h)}^{2q}]^{1/2}\,\mathbb F(dh).
\end{aligned}
\]
The Poisson count has moments of every order uniformly in \(i,n,I\), and
Corollary~\ref{lem:cover-time-moment-integrability} gives
\(
 \sup_{h\in J}\E[A_{(0,h)}^r]<\infty
 \text{for every fixed }r.
\)
This proves the moment bound for products of block scores. Centering changes the norm by at most a
factor two, and H\"older's inequality gives the product bounds.
\end{proof}

We next prove that separation of the cubes yields decay of covariance.
\begin{lemma}[Covariance decay for products of weighted block scores]
\label{lem:weighted-block-cov-decay}
Let \(X_{\mathcal A}\) and \(X_{\mathcal B}\) be centered products of at most three centered weighted block variables, associated with sets of cube indices \(\mathcal A\) and \(\mathcal B\), each of cardinality at most three. If \(m:=\operatorname{dist}(\{z_i:i\in\mathcal A\},\{z_j:j\in\mathcal B\})\), then
\(
	|\Cov(X_{\mathcal A},X_{\mathcal B})|\in O\big(m^{-3d-1}\big).
\)
\end{lemma}

We can now prove the bound for the fourth cumulant.

\begin{proof}[Proof of Proposition~\ref{pr:cumV}]
By multilinearity of cumulants,
\(
c_4(G_n^{\ms V}(I))
=
\sum_{i_1,i_2,i_3,i_4=1}^{k_n}
c_4(\Delta_{i_1}^{\ms V},\Delta_{i_2}^{\ms V},\Delta_{i_3}^{\ms V},\Delta_{i_4}^{\ms V}).
\)
Fix \(\mathbf i=(i_1,i_2,i_3,i_4)\), and put
\(Y_a:=\Delta_{i_a}^{\ms V}-\E[\Delta_{i_a}^{\ms V}]\). Since cumulants of order at
least two are invariant under centering,
\(
c_4(\Delta_{i_1}^{\ms V},\Delta_{i_2}^{\ms V},\Delta_{i_3}^{\ms V},\Delta_{i_4}^{\ms V})
=
c_4(Y_1,Y_2,Y_3,Y_4).
\)
Choose a nonempty proper subset \(G\subsetneq K:=\{1,2,3,4\}\) attaining the
maximum in the definition of \(\rho(\mathbf i)\), and put \(H:=G^{\complement}\). Applying
Lemma~\ref{lem:fourth-semicluster-split} to \(X_a=Y_a\), we write the fourth
cumulant as a finite sum of terms
\[
\Cov
\Big(
\prod_{a\in S}Y_a,
\prod_{b\in T}Y_b
\Big)
\prod_{\ell=1}^{r}
\E\Big[\prod_{c\in K_\ell}Y_c\Big],
\]
where \(\varnothing\neq S\subseteq G\), \(\varnothing\neq T\subseteq H\), and
\(S\cup T,K_1,\ldots,K_r\) refines a partition of \(K\). All moment factors are
uniformly bounded by Lemma~\ref{lem:weighted-block-product-moments}. Moreover,
the two products in the covariance factor involve at most three centered block
variables each, and their sets of cube indices are separated by at least
\(\rho(\mathbf i)\). Lemma~\ref{lem:weighted-block-cov-decay} therefore gives
\(
\big|
c_4(\Delta_{i_1}^{\ms V},\Delta_{i_2}^{\ms V},\Delta_{i_3}^{\ms V},\Delta_{i_4}^{\ms V})
\big|
\le
C(1+\rho(\mathbf i))^{-3d-1}.
\)
By Lemma~\ref{lem:labelled-quadruple-shell-count},
\[
|c_4(G_n^{\ms V}(I))|
\le
C|W_n|+
C|W_n|\sum_{m\ge1}(1+m)^{3d-1}(1+m)^{-3d-1}
	\in O(|W_n|).
\]
\end{proof}

For the localization estimate, use the fixed exponents
\[
 p:=Q_*,
 \qquad
 a_0:=\frac{Q_*}{Q_*-1},
 \qquad
 b_0:=Q_*.
\]
Then \(1<a_0<p\), \(b_0=p\), and
\(
 5d(\frac1{a_0}-\frac1p)
 =
 5d(1-\frac2{Q_*})
 \ge3d+1.
\)
Lemma~\ref{lem:weighted-block-product-moments} gives uniform \(L^p\)- and
\(L^{b_0}\)-bounds for every centered product of at most three weighted block
variables.
For a finite set \(\mathcal A\) of cube indices, write
\(Q_{\mathcal A}:=\bigcup_{i\in\mathcal A}Q_i\) and
\(S_r(\mathcal A):=N_r(\mathcal A)\times\mathbb R\) where
\(
N_r(\mathcal A):=
\{v\in\mathbb R^d:\operatorname{dist}(v,Q_{\mathcal A})\le r\},
\)
Also put \(Q_{\mathcal A}^I:=(Q_{\mathcal A}\cap W_n)\times I\), and let
\(A_i^{\mathrm{ext}}\) be the region from
Lemma~\ref{lem:det-local-extreme}, translated to the cube \(Q_i\).

For a locally finite configuration \(\mu\), define
\[
\begin{aligned}
 \mathsf B_{\mathcal A}(r;\mu)
 :=
 &\Big\{
 \mu\Big(
 \bigcup_{i\in\mathcal A}A_i^{\mathrm{ext}}
 \setminus S_r(\mathcal A)
 \Big)>0
 \Big\}
 \cup
 \left\{
 \exists x\in
 \EE(\mu)\cap Q_{\mathcal A}^I:
 \mu\big(D^{\ms V}(x,T(x;\mu))\setminus S_r(\mathcal A)\big)>0
 \right\}.
\end{aligned}
\]

Appendix~\ref{apx:func} shows that the probability of this event decays polynomially in $r$.

\begin{lemma}[Probability of failure of localization]
\label{lem:weighted-bad-event}
Assume \(\E[H_-^{3d}]<\infty\), and let \(\gamma>3d\) be as in
Lemma~\ref{lem:weighted-envelope-tail}. Then,
\(
 \sup_{|\mathcal A|\le3}\sup_{I=(s,t]\subseteq J}
 \P\bigl(\mathsf B_{\mathcal A}(r;\PP)\bigr)
 \in O(r^{-5d}).
\)
\end{lemma}
\begin{proof}
See Appendix~\ref{apx:func}.
\end{proof}

Appendix~\ref{apx:func} then derives an approximation in \(L^{a_0}\) by conditioning on a bounded region.
\begin{lemma}[Approximation by conditioning on a bounded region]
\label{lem:weighted-spatial-truncation}
Let \(\mathcal A\) be a set of at most three cube indices, and let \(X_{\mathcal A}\) be a centered product of at most three centered variables of the form \(\Delta_i^{\ms V}(I)-\E[\Delta_i^{\ms V}(I)]\), with \(i\in\mathcal A\). Let \(\mathcal F_{\mathcal A}(r):=\sigma(\PP|_{S_r(\mathcal A)})\) and \(X_{\mathcal A}^{(r)}:=\E[X_{\mathcal A}\mid \mathcal F_{\mathcal A}(r)]\). Then,
\[
 \sup_{n,I,\mathcal A}
 \|X_{\mathcal A}-X_{\mathcal A}^{(r)}\|_{a_0}
 \in O\big((1+r)^{-3d-1}\big).
\]
\end{lemma}
\begin{proof}
	See Appendix~\ref{apx:func}.
\end{proof}

We now derive the covariance estimate.

\begin{proof}[Proof of Lemma~\ref{lem:weighted-block-cov-decay}]
For bounded \(m\), the result follows from the uniform moment bound for products of block scores. Assume therefore that \(m\) is large, and put \(r:=m/4\). Then, the
sigma-fields \(\mathcal F_{\mathcal A}(r)\) and \(\mathcal F_{\mathcal B}(r)\)
are independent. Hence, with
\(X_{\mathcal A}^{(r)}:=\E[X_{\mathcal A}\mid\mathcal F_{\mathcal A}(r)]\)
and
\(X_{\mathcal B}^{(r)}:=\E[X_{\mathcal B}\mid\mathcal F_{\mathcal B}(r)]\),
one has \(\Cov(X_{\mathcal A}^{(r)},X_{\mathcal B}^{(r)})=0\), and therefore
\[
\Cov(X_{\mathcal A},X_{\mathcal B})
=
\Cov(X_{\mathcal A}-X_{\mathcal A}^{(r)},X_{\mathcal B})
+
\Cov(X_{\mathcal A}^{(r)},X_{\mathcal B}-X_{\mathcal B}^{(r)}).
\]
By H\"older with the conjugate exponents \(a_0\) and \(b_0\), the
contraction property of conditional expectation, and the uniform \(L^{b_0}\)-bound
from Lemma~\ref{lem:weighted-block-product-moments}, we have
\(
|\Cov(X_{\mathcal A},X_{\mathcal B})|
\le
C\|X_{\mathcal A}-X_{\mathcal A}^{(r)}\|_{a_0}
+
C\|X_{\mathcal B}-X_{\mathcal B}^{(r)}\|_{a_0}.
\)
	Lemma~\ref{lem:weighted-spatial-truncation} gives \(O\big((1+r)^{-3d-1}\big)\). Since
\(r=m/4\), this gives \(O\big((1+m)^{-3d-1}\big)\).
\end{proof}
\bigskip

\noindent {\bf Acknowledgements.}
OpenAI's ChatGPT was used during the preparation of this manuscript. All mathematical statements, proofs, computations, references, and conclusions were independently checked and verified by the authors, who take full responsibility for the content of the manuscript.
CH was supported by a research grant (VIL69126) from VILLUM FONDEN. AG was supported by the Deutsche Forschungsgemeinschaft (DFG, German Research Foundation) under Germany’s Excellence Strategy (EXC 2044/2 – 390685587, Mathematics Münster: Dynamics–Geometry–Structure) and through SPP 2265 Random Geometric Systems.
\appendix

\bibliographystyle{abbrv}
\bibliography{lit}

\clearpage

\section{Proofs for Section~\ref{sec:varlb}}
\label{apx:varlb}

Now, we give the proof of Lemma~\ref{lem:local_volume_configuration}, where the most delicate item is  part (c). Here, we proceed by conditioning on the central height.
This leaves the central spatial location as the source of variation,
even when the mark distribution is purely atomic.
The required deterministic estimate must be uniform in the central
height and in the surrounding generators.

Let $\nu$ be the uniform probability measure on $B_{1}(0)$, and let
$V$ has law $\nu$.
Fix $q_0\in(0,0.01)$ small enough that the balls $B_{q_0}(u_i)$ are
pairwise disjoint.
For $q\in(0,q_0]$ and $\varepsilon\in(0,1]$, let
$\mathcal A_{q,\varepsilon}$ consist of pairs
$(\mathbf p,\mathbf t)\in(\mathbb R^d)^k\times\mathbb R^k$ satisfying
\(
	\max_{i \le k}\|p_i-u_i\|\le q,\) and \(\mathbf t \in  [26-\varepsilon, 26]^k\).
For $\|v\|\le1$, $h\in[-1,0]$, and
$(\mathbf p,\mathbf t)\in\mathcal A_{q_0,1}$, define
\(
f(v,h;\mathbf p,\mathbf t)
	:=|C((v,h);\sum_{i=1}^k\delta_{(p_i,t_i)})|, \; \text{ and } \;
m(h;\mathbf p,\mathbf t)
:=\operatorname{Var}_{V\sim\nu}(f(V,h;\mathbf p,\mathbf t)).
\)
Write $\mathbf u:=(u_1,\ldots,u_k)$ and
$\mathbf t_*:=(26,\ldots,26)$ for the reference surrounding
configuration, and put
$m_*:=\inf_{h\in[-1,0]}m(h;\mathbf u,\mathbf t_*)$.
The following lemma gives positive spatial variance at this
configuration and a uniform lower bound under small perturbations.

\begin{lemma}\label{lem:central_spatial_variance}
The constant $m_*$ is positive.
There exist $q\in(0,q_0]$ and $\varepsilon\in(0,1]$ such that
\[
m(h;\mathbf p,\mathbf t)\ge\frac12m_*>0
\qquad\text{for every }h\in[-1,0]
\text{ and }(\mathbf p,\mathbf t)\in\mathcal A_{q,\varepsilon}.
\]
\end{lemma}

\begin{proof}[Proof of Lemma~\ref{lem:local_volume_configuration}]


 \textit{Proof of (a)}: First we show that $(V_0,H_0)\in \EE(\PP)\cap W_n^z$ with probability one. 
This follows since for any $(v,h)\in K_0(y)$ we have $D(v,h)\su K(y)\setminus \bigcup_{i=1}^k K_i(y)$. Indeed for any $(v',h')$ with $\|v-v'\|\leq \sqrt{h-h'}$ it holds that $h'\leq a<a+25\delta^2$ and
\[
\|v'-y\|\leq \|v'-v\|+\|v-y\|\leq \sqrt{h-h'}+\delta< \sqrt{a+48\delta^2 - h'}+10\delta.
\]
At the same time $(V_i,H_i)\not\in \EE(\PP)\cap W_n^z$ for all $1\le  i\le  k$ with probability one, since for any $(v,h)\in K_i(y)$ and any $1\le  i\le  k$ we have $K_0(y)\su D(v,h)$. Indeed let $v'\in B_{\delta}(y)$ and $a-\delta^2\le  h'\le  a$, then
\[
\|v'-v\|\le  \|v'-y\|+\|y-v\|\le  \delta +(3.5+q)\delta <  5\delta\le  \sqrt{h-h'}.
\]

\textit{Proof of (b):} Our configuration ensures that $C((V_0,H_0); \PP)\subset B_{8\delta}(y)$ and the only possible neighbors of $C((V_0,H_0); \PP)$ are $C((V_i,H_i); \PP)$, $1\leq i\leq k$ (see Figure \ref{fig:Construction4}). In order to show this we need to ensure that $T((V_0,H_0);\PP)< a+48\delta^2$. This follows since $C_{a+48\delta^2}((V_0,H_0);\PP)=C((V_0,H_0);\PP)\cap B_{\sqrt{a+48\delta^2 - H_0}}(V_0)$ and $C_{a+48\delta^2}((V_i,H_i);\PP)=C((V_i,H_i);\PP)\cap B_{\sqrt{a+48\delta^2 - H_i}}(V_i)$ for $1\le  i\le  k$. Noting that  
 \begin{equation}\label{eq:3_3_26_1}
 \partial B_{\sqrt{a+48\delta^2 - H_0}}(V_0)\su L
 \end{equation}
 and since $q<0.01$ and $\|V_i-w_i'\|\leq 0.01\delta+q\delta<0.02\delta$ that
 \[
 B_{4.65\delta}(w_i')\su B_{4.67\delta}(V_i)\su B_{\sqrt{a+48\delta^2 - H_i}}(V_i),
 \]
we conclude by \eqref{eq:Condition1} that
\(
\partial B_{\sqrt{a+48\delta^2 - H_0}}(V_0)\su \bigcup_{i=1}^kB^\circ_{\sqrt{a+48\delta^2 - H_i}}(V_i).
\)
Hence, for any $w\in \partial B_{\sqrt{a+48\delta^2 - H_0}}(V_0)$ there is $1\le  i\le  k$, such that
\[
L((V_0,H_0),w)=\|w-V_0\|^2+H_0=a+48\delta^2>\|w-V_i\|^2+H_i=L((V_i,H_i),w), 
\]
implying by \eqref{eq:3_3_26_1} that $C((V_0,H_0);\PP)\subset B_{8\delta}(y)$. Moreover, there are no other points $(V',H')\in\PP$ with 
\[
L((V',H'),w)\leq \min_{0\leq i\leq k}L((V_i,H_i),w)\leq a+48\delta^2
\]
for some $w\in B_{8\delta}(y)$, since it would imply $H'\leq a+48\delta^2$ and 
\[
\|V'-y\|\leq \|V'-w\|+\|w-y\|\leq \sqrt{a+48\delta^2-H'}+8\delta,
\]
meaning $(V',H')\in K(y)$, contradicting $\mathcal{A}_y$. Hence, $C((V_0,H_0);\PP)$ is determined by its own generator and $(V_i,H_i)$, $1\leq i\leq k$. The bounds in (a) and (b) hold for every central generator in $K_0(y)$ with the other generators fixed. 

\emph{Proof of (c).}
Choose $q$ and $\varepsilon$ as in
Lemma~\ref{lem:central_spatial_variance}, and work on $\mathcal A_y$.
Define the normalized central and surrounding generators by
\(
	\widetilde V_0:=(V_0-y)/\delta,\) \(
\widetilde H_0:=(H_0-a)/\delta^2,\)
	\(\mathbf p:=((V_i-y)/\delta)_{i=1}^k,\) and \(
\mathbf t:=((H_i-a)/\delta^2)_{i=1}^k.
\)
The vectors $\mathbf p$ and $\mathbf t$ are $\mathcal G$-measurable,
and $(\mathbf p,\mathbf t)\in\mathcal A_{q,\varepsilon}$.
Conditionally on the central-box counts and the configuration outside
those boxes, their contents are independent.
The events $\mathcal A_y$ are measurable with respect to this
conditioning, so additionally revealing the contents of failed boxes
leaves the conditional law inside each successful box unchanged.
Thus, conditionally on $\mathcal G$, the variable $\widetilde V_0$
has law $\nu$ and is independent of $H_0$, whose law is $\mathbb F$
restricted and normalized to $[a-\delta^2,a]$.

Part (b) ensures that the central cell is determined by its own
and the surrounding generators.
After rescaling, the conditional variance decomposition and
Lemma~\ref{lem:central_spatial_variance} give
	for \(X_y=\delta^d f(\widetilde V_0,\widetilde H_0;\mathbf p,\mathbf t)\) that
\[
\begin{aligned}
\Var(X_y\mid\mathcal G)
\ge\E[\Var(X_y\mid\mathcal G,H_0)\mid\mathcal G]
=\delta^{2d}\E[m(\widetilde H_0;\mathbf p,\mathbf t)\mid\mathcal G]
\ge\frac12\delta^{2d}m_*.
\end{aligned}
\]
Taking $c_2(\mathbb F,z):=\delta^{2d}m_*/2$ proves (c).
The constant is positive and independent of $n$ and $y$.

\textit{Proof of (d):} Finally, we show that on $\mathcal{A}_{y}\cap \mathcal{A}_{y'}$ the random variables $X_{y}$ and $X_{y'}$ are independent conditionally on $\mathcal{G}$. Indeed, in this case the only randomness comes from the points $(V_0(\widetilde y), H_0(\widetilde y))$ with $\widetilde y\in \Lambda_n$ such that $\one\{\mathcal{A}_{\widetilde y}\}=1$. Conditionally on $\mathcal G$, these central generators are independent, as justified in (c). Moreover, as shown above for every such $\widetilde y$ we have $C((V_0(\widetilde y),H_0(\widetilde y)); \PP)\subset B_{8\delta}(\widetilde y)$. By construction of $\Lambda_n$ the balls $B_{8\delta}(\widetilde y)$ are disjoint and, hence, all cells $C((V_0(\widetilde y),H_0(\widetilde y));\PP)$ are disjoint almost surely. By (b), each $X_{\widetilde y}$ is a function of its own central generator and the fixed surrounding generators. Thus, $X_y$ and $X_{y'}$ are conditionally independent for distinct successful blocks; if either block fails, its variable is zero and the same conclusion holds.
\end{proof}


\begin{proof}[Proof of Lemma~\ref{lem:central_spatial_variance}]
We will first show by using a radial formula, that at the reference surrounding configuration the cell volume varies with the central location $v$ for every
fixed central height $h$.
Using continuity in the height and surrounding generators will allow to extend the statement to the 
perturbations of the reference surrounding configuration, while compactness will then give the uniform variance bound.

Let
$s(\theta):=\max_{1\le i\le k}u_i\cdot\theta$ for any $\theta \in S^{d-1}$.
The covering property gives
$s(\theta)>(350/349)(8-4.65)>1$. Indeed, since $B_8(0)$ is covered there is index $1\leq i\leq k$, such that $\|8\theta- 349u_i/350\|<4.65$. Then using Cauchy-Schwarz inequality we get 
\[
(8\theta- 349u_i/350)\cdot \theta=8-349/350 (u_i\cdot\theta)\leq \|8\theta- 349u_i/350\|<4.65
\]
and the bound follows. At the same time 
$s(\theta)\le 3.5$.

For $h\in[-1,0]$, write $\beta_h:=(3.5)^2+26-h$ and
$f_h(v):=f(v,h;\mathbf u,\mathbf t_*)$.
The cell volume can be computed by integrating along rays from the
origin. The cell can be described as follows
\[
	C((v,h);\sum_{i=1}^k\delta_{(p_i,t_i)})=\{u\in \R^d\colon 2\max_{i \le k}(u_i-v)\cdot u \le \beta_h-\|v\|^2\}.
\]
Since $\beta_h>\|v\|^2$, the origin belongs to the interior of the cell. Fix $\theta\in S^{d-1}$ and write $u=r\theta$, where $r\ge0$.
The inequalities in definition of $C(v,h)$ hold simultaneously precisely when
$2r\max_{1\le i\le k}(u_i-v)\cdot\theta
\le\beta_h-\|v\|^2$.
By the definition of $s$,
$\max_{1\le i\le k}(u_i-v)\cdot\theta
=s(\theta)-v\cdot\theta>0$.
Thus, the intersection of the cell with this ray is the segment
$0\le r\le\rho_h(v,\theta)$, where
\[
\rho_h(v,\theta)
:=
\frac{\beta_h-\|v\|^2}
     {2(s(\theta)-v\cdot\theta)}.
\]
Integration in polar coordinates, with $\sigma$ denoting surface
measure on $S^{d-1}$, therefore gives
	for \( \|v\|\le1\),
\[
\begin{aligned}
f_h(v)
=
\int_{S^{d-1}}\int_0^{\rho_h(v,\theta)}
r^{d-1}\,dr\,\sigma(d\theta)
=
\frac1d\int_{S^{d-1}}\rho_h(v,\theta)^d\,\sigma(d\theta)
=
\frac{(\beta_h-\|v\|^2)^d}{d\,2^d}
\int_{S^{d-1}}
(s(\theta)-v\cdot\theta)^{-d}\,\sigma(d\theta),
\end{aligned}
\]
For $d=1$, the same calculation uses counting measure on
$S^0=\{-1,1\}$ and adds the lengths of the two segments.
The denominators are uniformly bounded away from zero, which
justifies differentiation under the integral near $v=0$.
With $\Delta_v$ denoting the Laplacian in $v$, we obtain
\[
\begin{aligned}
\Delta_v f_h(0)
=f_h(0)\Big(
d(d+1)
\frac{\int_{ S^{d-1}}s(\theta)^{-d-2}\,\sigma(d\theta)}
     {\int_{ S^{d-1}}s(\theta)^{-d}\,\sigma(d\theta)}
-\frac{2d^2}{\beta_h}\Big)
\ge d f_h(0)\Big(\frac{d+1}{(3.5)^2}-\frac{2d}{\beta_h}\Big)>0,
\end{aligned}
\]
since $\beta_h\ge (3.5)^2+26>2(3.5)^2$.
Thus, $f_h$ is nonconstant on $B_{1}(0)$.
Its continuity and the positive $\nu$-measure of every nonempty
open subset of the ball imply
$m(h;\mathbf u,\mathbf t_*)>0$ for every $h\in[-1,0]$.

To make this positivity uniform, consider $\|v\|\le1$,
$h\in[-1,0]$, and
$(\mathbf p,\mathbf t)\in\mathcal A_{q_0,1}$.
The deterministic estimates used in part (b) of
Lemma~\ref{lem:local_volume_configuration} apply throughout this
parameter set and, together with $\|p_i\|\ge r-q_0$, give
\[
C((v,h);\sum_{i=1}^k\delta_{(p_i,t_i)})\subseteq B_{8}(0),
\qquad\|p_i-v\|\ge r-q_0-1>0.
\]
Every defining boundary is therefore a proper affine hyperplane.
For a convergent sequence of parameters, the cell indicators
converge outside the union of the limiting constraint hyperplanes,
which has Lebesgue measure zero.
Since each cell indicator is bounded by the indicator function
of $B_8(0)$, the dominated convergence theorem implies joint
continuity of $f(v,h;\mathbf p,\mathbf t)$.
Since the volume of the cell is uniformly bounded by $|B_{8}(0)|$,
integration of its first and second powers against $\nu$ also
proves continuity of $m(h;\mathbf p,\mathbf t)$.
In particular, $h\mapsto m(h;\mathbf u,\mathbf t_*)$ is continuous
and positive on $[-1,0]$, so
$m_*=\min_{h\in[-1,0]}m(h;\mathbf u,\mathbf t_*)>0$.
Uniform continuity of $m$ on the compact set
$[-1,0]\times\mathcal A_{q_0,1}$ provides
$q\in(0,q_0]$ and $\varepsilon\in(0,1]$ such that
$m(h;\mathbf p,\mathbf t)\ge m_*/2$ throughout
$[-1,0]\times\mathcal A_{q,\varepsilon}$.
\end{proof}

\section{Proofs for Section~\ref{sec:var}}
\label{apx:out}
This appendix proves Lemmas~\ref{lem:weighted-deterministic-resampling} and~\ref{lem:weighted-deterministic-overlap}.

\begin{proof}[Proof of Lemma~\ref{lem:weighted-deterministic-resampling}]
We first prove part~(i). Put \(B=D^V(x,t)\).  Since
\(t\ge h\), the downward paraboloid \(D(x)\) is contained in \(B\).  Indeed,
if \(y=(w,s)\in D(x)\), then \(s\le h\le t\) and
\(
 \|w-v\|\le\sqrt{h-s}
 \le
 \sqrt{t-h}+\sqrt{t-s}.
\)
It remains to check that the truncated cell is also determined by the points
of \(\PP\) in \(B\).  Let \(u\) satisfy \(L(x,u)\le t\), and suppose
that a point \(y=(w,s)\) can beat \(x\) at \(u\).  Then \(s\le t\) and we have
\(y\in B\) because
\[
 \|w-v\|
 \le
 \|w-u\|+\|u-v\|
 \le
 \sqrt{t-s}+\sqrt{t-h}.
\]
\eqref{eq:deterministic-local-score} and the truncated cell are measurable
with respect to \(\PP|_B\).

If \(x\) is not extreme, both \(A_x\) and \(A_x^{[t]}\) vanish.  On the
extremality event, Lemma~\ref{lem:truncated-cell-geometry} gives
\(C_t(x;\PP)\subseteq C(x;\PP)\), and equality holds when
\(T_x\le t\).  Hence,
\(
 0\le A_x^{[t]}\le A_x,
	\)
	and
	\(
 A_x-A_x^{[t]}=0\quad\text{on }\{T_x\le t\}.
\)
Put \(S_x:=(T_x-h)_+\).  By \eqref{eq:cell-at-cover-time}, we have
\(
 |A_x-A_x^{[t]}|^p
 \le
 \kappa_d^p\one\{A_x>0\}S_x^{pd/2}\one\{S_x>t-h\}.
\)
For every \(a>0\), we have
\(
 S_x^{pd/2}\one\{S_x>t-h\}
 \le
 (t-h)^{-a}S_x^{pd/2+a}.
\)
The estimate \eqref{eq:deterministic-truncation-error} now follows from
Corollary~\ref{lem:cover-time-moment-integrability}.

We next prove part~(ii).  Adding a point can only shrink the cell
and can only destroy the extremality of \(x\).  Thus, if
\(A(x;\mu)=0\), the assertion is immediate.  Suppose
\(A(x;\mu)>0\), write \(T=T(x;\mu)\), and let
\(y=(w,s)\notin D^V(x,T)\).  Every \(u\in C(x;\mu)\) satisfies
\(\|u-v\|\le\sqrt{T-h}\).  If \(s>T\), then
\(L(y,u)>T\ge L(x,u)\).  If \(s\le T\), then
\(
 \|w-v\|>\sqrt{T-h}+\sqrt{T-s},
\)
and therefore
\(
 \|w-u\|
 \ge
 \|w-v\|-\|u-v\|
 >
 \sqrt{T-s}.
\)
Again \(L(y,u)>T\ge L(x,u)\).  Hence, \(y\) loses against \(x\) throughout
the original cell, so neither the cell nor the extremality indicator changes.
This proves (ii).

It remains to prove~(iii).  First, note that
\(\PP_x^{[t]}=\PP|_{B^{\complement}}+\PP'|_B\) has the same law as
\(\PP\), and that \(\PP|_B\) is independent of
\(\PP_x^{[t]}\).  Since \(A_x^{[t]}\) is measurable with respect to
\(\PP|_B\), we have
\(
 \E\Big[
 (A_x^{[t]}-\E[A_x^{[t]}])A_{x'}^{(x,t)}
 \Big]=0.
\)
Subtracting this identity from
\(
 \Cov(A_x^{[t]},A_{x'})
 =\E[(A_x^{[t]}-\E[A_x^{[t]}])A_{x'}]
\)
proves \eqref{eq:deterministic-resampling-identity}.

It remains to prove \eqref{eq:deterministic-resampling-change-probability}.
Write
\(
 R:=\PP|_{B^{\complement}},
 \)
 \(
 U:=\PP|_B,
 \)
 and
 \(
 U':=\PP'|_B,
\)
so that the original and resampled configurations are
\(R+U\) and \(R+U'\), respectively.  The measure \(\Q(B)\) is finite,
so \(U\) and \(U'\) contain only finitely many points almost surely.  Let
\[
 \mathscr D_0:=\mathscr D_{x'}(R+U),
 \qquad
 \mathscr D_1:=\mathscr D_{x'}(R+U').
\]
Suppose first that \(U'(B\cap\mathscr D_0)=0\). Enumerate the finitely many
points of \(U'\) as \(y_1,\ldots,y_k\). Each \(y_j\) lies outside
\(\mathscr D_0\). By part~(ii), insertion of \(y_1\) leaves the score, the
cell, and the cover time unchanged. Hence, the support region after the first
insertion is still \(\mathscr D_0\). Repeating this argument inductively shows
that insertion of all points of \(U'\) leaves the score unchanged.
The same argument, starting from \(R+U'\), shows that
\(U(B\cap\mathscr D_1)=0\) implies that insertion of all points of \(U\)
leaves the score unchanged. Both insertion procedures lead to the common
configuration \(R+U+U'\). Therefore,
\[
 \{A_{x'}\ne A_{x'}^{(x,t)}\}
 \subseteq
 \{U'(B\cap\mathscr D_0)\ge1\}
 \cup
 \{U(B\cap\mathscr D_1)\ge1\}.
\]
Conditionally on \(\PP\), the first count is Poisson with mean
\(\Q(B\cap\mathscr D_0)\).  Hence, its probability is bounded by that
mean.  For the second term, observe that \(U\) is independent of
\(R+U'=\PP_x^{[t]}\).  Conditioning on \(R+U'\) gives the same bound
with \(\mathscr D_1\).  Since \(R+U'\) has the same law as \(\PP\),
the two expectations coincide and yield
\(
 \P(A_{x'}\ne A_{x'}^{(x,t)})
 \le
 2\E\big[
 \one\{A_{x'}>0\}\Q(B\cap D_{x'}^V)
 \big].
\)
This completes the proof.
\end{proof}

\begin{proof}[Proof of Lemma~\ref{lem:weighted-deterministic-overlap}]
On
\(\EE_{x'}\), put
\(
 S_{x'}:=T_{x'}-h'\ge0,
	\)
	and
	\(
 a:=h\vee h'.
\)
Fix \((u,s)\in B_r\cap D_{x'}^V\), and suppose that
\(S_{x'}<r^2/256\). If \(s>a-r^2/16\), membership in \(B_r\) and
\(D_{x'}^V\) gives \(\|u-v'\|\le \sqrt{S_{x'}}+\sqrt{h'+S_{x'}-s}\) and
\[
 \|u-v\|\le \frac r{16}+\sqrt{h+\frac{r^2}{256}-s}.
\]
Since \(h-a\le0\), \(h'-a\le0\), and \(s>a-r^2/16\), both quantities
under the last two square roots are smaller than \(17r^2/256\). Therefore,
\(
 r\le\|u-v\|+\|u-v'\|
 <\frac{2+2\sqrt{17}}{16}r<r,
\)
a contradiction. The point \((u,s)\) was arbitrary. Hence, on
\(\{S_{x'}<r^2/256\}\), we have
\(
 B_r\cap D_{x'}^V
 \subseteq
 \Big\{(u,s):s\le a-{r^2}/{16}\Big\}.
\)
We next bound \(\Q(B_r)\). The section of
\(B_r\) at mark \(s\le h+r^2/256\) is a ball of radius
\(
 r/{16}
 +
 \sqrt{h+{r^2}/{256}-s},
\)
and therefore has \(d\)-dimensional volume at most
\(
 C\big(
 r^d+
 \big(h+{r^2}/{256}-s\big)_+^{d/2}
 \big).
\)
It follows that
\begin{align}
 \Q(B_r)
 \le
 C r^d
 +
 C\int_{-\infty}^{h+r^2/256}
 \big(h+\frac{r^2}{256}-s\big)^{d/2}\mathbb F(ds)
 \le
 C r^d
 +
 C\,\E\big[
 \big(z+\frac{r^2}{256}-H\big)_+^{d/2}
 \big]
 \le
 C(1+r^d),
 \label{eq:deterministic-region-measure}
\end{align}
uniformly in \(h\le z\). Here we used
\(\E[H_-^{d/2}]<\infty\), which follows from
\(\E[H_-^q]<\infty\).
Therefore,
\begin{align}
&\E\Big[
 \one\{\EE_{x'}\}\,
 \Q(B_r\cap D_{x'}^V)
 \Big]
\le
 \Q(B_r)\,
 \P\big(
 \EE_{x'},
 S_{x'}\ge{r^2}/{256}
 \big)
 +
 \Q\Big(
 B_r\cap
 \left\{(u,s):s\le a-{r^2}/{16}\right\}
 \Big).
 \label{eq:overlap-decomposition}
\end{align}
For the first term, Corollary~\ref{lem:cover-time-moment-integrability} and
Markov yield
\(
 \P\big(
 \EE_{x'},
 S_{x'}\ge{r^2}/{256}
 \big)
 \le
 C\bigl(1+(z'-h')_+^q\bigr)r^{-2q}.
\)
Combining this estimate with \eqref{eq:deterministic-region-measure} gives
\begin{equation}\label{eq:overlap-large-cover-time}
 \Q(B_r)\,
 \P\left(
 \EE_{x'},
 S_{x'}\ge{r^2}/{256}
 \right)
 \le
 C\bigl(1+(z'-h')_+^q\bigr)
 r^{-2(q-d/2)}.
\end{equation}
For the second term in \eqref{eq:overlap-decomposition}, the section bound
above gives
\begin{align}
&\Q\Big(
 B_r\cap
 \Big\{(u,s):s\le a-\frac{r^2}{16}\Big\}
 \Big)
\le
 C\int_{-\infty}^{a-r^2/16}
 \Big[
 r^d+
 \big(h+\frac{r^2}{256}-s\big)_+^{d/2}
 \Big]\mathbb F(ds).
 \label{eq:overlap-lower-mark-integral}
\end{align}
Since \(a\le z\vee z'\), there exist \(r_0<\infty\) and \(c>0\), such that for \(r\ge r_0\),
\(
 s\le a-{r^2}/{16}
 \quad\Longrightarrow\quad
 s_-\ge c r^2.
\)
On this integration range, using \(h\le z\), we have
\(
 \Big(h+{r^2}/{256}-s\Big)_+^{d/2}
 \le
 C\Big(r^d+s_-^{d/2}\Big).
\)
Therefore,
\begin{align*}
&\int_{-\infty}^{a-r^2/16}
 \Big[
 r^d+
 \big(h+\frac{r^2}{256}-s\big)_+^{d/2}
 \Big]\mathbb F(ds)
\le
 C r^d\P(H_-\ge c r^2)
 +
 C\E\left[
 H_-^{d/2}\one\{H_-\ge c r^2\}
 \right]
\le
 C r^{-2(q-d/2)}.
\end{align*}
For \(1\le r<r_0\), the same estimate, with a larger constant and with
\((1+r)^{-2\rho}\) on the right-hand side, follows from
\eqref{eq:deterministic-region-measure}. Since
\(q>\rho+d/2\), we have \(2(q-d/2)>2\rho\). Combining the last estimate with
\eqref{eq:overlap-large-cover-time} proves
\eqref{eq:deterministic-overlap-final}.
It remains to prove \eqref{eq:point-hit-final}. On the event
\(\{A_x>0,\ x'\in D_x^V\}\), put
\(
 S_x:=T_x-h\ge0,
\) and \(
 \Delta:=(h-h')_+.
\)
The definition of \(D_x^V\) implies
\(
 r
 \le
 \sqrt{S_x}+\sqrt{h+S_x-h'}
 \le
 \sqrt{S_x}+\sqrt{S_x+\Delta}
 \le
 2\sqrt{S_x}+\sqrt\Delta.
\)

Suppose first that \(\sqrt\Delta<r/2\). Then
\(
 2\sqrt{S_x}
 \ge
 r-\sqrt\Delta
 >
 \frac r2,
\)
and hence \(S_x>r^2/16\). Corollary~\ref{lem:cover-time-moment-integrability}
and Markov's inequality give
\begin{align}
&\P\left(
 A_x>0,\ x'\in D_x^V,\ \sqrt\Delta<\frac r2
 \right)
\le
 \P\left(
 A_x>0,\ S_x>\frac{r^2}{16}
 \right)
 \le
 C\bigl(1+(z-h)_+^m\bigr)r^{-2m}.
 \label{eq:point-hit-large-cover-time}
\end{align}
Since \(m>\rho\) and \(r\ge1\), this is bounded by
\(
 C\bigl(1+(z-h)_+^m\bigr)(1+r)^{-2\rho}.
\)
Suppose finally that \(\sqrt\Delta\ge r/2\). Then
\(
 h-h'\ge\frac{r^2}{4}.
\)
Since \(h\le z\) and \(h'\le z'\), we have
\(
 h-h'
 \le
 z-h'
 \le
 (z'-h')_+ +(z-z')_+.
\)
Thus, there exist \(r_1<\infty\) and \(c_1>0\), depending only on \(z,z'\),
such that for \(r\ge r_1\),
\(
 (z'-h')_+\ge c_1r^2.
\)
Because \(m>\rho\), we have
\(
 1
 \le
 C\bigl(1+(z'-h')_+^m\bigr)(1+r)^{-2\rho}
\)
on this event. For \(1\le r<r_1\), the same inequality follows after
increasing \(C\), because the mark factors are at least one. Combining this estimate
with \eqref{eq:point-hit-large-cover-time} proves
\eqref{eq:point-hit-final}. Interchanging
\((x,z)\) and \((x',z')\) proves the symmetric estimate.
\end{proof}

\section{Appendix for Section \ref{sec:univ_proof}}\label{apx:univ}

\subsection{Proof of Lemma \ref{lem:fmombd}}\label{apx:univprf}
%
%
We first prove the bounds for $f_\alpha^{(i)}$, $i=1,2$ and then prove the bounds for the remaining integrals.
Define 
$u(h):=\int_{h}^z (t - h)^{d/2} \, F(dt),$ $h \le z$. 
For fixed $z\in\R$, bounding the integrand using $t\le z$
gives $u(h)\in O(|h|^{d/2})$ as $|h|\to\infty$.
Then, for $x_1\in W_n^z$,
\begin{align*}
g^{(z)}(x_1) &\le \int_{\R^{d+1}} \one \big\{v_2 \in W_n, h_2 \le z, x_1 \in D(x_2)\big\} \exp\big(-\zeta {\Q}(h_2)\big) \mathbb{Q}(dx_2) \\
&\le \kappa_d\int_{h_1}^z (h_2 - h_1)^{d/2} \exp\big(-\zeta \Q(h_2)\big) \, \mathbb \mathbb F(dh_2) \leq \kappa_d u(h_1) \in O(|h_1|^{d/2}). 
\end{align*}

\smallskip

\noindent \underline{\em Moment bounds on $f_\alpha^{(i)}$, $i=1,2$.} 
\label{ssec:moment_bounds}
Noting that $S^{(z)}(x_1) \le C( 1+ g^{(z)}(x_1)^5) \in O (|h_1|^{5d/2})$, with $f_\alpha^{(1)}$ as in \eqref{eq:fal}, we obtain
\begin{align}\label{eq:f1bd1}
	f_\alpha^{(1)}(x_2) 
    & \lesssim \int_{\R^{d+1}} \one \{x_1 \in W_n^z, x_2 \in D(x_1)\} (1+|h_1|^{5d/2}) \exp\big(-\alpha {\Q}(h_1)\big) \mathbb{Q}(dx_1).
\end{align}
Note that when $v_2 \in W_{2n}$, we can simply integrate over $v_1 \in B(v_2, \sqrt{h_1 - h_2})$. On the other hand, {since $W_n = n W$ for a compact and convex $W$ with $0 \in \operatorname{int}(W)$,} for $v_1 \in W_n, v_2 \in W_{2n}^c$, we have $(v_2,h_2) \in D(v_1,h_1) \implies h_1-h_2 \ge \|v_1 - v_2\|^2 \ge c \|v_2\|^2$ for some $c>0$. Thus, from \eqref{eq:f1bd1}, noting that $u(h)$ is decreasing  in $h$ and since $\E\big[H_{-}^{3d}\big]<\infty$, we can further bound
\begin{align}\label{eq:f1bd}
	f_\alpha^{(1)}(x_2) & \le C  \one\{v_2 \in W_{2n}\} \one\{h_2 \le z\}  \int_{h_2}^z (h_1 - h_2)^{d/2} (1+|h_1|^{5d/2}) \, \mathbb F(dh_1)\nonumber \\
	&\qquad + C\one\{v_2 \in W_{2n}^c\} \one\{h_2 \le z-c\|v_2\|^2\} |W_n| \int_{h_2}^z (1+|h_1|^{5d/2})\, \mathbb F(dh_1)\nonumber\\
	& \le C \one\{x_2 \in W_{2n}^z\}  (1+|h_2|^{d/2}) + C\one\{v_2 \in W_{2n}^c\} \one\{h_2 \le z-c\|v_2\|^2\}  |W_n|.
\end{align}
Then, again using $\E\big[ H_{-}^{d}\big]<\infty$, there exists a constant $C_1$ depending on $z$ so that for $i= 1,2$,
\begin{equation}\label{eq:f12}
\mathbb{Q}(f_{\alpha}^{(1)} )^i \le C_1 |W_{2n}| + C_1|W_n|^i \int_n^\infty F(z-cr^2) r^{d-1} dr  \lesssim n^d =O (|W_n|).
\end{equation}
The final inequality above is obtained by noting that for $s\leq 0$ we have by Markov's inequality that
\[
F(s)=\P(H\leq s)=\P(-H\ge -s)=\P(H_-^d\ge (-s)^d)\leq \E[H_-^d](-s)^{-d}.
\]
This indeed yields for $n>\sqrt{(z+1)/c}$ (and hence $z-cr^2\leq -1$) that 
\begin{align}\label{eq:intFbd}
    \int_n^\infty F(z-cr^2) r^{d-1} dr\leq \E[H_-^d]\int_n^\infty (cr^2-z)^{-d}r^{d-1}dr \leq C n^{-d} = O(|W_n|).
\end{align}
Similarly, for any $\alpha>0$, we have using $\E[H_-^{3d}]<\infty$ that
\begin{align*}
	f_\alpha^{(2)}(x_2) &\le C\one \{x_2 \in W_n^z\} \int_{D(x_2)}  (1+u(h_1)^5) \exp\big(-\alpha \Q(h_2) \big) \mathbb{Q}(dx_1)\\
	& \le C \one \{x_2 \in W_n^z\} \int_{-\infty}^{h_2} (h_2-h_1)^{d/2} (1+|h_1|^{5d/2})\, \mathbb F(dh_1) \le C \one \{x_2 \in W_n^z\} (1+|h_2|^{d/2}).
\end{align*}
Since we assume $\E\big[H_-^{d}\big]<\infty$, noting that $h_2 \le z$, one then obtains
  \begin{equation}\label{eq:f22}
      \mathbb{Q}f_{\alpha}^{(2)} , \mathbb{Q}(f_{\alpha}^{(2)} )^2 \in O(|W_n|).
  \end{equation}

\smallskip

%
%
\noindent \underline{\em Moment bounds on $f_\alpha^{(3)}$ and remaining terms.}

Finally, we turn to $f_\alpha^{(3)}$. Note that $q(x_1,x_2) = 0$ if $h_1 \vee h_2>z$ while for $h_1\vee h_2 \le z$, we have 
\begin{align*}
	q(x_1,x_2) &= \int_{W_n^z}  \one \left\{ \{x_1,x_2\} \subset D(x_3) \right\} e^{-\Q(h_3)} \mathbb{Q}(dx_3)\\
	&\le C\int_{W_n^z}  \one \left\{ \|v_1 - v_3\|^2 \le h_3 - h_1 \right\} \one \left\{ \|v_2 - v_3\|^2 \le h_3 - h_2 \right\} \mathbb{Q}(dx_3).
\end{align*}
For the indicator regarding $x_2$ to be non-zero, since $v_3 \in W_n$, we must have 
$$
d(v_2, W_n)^2:= \min_{v \in W_n} \|v_2-v\|^2 \le C\|v_2 - v_3\|^2 \le C( z-h_2).
$$
Furthermore, for $x_1, x_2 \in D(x_3)$, the triangle inequality yields 
$$\|v_1 - v_2\|^2 \le 2 \|v_1 - v_3\|^2 + 2\|v_2 - v_3\|^2 \le 2(h_3 -h_1) + 2(h_3 - h_2) \le 4(z - h_1 \wedge h_2).$$ Hence, using $\Q(h_1 \vee h_2) \le \Q(h_3)$ for $h_1 \vee h_2 \le h_3$, we have
\begin{align*}
	q(x_1,x_2) &\le C \one\{\|v_1 - v_2\|^2 \le 4(z - h_1 \wedge h_2)\} \one\{d(v_2, W_n)^2 \le z - h_2\} \\
	&\qquad \times \int_{h_1 \vee h_2}^z \big(\int_{\R^d} \one\{v_3 \in B_{\sqrt{h_3-h_1}}(v_1) \cap B_{\sqrt{h_3-h_2}}(v_2)\}\; dv_3 \big) \, \mathbb F(dh_3) \\
	&\le \one\{\|v_1 - v_2\|^2 \le 4(z - h_1 \wedge h_2)\} \one\{d(v_2, W_n)^2 \le z - h_2\} \int_{h_1 \vee h_2}^z (h_3-(h_1 \vee h_2))^{d/2} \, \mathbb F(dh_3) \\
	&= \one\{\|v_1 - v_2\|^2 \le 4(z - h_1 \wedge h_2)\} \one\{d(v_2, W_n)^2 \le z - h_2\} \, u(h_1 \vee h_2).
\end{align*}
Thus, $f_\alpha^{(3)}(x_2)=0$ for $h_2>z$ while for $h_2 \le z$, integrating w.r.t.\ $v_1$ in the first step, for $\alpha \in (0,1)$ we have
\begin{align*}
	f_\alpha^{(3)}(x_2) 
	&\lesssim \one\{d(v_2, W_n)^2 \le z - h_2\}\int_{-\infty}^z (z-h_1 \wedge h_2)^{d/2} u(h_1 \vee h_2)^{\alpha} (1+u(h_1)^5) \mathbb F(dh_1)\\
	& \le\one\{d(v_2, W_n)^2 \le z - h_2\}\Big[u(h_2)^{\alpha} \int_{-\infty}^{h_2} (z-h_1)^{d/2} (1+u(h_1)^5) \mathbb F(dh_1) \\
    & \qquad \qquad + (z-h_2)^{d/2} \int_{h_2}^z u(h_1)^{\alpha} (1+u(h_1)^5) \mathbb F(dh_1)\Big]\\
	& \le C\one\{d(v_2, W_n)^2 \le z - h_2\} \Big[(1+|h_2|^{\alpha d/2}) \int_{-\infty}^{h_2} (z-h_1)^{d/2} (1+|h_1|^{5d/2}) \mathbb F(dh_1)\\
    & \qquad \qquad + C(z-h_2)^{d/2}\int_{h_2}^z(1+|h_1|^{(5+\alpha)d/2}) \mathbb F(dh_1)\Big]\\
	&\le C  \one\{d(v_2, W_n)^2 \le z - h_2\} \Big[(1+|h_2|^{\alpha d/2}) +  (z-h_2)^{d/2}\Big],
\end{align*}
where the final step is due to our assumption that $\E\big[ H_-^{3d}\big]<\infty$. Now an integration over $v_2 \in W_{2n}$ dropping the indicator $\one\{d(v_2, W_n)^2 \le z - h_2\}$ yields
$
\Q f^{(3)}_\alpha \one\{v_2 \in W_{2n}\},\Q (f^{(3)}_\alpha)^2 \one\{v_2 \in W_{2n}\} \in O( |W_n|),
$
which is again due to the assumption $\E\big[ H_-^{d}\big]<\infty$.

\noindent Finally, let $v_2 \in W_{2n}^c$. Here, $\|v_2\| \lesssim d(v_2,W_n)$, and $\one\{d(v_2, W_n)^2 \le z - h_2\}$ forces $ \|v_2\| \le C \sqrt{z-h_2}$. 
 Thus, using the $f_\alpha^{(3)}(x_2)$-bound above and integrating w.r.t.\ $v_2$ first, noting $\E\big[ H_-^{d}\big]<\infty$, we obtain
\begin{align*}
	\Q f^{(3)}_\alpha \one\{v_2 \in W_{2n}^c\}
    &\le C \int_{-\infty}^z (z-h_2)^{d/2} \left[ (1+|h_2|^{\alpha d/2})  +  (z-h_2)^{d/2} \right] \mathbb F(dh_2) <\infty,
\end{align*}
while since $\E\big[ H_-^{3d/2}\big]<\infty$, we have
\begin{align*}
	\Q (f^{(3)}_\alpha )^2 \one\{v_2 \in W_{2n}^c\} \le C \int_{\R} (z-h_2)^{d/2} [ (1+|h_2|^{\alpha d})  +  (z-h_2)^{d}] \mathbb F(dh_2) <\infty.
\end{align*}
Combining the two bounds above yields for $\alpha \in (0,1)$ that
\begin{equation}\label{eq:f32}
    \mathbb{Q}f_{\alpha}^{(3)} , \mathbb{Q}(f_{\alpha}^{(3)} )^2 \in O(|W_n|).
\end{equation}
For the final terms in the bound in Theorem \ref{thm:KolBd}, note that 
$
\kappa(x) = \one\big\{x \in W_n^z \big\} \; \exp(-\Q(x)).
$
Hence, using $\E \big[H_-^{5d/2}\big]<\infty$, we have
\begin{equation}\label{eq:kappabd}
    \Q \kappa^\alpha S \le C |W_n| \int_{-\infty}^z (1+u(h)^5)e^{-\alpha \Q(h)} F (dh) \in O ( |W_n|).
\end{equation}
On the other hand, noting that $g \le f_\zeta^{(1)}$, one can bound
$\Q g^\alpha S$ for $\alpha \in (0,1)$ using \eqref{eq:f1bd}, similar to $\Q (f_\zeta^{(1)})^2$. In particular, one obtains from \eqref{eq:f1bd} and using $\E H_-^{3d}<\infty$ that
\begin{align}\label{eq:gbd}
\Q g^\alpha S &\le C \Q (f_\zeta^{(1)})^\alpha S \nonumber\\
&\le C |W_{2n}| +  C|W_n|^\alpha \int \one\{v_2 \in W_{2n}^c\} \one\{h_2 \le z-c\|v_2\|^2\} (1+ |h_2|^{5d/2}) \, \Q (dx_2)\nonumber\\
& \le C  |W_{2n}| +C  |W_n|^\alpha  \int_{-\infty}^z (1+ |h_2|^{5d/2}) \, F (dh_2)\int_n^\infty\one\{r \le \sqrt{(z-h_2)/c}\}  \, r^{d-1} dr \nonumber\\
& \le C  |W_{2n}| + C |W_n|^\alpha \int_{-\infty}^z (1+ |h_2|^{5d/2}) (1+|h_2|^{d/2})\, F (dh_2)\nonumber\\
& \le C |W_{2n}| + C|W_n|^\alpha \in O ( |W_n|),
\end{align}
where the penultimate step uses $\E \big[H_-^{3d}\big]<\infty$.
\qed

\subsection{Proofs of Lemmas \ref{lem:Ia} and \ref{lem:fmombdV}}\label{apx:univprfV}
Next we consider the volume weighted case.

\begin{proof}[Proof of Lemma \ref{lem:Ia}] In view of Remark \ref{rem:Qlb}, we split $I_\alpha$ as
\begin{align*}
	I_\alpha(x_2) &\le C \int_{\R^{d+1}}  (1+h_{1,-}^t) \one\left\{a\|v_1 - v_2\|^2+ bh_1+ch_2 \ge x_0\right\}\exp\big(-C \big(a\|v_1 - v_2\|^2+ bh_1+ch_2-x_0\big)^{d/2}\big) \mathbb{Q}(dx_1) \\
	& + \int_{\R^{d+1}} (1+h_{1,-}^t)\one\left\{0 \le a\|v_1 - v_2\|^2 \le x_0 - bh_1 - c h_2 \right\} \mathbb{Q}(dx_1) =: I_1 + I_2.
\end{align*}
	Then using $\E H_-^{t+d/2}<\infty$ and the change of variables $s := a\|v_1 - v_2\|^2$,
\begin{align*}
	I_1 & \le C \int_{-\infty}^\infty \int_{4x_0}^\infty  (1+h_{1,-}^t) (s-bh_1-ch_2 + x_0)_+^{\frac{d}{2} -1} \exp\big(-C s^{d/2}\big) \, ds \mathbb F(dh_1)\\
	& \le C \int_{-\infty}^\infty \int_{4x_0}^\infty  (1+h_{1,-}^t) (1+ |s|^{\frac{d}{2} -1} +h_{1,-}^{\frac{d}{2} -1} + h_{2,-}^{\frac{d}{2} -1}) \exp\big(-C s^{d/2}\big) \, ds \mathbb F(dh_1) \in  O(h_{2,-}^{d/2-1}),
\end{align*}
while
$
I_2 \le C \int_{-\infty}^\infty  (1+h_{1,-}^t) \big(x_0 - bh_1 - ch_2\big)_+^{d/2} \, \mathbb F(dh_1) \in O(h_{2,-}^{d/2}). 
$
Combining these bounds, we obtain
$I_\alpha(x_2)\in O(h_{2,-}^{d/2})$
for every $t\in[0,(m-1/2)d]$ and $\alpha>0$.
\end{proof}

\subsubsection{Proof of Lemma \ref{lem:fmombdV}}
    For bounding moments of $f_\alpha^{(3)}$, it will be important to maximize our choice of $p$. In the rest of this section we take $m=5.5+\varepsilon$ for some $\varepsilon>0$, that is, we assume $\E H_-^{(5.5+\varepsilon)d}<\infty$, as in the statement of Theorem \ref{thm:univ}. Consequently, we will take $p=12$ satisfying the above constraint.

\smallskip

\noindent \underline{\em Moment bounds on $f_\alpha^{(i)}$, $i=1,2$.}
It is straightforward to see that $S^{(z, \ms V)}(x_1) \lesssim 1+ h_{1,-}^{2d} + h_{1,-}^{5d/2} \in O (h_{1,-}^{5d/2})$.

By Lemma \ref{lem:Ia}, for any $\alpha>0$, using that $\E[H_-^{3d}]<\infty$ we have
\begin{align*}
	f_\alpha^{(2)}(x_2) & \lesssim \int_{\R^{d+1}} \one \{x_2 \in W_n^z\} (1+h_{1,-}^{5d/2}) \exp\big(-\frac{\alpha}{2^d} \Q\big(\frac{\|v_1 - v_2\|^2}{16}+ \frac{3h_1}{4}+\frac{h_2}{4}\big)\big) \mathbb{Q}(dx_1)\nonumber\\
	& \lesssim \one \{x_2 \in W_n^z\} (1+ h_{2,-}^{d/2}).
\end{align*}
Since $\E\big[H_-^{d}\big]<\infty$, upon integrating w.r.t.\ $(v_2,h_2)$, one obtains for $i \in \{1,2\}$ that
\begin{equation}\label{eq:f22V}
    \mathbb{Q}(f_{\alpha}^{(2)})^i \in O( |W_n|).
\end{equation}
Next, we bound the integral of $f_\alpha^{(1)}$. Again by Lemma \ref{lem:Ia}, we can bound
\begin{align}\label{eq:f1bd1V}
	&f_\alpha^{(1)}(x_2) \lesssim \int_{W_n^z} (1+h_{1,-}^{5d/2}) \exp\bigg(-\frac{\alpha}{2^d} \Q\bigg(\frac{\|v_1 - v_2\|^2}{16}+ \frac{3h_1}{4}+\frac{h_2}{4}\bigg)\bigg) \mathbb{Q}(dx_1)
	 \in O\big((1+ h_{2,-}^{d/2})\big).
\end{align}
\noindent \textbf{Case 1 ($v_2 \in W_{2n}$):} For $i \in \{1,2\}$, integrating over $W_{2n} \times \R$, we obtain 
$
\Q \big((f_\alpha^{(1)}(x_2))^i \one\{v_2 \in W_{2n}\}\big) \in O( |W_n|).
$

\noindent \textbf{Case 2 ($v_2 \in W_{2n}^c, h_2$ very negative):} Now let $v_2 \in W_{2n}^c$ so that $\|v_1 - v_2\| \ge c_1 \|v_2\|$ for some $c_1>0$. Fix $\eta=c_1^2/32$ and consider the region $h_{2} \le - \eta \|v_2\|^2$. Using the pointwise bound in \eqref{eq:f1bd1V} and our assumption $\mathbb{E}[H_-^{3d/2}] < \infty$ yields for $i \in \{1,2\}$
	\begin{align*}
		&\Q \big((f_\alpha^{(1)}(x_2))^i  \one\{v_2 \in W_{2n}^c, h_2\le -\eta \|v_2\|^2\} \big)
		\le C \int_{W_{2n}^c} dv_2 \int_{-\infty}^{-\eta \|v_2\|^2} \big(1 + h_{2,-}^d\big) \,\mathbb F(dh_2) \\
		&=  \int_{-\infty}^0 \mathbb F(dh_2) \big(1 + h_{2,-}^d\big) \int_{\{\|v_2\|^2 \le h_{2,-}/\eta\}} dv_2 \\
		& \le C\int_{-\infty}^0 \big(1 + h_{2,-}^d\big) h_{2,-}^{d/2} \mathbb F(dh_2) \in O(1).
	\end{align*}
	
	\noindent \textbf{Case 3(a) ($v_2 \in W_{2n}^c, h_{2} > -\eta \|v_2\|^2, h_{1} \le -\eta \|v_2\|^2$):} By our assumption $\E[H_-^{3.5 d}]<\infty$, 
    we drop the exponential factor to bound
	\begin{align*}
		f_\alpha^{(1)}(x_2)=&\int_{W_n^{-\eta \|v_2\|^2}} (1+h_{1,-}^{5d/2}) \exp\big(-\frac{\alpha}{2^d} \Q\big(\frac{\|v_1 - v_2\|^2}{16}+ \frac{3h_1}{4}+\frac{h_2}{4}\big)\big) \mathbb{Q}(dx_1)\\
		&\le \int_{W_n} dv_1 \int_{h_{1,-} \ge \eta \|v_2\|^2} \big(1 + h_{1,-}^{5d/2}\big) \mathbb F(dh_1) \\
		&\le \int_{W_n} dv_1 \frac{1}{(\eta \|v_2\|^2)^{d}} \int_{-\infty}^0 h_{1,-}^{d} \big(1 + h_{1,-}^{5d/2}\big) \mathbb F(dh_1) \lesssim |W_n| \|v_2\|^{-2d}.
	\end{align*}
	
	\noindent\textbf{Case 3(b) ($v_2 \in W_{2n}^c, h_{1}\wedge h_{2} >-\eta \|v_2\|^2$):} When neither mark $h_1,h_2$ is extreme, we have
	\begin{align*}
		u := \frac{\|v_1-v_2\|^2}{16} + \frac{3h_1}{4} + \frac{h_2}{4} &\ge \frac{c_1^2 \|v_2\|^2}{16} - \frac{3}{4}h_{1,-} - \frac{1}{4}h_{2,-} \ge \|v_2\|^2 \big( \frac{c_1^2}{16} - \eta \big) = \frac{c_1^2}{32} \|v_2\|^2.
	\end{align*}
	Since $v_2 \in W_{2n}^c$, for $n$ large enough, $u \ge x_0$ so that $\mathbb{Q}(u) \ge C u^{d/2} \ge C'' \|v_2\|^d$. Thus, 
	\begin{align*}
	f_\alpha^{(1)}(x_2)=&\int_{\R^{d+1}} \one \{v_1 \in W_n, h_{1}\wedge h_{2}  \ge - \eta \|v_2\|^2\} (1+h_{1,-}^{5d/2}) \exp\big(-\frac{\alpha}{2^d} \Q\big(\frac{\|v_1 - v_2\|^2}{16}+ \frac{3h_1}{4}+\frac{h_2}{4}\big)\big) \mathbb{Q}(dx_1)\\ 
	&\le C \int_{W_n} dv_1 \int_{-\infty}^z \mathbb F(dh_1) \big(1 + h_{1,-}^{5d/2}\big) \exp\big(-(\alpha C''/2) \|v_2\|^d\big) \le C |W_n| \exp\big(-(\alpha C''/2) \|v_2\|^d\big).
	\end{align*}
Combining Cases 3(a) and 3(b) and integrating over the region $v_2 \in W_{2n}^c, h_{2} \ge - \eta \|v_2\|^2$ yields
	\begin{align*}
			&\Q \big((f_\alpha^{(1)}(x_2))^i  \one\{v_2 \in W_{2n}^c, h_{2,-} < \eta \|v_2\|^2\} \big)\\
			& \le C \int_{W_{2n}^c} \big(|W_n| \|v_2\|^{-2d} \big)^i dv_2 \int_{\mathbb{R}} \mathbb F(dh_2) +  \int_{W_{2n}^c} \big(|W_n| \exp(-(\alpha C''/2) \|v_2\|^d) \big)^i dv_2 \int_{\mathbb{R}} \mathbb F(dh_2)\\
			&= |W_n|^i \int_{\|v_2\| \ge c_1 n} \|v_2\|^{-2id} dv_2 + |W_n|^i \int_{\|v_2\| \ge c_1 n} \exp(-i(\alpha C''/2) \|v_2\|^d) dv_2 \in O(|W_n|).
	\end{align*}
	Combining this with Cases 1 and 2 now yields  
    \begin{equation}\label{eq:f12V}
        \Q (f_\alpha^{(1)})^i \in O( |W_n|)
	\qquad \text{for $i \in \{1,2\}.$}
    \end{equation}

\noindent \underline{\em Moment bounds for $g^{(z,\ms V)}$ and $\kappa$ terms.}
We can bound $\Q (g^{(z, \ms V)})^{2\tau} S^{(z, \ms V)} \equiv \Q g^{2\tau} S$ in a similar manner as for $f_\alpha^{(1)}$. Note that with $p\ge 4$, we have $2\tau=p/(16+2p) \ge 1/6$. Since $g \lesssim f_\zeta^{(1)}$, the bounds in Cases 1, 2 and 3(b) go through exactly the same way, even with the additional integrand $S \in O(h_-^{5d/2})$. In Case 3(a), considering $h_{2,-} < \eta \|v_2\|^2$, by our assumption $\E H_-^{3.5 d}<\infty$, we have
\begin{align*}
	&\int_{W_n^{-\eta \|v_2\|^2}} \exp\big(-\frac{\alpha}{2^d} \Q\big(\frac{\|v_1 - v_2\|^2}{16}+ \frac{3h_1}{4}+\frac{h_2}{4}\big)\big) \mathbb{Q}(dx_1)\\
	&\le \Q(W_n^{-\eta \|v_2\|^2})  \le \int_{W_n} dv_1 \frac{1}{(\eta \|v_2\|^2)^{3.5d}} \int_{-\infty}^0 h_{1,-}^{3.5 d} \mathbb F(dh_1) \lesssim |W_n| \|v_2\|^{-7d}.
\end{align*}
Arguing as for $f_\alpha^{(1)}$, combining the above bound and the Case 3(b) yields
\begin{align}\label{eq:gbdV}
	&\Q (g^{(z,\ms V)})^{2\tau}  S^{(z,\ms V)} \,  \one\{v_2 \in W_{2n}^c, h_{2} >- \eta \|v_2\|^2\}\nonumber\\
	& \lesssim \int_{W_{2n}^c} \big(|W_n| \|v_2\|^{-7d} \big)^{1/6} dv_2 \int_{\mathbb{R}} (1+ h_{2,-}^{5d/2})\, \mathbb F(dh_2) +  \int_{W_{2n}^c} \big(|W_n| \exp(-(\alpha C''/2) \|v_2\|^d) \big)^{1/6} dv_2 \int_{\mathbb{R}} \mathbb F(dh_2)\nonumber\\
	&= |W_n|^{1/6} \int_{\|v_2\| \ge c_1 n} \|v_2\|^{-(7/6)d} dv_2 + |W_n|^i \int_{\|v_2\| \ge c_1 n} \exp(-i(\alpha C''/2) \|v_2\|^d) dv_2 \in O(|W_n|).
\end{align}
Combining this with the bounds in Cases 1 and 2 now yields
$
    \Q( (g^{(z,\ms V)})^{2\tau} S)\in O( |W_n|).
	$

Also note that  in the volume weighted case, the function $\kappa^{(z,\ms V)}$ is exactly the same as the function $\kappa$ in the proof of Theorem \ref{thm:univ} for the unweighted case. Thus, the corresponding bound in Lemma \ref{lem:fmombd} also holds here, yielding
$\Q((\kappa^{(z,\ms V)})^{2\tau} S )\in O( |W_n|)$.

\smallskip

\noindent \underline{\em Moment bounds on $f_\alpha^{(3)}$.} We start by obtaining an appropriate upper bound on the function $q(x_1,x_2)$ defined at \eqref{eq:g2s}. {Without loss of generality, we now assume that $x_0 \ge 0$ with $x_0$ as in Remark \ref{rem:Qlb}.} For $x_1, x_2 \in \mathbb{R}^{d+1}$, define
	\begin{equation}\label{def:A}
		A(x_1, x_2) := \frac{1}{256}\|v_1-v_2\|^2 
		+ \frac{1}{128}d(v_2, W_n)^2 + \frac{h_1}{8} + \frac{h_2}{8},
	\end{equation}
where we recall $d(v, W_n) := \inf_{w \in W_n} \|v - w\|$. 

\begin{lemma}
	\label{lem:qbound}
	Assume $\E[H_-^{md}]<\infty$.	There exists a constant $c > 0$ such that $q(x_1, x_2)$ satisfies
	\begin{equation*}
		q(x_1, x_2) \lesssim \big( 1 + h_{1,-}^{d/2} + h_{2,-}^{d/2}\big)  \one\{A \le 2x_0\}  + [(1+A)^{-(m-1/2)d} + \exp(-c A^{d/2})] \one\{A >0\}.
	\end{equation*}
\end{lemma}

\begin{proof}
	Recall that $q(x_1, x_2) = \int_{W_n^z} \mathbb{P}(\{x_1, x_2\} \subseteq D^{\ms V}(x_3, T_3)) \Q(d x_3)$, where we write $T_3 \equiv T(x_3, \mathcal{P}+ \delta_{x_3})$ for notational simplicity. As noted in \eqref{eq:tailbd}, for $x_1$ and $x_2$ to be covered, the cover time $T_3$ must satisfy
	\begin{equation*}
		T_3 \ge \frac{1}{8}\|v_1-v_3\|^2 + ({h_1 + h_3})/{2} \quad \text{and} \quad T_3 \ge \frac{1}{8}\|v_2-v_3\|^2 + (h_2 + h_3)/{2}.
	\end{equation*}
	Since $v_3 \in W_n$, we have $\|v_i-v_3\|^2 \ge d(v_i, W_n)^2$ for $i \in \{1, 2\}$. Decompose for $i \in \{1,2\}$,
	\begin{align*}
		\frac{1}{8}\|v_i-v_3\|^2 &\ge \frac{1}{16}\|v_i-v_3\|^2 + \frac{1}{32}d(v_i, W_n)^2 + \frac{1}{32}\|v_i-v_3\|^2.
	\end{align*}
	Substituting these into the bounds for $T_3$ and taking their arithmetic mean yields
	\begin{equation*}
		T_3 \ge \frac{1}{32}\|v_1-v_3\|^2 + \frac{1}{32}\|v_2-v_3\|^2 + \frac{1}{64}d(v_1, W_n)^2 + \frac{1}{64}d(v_2, W_n)^2 + \frac{1}{64}\|v_1-v_3\|^2 + \frac{1}{64}\|v_2-v_3\|^2 + \frac{h_1+h_2+2h_3}{4}.
	\end{equation*}
	Applying the geometric triangle inequality $\frac{1}{64}\|v_1-v_3\|^2 + \frac{1}{64}\|v_2-v_3\|^2 \ge \frac{1}{128}\|v_1-v_2\|^2$, and dropping $\frac{1}{32}\|v_2-v_3\|^2 + \frac{1}{64}d(v_1, W_n)^2\ge 0$, we obtain
	\begin{equation*}
		T_3 \ge \frac{1}{32}\|v_1-v_3\|^2 + \underbrace{ \frac{1}{128}\|v_1-v_2\|^2 
			+ \frac{1}{64}d(v_2, W_n)^2 + \frac{h_1+h_2}{4} }_{2 A(x_1, x_2)} + \frac{h_3}{2}.
	\end{equation*}
	Since the event $\{\{x_1, x_2\} \subseteq D^{\ms V}(x_3, T_3)\}$ also forces $T_3 >h_3$, using the exponential tail bound in Proposition \ref{prop:CoverTimetail}(iii), we obtain
	\begin{equation*}
		q(x_1, x_2) \le \int_{\mathbb{R}^{d, z}} \exp\big(-\frac{1}{2^d}\mathbb{Q}\big( \frac{1}{64}\|v_1-v_3\|^2 + A(x_1, x_2) + \frac{3h_3}{4} \big)\big) \Q(dx_3).
	\end{equation*}
	Further denote $K = A(x_1, x_2) + {3h_3}/{4}$. We first evaluate the inner spatial integral over $v_3 \in \mathbb{R}^d$:
	\begin{equation*}
		J(K) := \int_{\mathbb{R}^d} \exp\big(-\frac{1}{2^d}\mathbb{Q}\big(\frac{1}{64}\|v_1-v_3\|^2 + K\big)\big) \,dv_3.
	\end{equation*}
	Applying the substitution $y = (v_1-v_3)/\sqrt{32}$ (which implies $dv_3 \asymp dy$), we partition the integration domain into regions $U_1 = \{ \|y\|^2 + K \le x_0 \}$ and $U_2 = \{ \|y\|^2 + K > x_0 \}$. Integrating over $U_1$ yields
	\begin{equation*}
		\int_{U_1} \exp\big(-\frac{1}{2}\mathbb{Q}(\|y\|^2 + K)\big) \,dy \le \int_{\|y\|^2 \le (x_0 - K)_+} 1 \, dy = \kappa_d (x_0 - K)_+^{d/2}.
	\end{equation*}
	On $U_2$, the argument exceeds $x_0$, so that the polynomial lower bound $\mathbb{Q}(u) \ge \tilde{c} u^{d/2}$ applies by Remark \ref{rem:Qlb}.
	\begin{itemize}
		\item \textbf{Case 1 ($K \ge 0$):} In this case, we have $K = K_+$ and thus
		\begin{equation*}
			\int_{U_2} \exp\big(-\frac{\tilde{c}}{2}(\|y\|^2 + K)^{d/2}\big) \,dy \le \exp(-c K_+^{d/2}) \int_{\mathbb{R}^d} \exp(-c \|y\|^d) dy \lesssim \exp(-c K_+^{d/2}).
		\end{equation*}
		\item \textbf{Case 2 ($K < 0$):} Here we apply the substitution $u = \|y\|^2 + K$. The domain is $u > x_0$, and the radial measure transforms to $dy \propto (u + |K|)^{(d-2)/2} du$. Since $(u+|K|)^{(d-2)/2} \lesssim (u^{(d-2)/2} + |K|^{(d-2)/2})$, integrating against $\exp(-\frac{\tilde{c}}{2}u^{d/2})$ yields
		\(
			\int_{U_2} \exp\big(-\frac{\tilde{c}}{2}u^{d/2}\big) \, dy \lesssim 1+ |K|^{(d-2)/2} \lesssim (x_0 - K)_+^{d/2}.
		\)
	\end{itemize}
	Summing the bounds yields 
	\begin{equation}\label{eq:J}
		J(K) \le C \big( (x_0 - K)_+^{d/2} + \exp(-c K_+^{d/2}) \big).
	\end{equation}
	
	Finally, we integrate $J(K) =J(A(x_1, x_2) + \frac{3h_3}{4})$ over $h_3 \in (-\infty, z]$ with respect to $h_3$ to obtain a bound on $q(x_1,x_2)$. Writing $A\equiv A(x_1,x_2)$, for the first summand in \eqref{eq:J}, we must evaluate the integral
	\begin{equation*}
		I_{cut}(A) := \int_{-\infty}^z \big(x_0 - A - \frac{3h_3}{4}\big)_+^{d/2} \,\mathbb F(dh_3) \le \int_{-\infty}^z \big(x_0 - A + \frac{3}{4}h_{3,-}\big)_+^{d/2} \,\mathbb F(dh_3).
	\end{equation*}
	Note that the integrand is zero unless $\frac{3}{4}h_{3,-} > A - x_0$.
	
	\begin{itemize}
			\item \textbf{Case 1 ($A \le 2x_0$):} 
			Noting $-A \le \frac{1}{8}h_{1,-} + \frac{1}{8}h_{2,-}$ and using subadditivity, we bound the integral
			\begin{align*}
				I_{cut}(A) \one\{A \le 2x_0\} &\le \one\{A \le 2x_0\} \int_{-\infty}^z \big( x_0 + \frac{1}{8}h_{1,-} + \frac{1}{8}h_{2,-} + \frac{3}{4}h_{3,-} \big)^{d/2} \,\mathbb F(dh_3)\\
				&\le C \one\{A \le 2x_0\} \big( 1 + h_{1,-}^{d/2} + h_{2,-}^{d/2}\big).
			\end{align*}
		
		\item \textbf{Case 2 ($A > 2x_0$):} 
		Since $A > 2x_0$, we have $A - x_0 > {A}/{2}$. Thus, the integrand in $I_{cut}(A)$ is non-zero only if $\frac{3}{4}h_{3,-} > {A}/{2}$. On this restricted domain, dropping $x_0 - A \le 0$ term, and integrating yields
		\begin{equation*}
			I_{cut}(A) \one\{A > 2x_0\} \le \int_{-\infty}^{-\frac{2A}{3}} \big( \frac{3}{4}h_{3,-} \big)^{d/2} \mathbb F(dh_3) \le C \min\{\E H_-^{d/2},A^{-(m-1/2) d} \E H_-^{m d}\} \in O((1+A)^{-(m-1/2) d}).
		\end{equation*}
	\end{itemize}
	Thus,
	\begin{equation*}
		I_{cut}(A) \le C \big( 1 + h_{1,-}^{d/2} + h_{2,-}^{d/2}\big)\one\{A \le 2x_0\} + C(1+A)^{-(m-1/2) d} \one\{A > 2x_0\}.
	\end{equation*}
	
	Finally, we consider the exponential tail term in \eqref{eq:J}. If $A \le 0$, we  trivially bound the exponential by $1$. If $A > 0$, we partition the integration domain for $h_3$:
	\begin{itemize}
		\item \textbf{Case 1 $\big( 3h_{3,-} \le 2{A} \big)$:} Here, $A + \frac{3h_3}{4} \ge {A}/{2}$. The exponential is thus bounded by $\exp(-c (A/2)^{d/2})$. Integrating over $h_3$ thus yields a bound $\exp(-c' A^{d/2})$ for some $c'>0$.
		\item \textbf{Case 2 $\big( {3}h_{3,-} > 2{A} \big)$:} In this case, we bound the exponential term by $1$. By Markov's inequality and our assumed $md$ finite moments for $H_-$,
		\begin{equation*}
			\mathbb{P}\big( H_- > {2A}/{3} \big) \lesssim \min\{1,A^{-m d} \mathbb{E}[H_-^{m d}]\} \lesssim (1+A)^{-md}.
		\end{equation*}
	\end{itemize}
	
	Summing the different bounds and noting $x_0 \ge 0$, we obtain
	\begin{align*}
		q(x_1, x_2) &\le C \big( 1 + h_{1,-}^{d/2} + h_{2,-}^{d/2}\big)\one\{A \le 2x_0\} + (1+A)^{-(m-1/2)d} \one\{A > 2x_0\} \\
		& \qquad \qquad + \one\{A \le 0\} + [\exp(-c' A^{d/2}) + (1+A)^{-md}]\one\{A >0\} \\
		& \le C  \big( 1 + h_{1,-}^{d/2} + h_{2,-}^{d/2}\big)\one\{A \le 2x_0\}  + [(1+A)^{-(m-1/2) d} + \exp(-c' A^{d/2})] \one\{A >0\},
	\end{align*}
	completing the proof.
\end{proof}

\noindent \underline{\em Proof of \eqref{eq:f3bds}}
Recall that
$
	f_\alpha^{(3)}(x_2) \lesssim \int_{\mathbb{R}^d \times \R}S(x_1) q(x_1, x_2)^\alpha \, dv_1 \mathbb F(dh_1),
	$
where $S(x_1) \in O( (h_{1,-}^{5d/2}))$. For the application of Theorem \ref{thm:KolBd}, it is enough to prove
\[
\Q\big((f_\tau^{(3)})^2\big)\in O(|W_n|)
\qquad\text{and}\qquad
\Q f_{2\tau}^{(3)}\in O(|W_n|),
\]
where
$\tau=3/20$. For $x_i=(v_i,h_i)$, $i=1,2$, put
\(
B:=c_0\|v_1-v_2\|^2+D,
\)
\(c_0:=\frac1{256},
\) and \(
D:=\frac1{128}d(v_2,W_n)^2,
\)
and
\(
H:=\frac18h_{1,-}+\frac18h_{2,-}.
\)
By the definition \eqref{def:A},
\(
A(x_1,x_2)
=
B+\frac{h_1+h_2}{8}
\ge B-H.
\)

For $\alpha\in[\tau,2\tau]$, by Lemma \ref{lem:qbound} and using subadditivity of $x\mapsto x^\alpha$, we can bound
\begin{align}\label{eq:qbd}
	q(x_1, x_2)^\alpha &\le \left[ \big(1 + h_{1,-}^{\frac{d\alpha}{2}} + h_{2,-}^{\frac{d\alpha}{2}}\big)\one\{A \le 2x_0\}  + (1+A)^{-\lambda_\alpha} \one\{A > 0\} + \exp(-c\alpha A^{d/2}) \one\{A > 0\} \right],
\end{align}
where we write $\lambda_\alpha:=(m-\tfrac12)d\alpha$. Since $m>11/2$ and $\alpha\ge\tau=3/20$, note that
\begin{equation}\label{eq:lambda-lower-V}
	\lambda_\alpha
	\ge (m-\tfrac12)d\tau
	>5d\frac{3}{20}
	=\frac{3d}{4} >\frac{d}{2}.
\end{equation}

We now bound the contributions in \eqref{eq:qbd} separately. For the first term, note that since $A\ge B-H$, we have $A\le2x_0 \implies
B\le H+2x_0$.
Since $x_0\ge0$ is fixed, this implies
\(
\one\{A\le2x_0\}
\lesssim
{(1+H)^{\lambda_\alpha}}/
{(1+B)^{\lambda_\alpha}}.
\)
Consequently,
\begin{align}\label{eq:q-small-A-V}
	\big(
	1+h_{1,-}^{d\alpha/2}
	+h_{2,-}^{d\alpha/2}
	\big)\one\{A\le2x_0\}
	&\lesssim
	\frac{
		\big(
		1+h_{1,-}^{d\alpha/2}
		+h_{2,-}^{d\alpha/2}
		\big)
		(1+H)^{\lambda_\alpha}}
	{(1+B)^{\lambda_\alpha}} \nonumber\\
	&\lesssim
	\frac{
		1+h_{1,-}^{\lambda_\alpha+d\alpha/2}
		+h_{2,-}^{\lambda_\alpha+d\alpha/2}}
	{(1+B)^{\lambda_\alpha}} = \frac{
		1+h_{1,-}^{md\alpha}
		+h_{2,-}^{md\alpha}}
	{(1+B)^{\lambda_\alpha}}.
\end{align}

Next, on $\{A>0\}$, using $A\ge B-H$ gives
\(
1+A
\ge
1+(B-H)_+.
\)
Note the elementary inequality
\(
	\big(1+(t-s)_+\big)^{-\lambda}
	\le
	(1+s)^\lambda(1+t)^{-\lambda}\) for \(
	s,t\ge0, \lambda>0.
	\)
Indeed,
\(
1+t
\le
(1+s)\big(1+(t-s)_+\big),
\)
which immediately yields the above inequality.
Using this, we obtain
\begin{align}
	(1+A)^{-\lambda_\alpha}\one\{A>0\}
	&\le
	\big(1+(B-H)_+\big)^{-\lambda_\alpha}
\le
	\frac{(1+H)^{\lambda_\alpha}}
	{(1+B)^{\lambda_\alpha}} \lesssim
	\frac{
		1+h_{1,-}^{md\alpha}
		+h_{2,-}^{md\alpha}}
	{(1+B)^{\lambda_\alpha}}.
	\label{eq:q-poly-A-V}
\end{align}

Finally, since $\alpha\in[\tau,2\tau]$ is bounded away from zero and
$\lambda_\alpha$ ranges over a bounded interval, there exists
$C<\infty$, uniformly in $\alpha\in[\tau,2\tau]$, such that
\(
e^{-c\alpha u^{d/2}} 
\le
C(1+u)^{-\lambda_\alpha},
\)
for  \(u\ge0\).
Together with \eqref{eq:q-poly-A-V}, this yields
\begin{equation}\label{eq:q-exp-A-V}
	e^{-c\alpha A^{d/2}}\one\{A>0\} \lesssim (1+A)^{-\lambda_\alpha}\one\{A>0\} 
	\lesssim
	\frac{
		1+h_{1,-}^{md\alpha}
		+h_{2,-}^{md\alpha}}
	{(1+B)^{\lambda_\alpha}}.
\end{equation}

Combining \eqref{eq:q-small-A-V}, \eqref{eq:q-poly-A-V} and
\eqref{eq:q-exp-A-V} via \eqref{eq:qbd}, we arrive at the uniform bound
\begin{equation}\label{eq:qalphaV}
	q(x_1,x_2)^\alpha
	\lesssim
	\frac{
		1+h_{1,-}^{md\alpha}
		+h_{2,-}^{md\alpha}}
	{\big(
		1+c_0\|v_1-v_2\|^2+D
		\big)^{\lambda_\alpha}},
	\qquad
	\alpha\in[\tau,2\tau].
\end{equation}

We next integrate with respect to $v_1$. The change of variables
\(
v_1-v_2=\sqrt{1+D}\,y
\)
gives
\begin{align}
	&\int_{\R^d}
	\big(
	1+c_0\|v_1-v_2\|^2+D
	\big)^{-\lambda_\alpha}
	\,dv_1
	=
	(1+D)^{d/2-\lambda_\alpha}
	\int_{\R^d}
	\big(1+c_0\|y\|^2\big)^{-\lambda_\alpha}
	\,dy \lesssim
	(1+D)^{-(\lambda_\alpha-d/2)},
	\label{eq:spatial-qalpha-V}
\end{align}
where the final step follows upon noting that $\lambda_\alpha>d/2$ by \eqref{eq:lambda-lower-V}. Using \eqref{eq:qalphaV}, \eqref{eq:spatial-qalpha-V} and
$S(x_1)\lesssim1+h_{1,-}^{5d/2}$, we therefore obtain
\begin{align*}
	f_\alpha^{(3)}(x_2)
	&\lesssim
	\frac{1}{(1+D)^{\lambda_\alpha-d/2}}
	\int_{\R}
	(1+h_{1,-}^{5d/2})
	\big(
	1+h_{1,-}^{md\alpha}
	+h_{2,-}^{md\alpha}
	\big)
	\mathbb F(dh_1).
\end{align*}
Since
\(
\frac52+m\alpha
\le
\frac52+\frac{3m}{10}
<m,
\)
for $\alpha\in[\tau,2\tau]$, with the final inequality following from $m>11/2$, the assumed condition $\E H_-^{md}<\infty$ thus implies
\begin{equation}\label{eq:f3-pointwise-V}
	f_\alpha^{(3)}(x_2)
	\lesssim
	\frac{1+h_{2,-}^{md\alpha}}
	{(1+D)^{\lambda_\alpha-d/2}},
	\qquad
	\alpha\in[\tau,2\tau].
\end{equation}

Now let $\alpha=2\tau=3/10$. 
From \eqref{eq:f3-pointwise-V}, we have
\(
f_{2\tau}^{(3)}(x_2)
\lesssim
{1+h_{2,-}^{2md\tau}}/
{(1+D)^{\lambda_{2\tau}-d/2}}.
\)
For $v_2\in W_{2n}^c$, we have
\(
d(v_2,W_n)\ge c_W\|v_2\|
\)
for some $c_W>0$, and hence
\(
D\gtrsim\|v_2\|^2.
\)
Noting
\(
2m\tau=\frac{3m}{10}<m,
\)
so the $h_2$-integral is finite, we obtain
\begin{align*}
\Q f_{2\tau}^{(3)} &\lesssim \int_{\R}
(1+h_{2,-}^{2md\tau} )
\,\mathbb{F}(dh_2) \left(|W_{2n}|	+ \int_{W_{2n}^c}
	(1+D)^{-(\lambda_{2\tau}-d/2)}
	\,dv_2 \right)
	&\lesssim |W_n| +
	\int_{\R^d}
	(1+\|v\|^2)^{-(\lambda_{2\tau}-d/2)}
	\,dv.
\end{align*}
The integral above is finite since noting that $m>11/2$, we have
\(
\lambda_{2\tau}
=
(m-\tfrac12)d\frac{3}{10}
>
\frac{3d}{2}
>d.
\)
Therefore,
\(
\Q f_{2\tau}^{(3)}\in O(|W_n|)
\)
proving the first bound in \eqref{eq:f3bds}.

Finally for $\alpha=\tau=3/20$, we have
\(
\lambda_\tau
=
(m-\tfrac12)d\frac{3}{20}
>
\frac{3d}{4},
\)
because $m>11/2$. Squaring \eqref{eq:f3-pointwise-V} gives
\(
\big(f_\tau^{(3)}(x_2)\big)^2
\lesssim
{1+h_{2,-}^{2md\tau}}/
{(1+D)^{2(\lambda_\tau-d/2)}}.
\)
Since
\(
2m\tau
=
\frac{3m}{10}
<m\), arguing as above yields
\begin{align*}
	\Q\big((f_\tau^{(3)})^2\big) &\lesssim \int_{\R}
	(1+h_{2,-}^{2md\tau} )
	\,\mathbb{F}(dh_2) \left(|W_{2n}|	+ \int_{W_{2n}^c}
	(1+D)^{-2(\lambda_\tau-d/2)}
	\,dv_2 \right)\\
	&\lesssim |W_n| +
\int_{\R^d}
(1+\|v\|^2)^{-2(\lambda_\tau-d/2)}
\,dv.
\end{align*}
The latter integral is again finite since
\(
2(\lambda_\tau-d/2)>\frac d2,
\)
or equivalently
\(
\lambda_\tau>\frac{3d}{4},
\)
which was verified above. Hence,
\(
\Q\big((f_\tau^{(3)})^2\big)\in O(|W_n|).
\)
This yields the second bound in \eqref{eq:f3bds}, concluding the proof of Lemma \ref{lem:fmombdV}.

\subsection{Proof of Lemma \ref{lem:zetabds}}\label{apx:mult}
Finally, we provide a proof for Lemma \ref{lem:zetabds} bounding the terms $\zeta_1, \zeta_2^{(2)}, \zeta_3^{(2)}$ and $\zeta_4^{(2)}$ for both our models in the multivariate case. For this, we will use the following generalized version of \cite[Lemma 5.5]{BM22} for $p>0$.

\begin{lemma}[Extension of \cite{BM22}, Lemma 5.5]
\label{lem:BM55ext}
For $H$ as in Theorem \ref{thm:KolBd}, assume that conditions {\rm (A1)} and {\rm (A2)} hold. Set
\(
    r_p:=4+p/2
\) and \(
    k_p:=4+\lceil p/2\rceil
\).
Let $g$ and $h$ be defined as in \eqref{eq:g}. Then there exists a constant $C_p\in(0,\infty)$, depending
only on $p$, such that, for every $y\in\R^{d+1}$ and every
$\mu\in\mathbf N$ with $\mu(\R^{d+1})\le1$,
\begin{equation}
\label{eq:BM55ext}
    \E\big|D_yH(\PP+\mu)\big|^{r_p}
    \le
    C_p\left[
        M_p(y)^{r_p}
        +
        h(y)\big(1+g(y)^{k_p-1}\big)
    \right].
\end{equation}
In particular, with $S$ defined in \eqref{eq:g}, there exists a constant $C_p'\in(0,\infty)$, depending
only on $p$ such that
\begin{equation}
\label{eq:BM55extS}
    \E\big|D_yH(\PP+\mu)\big|^{r_p}
    \le C_p' S(y)^{r_p/4}.
\end{equation}
\end{lemma}

\begin{proof}
We only indicate the modifications to the proof of
\cite[Lemma 5.5]{BM22}. We first note that the analogue of \cite[Lemma 5.4]{BM22}
continues to hold. Now following the proof of \cite[Lemma 5.5]{BM22}, for $\mu=0$, Lemma 5.2 therein gives
\(
    D_yH(\PP)
    =
    \xi(y;\PP+\delta_y)
    +
    \sum_{x\in\PP}D_y\xi(x;\PP).
\)
By (A2) and monotonicity of $L^q$-norms,
\(
    \E\big|\xi(y;\PP+\delta_y)\big|^{r_p}
    \le M_p(y)^{r_p}
\). For the second term, arguing as in \cite[Lemma 5.5]{BM22}, one obtains 
\(
    \E\big|
        \sum_{x\in\PP}D_y\xi(x;\PP)
    \big|^{r_p}
    \le
    C_1\sum_{i=1}^{k_p} I_i,
\)
for some constant $C_1 \in (0,\infty)$ depending only on $p$, where
\[
    I_i
    :=
    \E
    \sum_{(x_1,\ldots,x_i)\in\PP_{\neq}^i}
    \one\left\{
        D_y\xi(x_j;\PP)\neq0,\ j=1,\ldots,i
    \right\}
    |D_y\xi(x_1;\PP)|^{r_p}.
\]
The only difference from \cite[Lemma 5.5]{BM22} is that the
expansion now contains terms indexed by $i=1,\ldots,k_p$. In the
H\"older step of the proof therein, with $\varepsilon=p/2$, the
probability exponent is
\[
    \frac{p-\varepsilon}{4i+pi}
    =
    \frac{p}{2i(4+p)}
    \ge
    \frac{p}{2k_p(4+p)}
    =
    \zeta .
\]
Hence, the same argument gives
\(
    I_i\le C_2 g(y)^{i-1}h(y),
    \;i\le k_p,
\)
and therefore
\(
    \sum_{i=1}^{k_p}I_i
    \le
    C_3 h(y)\bigl(1+g(y)^{k_p-1}\bigr)
\)
for some constants $C_2, C_3 \in (0,\infty)$ depending only on $p$.
This proves \eqref{eq:BM55ext} for $\mu=0$. If $\mu(\R^{d+1})=1$, the same argument applies with the
additional deterministic point inserted throughout.

It remains to obtain \eqref{eq:BM55extS}. Noting
\(
    k_p-1<r_p,
\)
by the definition of $S$ in \eqref{eq:g}, we have 
\[
    h(y)\big(1+g(y)^{k_p-1}\big)
    \le
    C_4
    \left[
        h(y)^{4/r_p}(1+g(y)^4)
    \right]^{r_p/4} \le C_4 S(y)^{r_p/4}
\]
for some constant $C_4 \in (0,\infty)$ depending only on $p$. Since
\(
    M_p(y)^4\le S(y),
\)
\eqref{eq:BM55extS} now follows from \eqref{eq:BM55ext}.
\end{proof}

\begin{proof}[Proof of Lemma \ref{lem:zetabds}]
    We first estimate \(\zeta_2^{(2)}\) and \(\zeta_3^{(2)}\). Fix \(i \in [m]\). We begin with the fourth moment of the iterated add-one cost operator. Write
\[
Y_{i,x,y} := D^{(2)}_{x,y}\hat G_{n,*}(z_i).
\]
Since \(4 < 4+p/2\), the elementary inequality \(\|Y\|_4 \le \|Y\|_{4+p/2} \P(Y \neq 0)^{1/4-1/(4+p/2)}\) yields
\begin{equation}\label{eq:Y_norm_step}
\|Y_{i,x,y}\|_4
\le
\|Y_{i,x,y}\|_{4+p/2}\,
\P\big(Y_{i,x,y} \neq 0\big)^{p/(32+4p)}.
\end{equation}
Moreover,
\(
D^{(2)}_{x,y}F(\mathcal P)
=
D_xF(\mathcal P+\delta_y)-D_xF(\mathcal P),
\)
and hence \(|D^{(2)}_{x,y}F(\mathcal P)| \le |D_xF(\mathcal P+\delta_y)| + |D_xF(\mathcal P)|\). Therefore,
\[
\|Y_{i,x,y}\|_{4+p/2}
\le
2 \max_{u \in \{0,y\}} \big\|D_x\hat G_{n,*}(z_i)(\mathcal P+\delta_u)\big\|_{4+p/2}
\le
2 \max_{u \in \R^{d+1}} \big\|D_x\hat G_{n,*}(z_i)(\mathcal P+\delta_u)\big\|_{4+p/2}.
\]
Combining this with \eqref{eq:Y_norm_step}, we obtain
\begin{equation}\label{eq:D2_basic_bound}
\|D^{(2)}_{x,y}\hat G_{n,*}(z_i)\|_4
\le C
\max_{u \in \R^{d+1}}
\big\|D_x\hat G_{n,*}(z_i)(\mathcal P+\delta_u)\big\|_{4+p/2}
\,
\P\big(D^{(2)}_{x,y}\hat G_{n,*}(z_i)\neq 0\big)^\tau,
\end{equation}
where \(\tau := p/(32+4p)\).
Arguing as in \cite[Lemma 5.6]{BM22}, one can bound
$$
\P \left(D_{x,y}^{(2)} \hat G_{n, *}(z_i) \neq 0\right) \lesssim e^{-r^{(z_i)}(x,y)} + e^{-r^{(z_i)}(y,x)} + q^{(z_i)}(x,y).
$$
On the other hand, by Lemma \ref{lem:BM55ext}, for $i \in [m]$ and $x=(v,h)$
\begin{align*}
&|W_n|^{\frac{4+p/2}{2}} \max_{u \in \R^{d+1}} \E \left|D_x \hat G_{n, *}(z_i) (\mathcal{P} + \delta_u)\right|^{4+p/2} \lesssim 
\left(S^{(z_i)}(x)\right)^{\frac{4+p/2}{4}}.
\end{align*}
Combining the bounds then yields
$$
\left(\mathbb{E}\left(D_{x,y}^{(2)} \hat G_{n, *}(z_i)\right)^4\right)^{\frac{1}{4}}  \lesssim \left(S^{(z_i)}\right)^{1/4}\left(e^{-r^{(z_i)}(x,y)} + e^{-r^{(z_i)}(y,x)} + q^{(z_i)}(x,y)\right)^\tau \, |W_n|^{-1/2}.
$$
A similar argument also immediately yields that 
\begin{equation}\label{eq:QD}
\left(\mathbb{E}\left(D_x \hat G_{n, *}(z_i)\right)^4\right)^{\frac{1}{4}} \le \left(\mathbb{E}\left|D_x \hat G_{n, *}(z_i)\right|^{4+p/2}\right)^{\frac{1}{4+p/2}} \lesssim \left(S^{(z_i)}\right)^{1/4} |W_n|^{-1/2}.
\end{equation}
Since in both our examples, $S^{(z_i)}(x) \asymp S(x) = 1+ |h|^{5d/2}$ for all $i \in [m]$ and $x=(v,h)$, we use the simple bound $2(S^{(z_i)})^{1/4}(S^{(z_j)})^{1/4} \asymp S^{1/2} \le S^{(z_i)} + S^{(z_j)} \le S$. Plugging these bounds into \eqref{eq:zetas_mult_2} and \eqref{eq:zetas_mult_3}, and dropping the probability factor in \eqref{eq:D2_basic_bound} in one of the integrands in $\zeta_2^{(2)}$, we thus obtain
	\begin{equation}\label{eq:zeta23bd}
		\max\{\zeta_2^{(2)}, \zeta_3^{(2)}\}^2 \lesssim m^3 \sum_{i=1}^m \frac{\Q f_{\tau,i}^2}{|W_n|^2}.
	\end{equation}
	
	Finally, we bound $\zeta_4^{(2)}$. As argued in the proof of \cite[Lemma 5.7]{BM22}, one has
	$$
	\P \left(D_{x} \hat G_{n, *}(z_i) \neq 0\right) \le \kappa^{(z_i)}(x) + g^{(z_i)}(x). 
	$$
	Thus, H\"{o}lder's inequality and \eqref{eq:QD} again yield
	\begin{align*}
		\mathbb{E}\left|D_x \hat G_{n, *}(z_i)\right|^3 & \lesssim \left(\mathbb{E}\left|D_x \hat G_{n, *}(z_i)\right|^{4+p/2}\right)^{\frac{3}{4+p/2}} \P \left(D_x\hat G_{n, *}(z_i) \neq 0\right)^{\frac{2+p}{8+p}}\\
		& \le \left(\mathbb{E}\left|D_x \hat G_{n, *}(z_i)\right|^{4+p/2}\right)^{\frac{3}{4+p/2}} \P \left(D_x\hat G_{n, *}(z_i) \neq 0\right)^{4\tau}\\
		& \lesssim |W_n|^{-3/2} \left(S^{(z_i)}\right)^{3/4} \left((\kappa^{(z_i)}(x))^{4\tau} + (g^{(z_i)}(x))^{4\tau}\right).
	\end{align*}
	Plugging this bound into \eqref{eq:zetas_mult_4} and dropping the probability factor above in one factor in the integrand in $\zeta_4^{(2)}$, we obtain
$$
\zeta_4^{(2)} \asymp m \sum_{i,j=1}^m \frac{\Q \left((\kappa^{(z_i)} + g^{(z_i)})^{4\tau} S\right)}{|W_n|^{3/2}}.
$$

A similar bound holds in the volume-weighted case, with the functions $f, \kappa, g$ modified accordingly. Next we bound the specific integrals appearing in the terms $\zeta_1, \zeta_2^{(2)}, \zeta_3^{(2)}$ and $\zeta_4^{(2)}$ separately for the unweighted and the volume-weighted cases.

\smallskip

\noindent \underline{\em Unweighted case.} We can now evaluate these terms as in Section \ref{sec:univ_proof}. First, for each $i \in [m]$, the integral $\Q f_{\tau,i}^2$ can be bounded in exactly the same manner as in the proof of Lemma \ref{lem:fmombd} in Appendix \ref{apx:univ} via \eqref{eq:f12}, \eqref{eq:f22} and \eqref{eq:f32}, by replacing $z$ by $z_i \in [a,b]$ therein, thus making all constants dependent only on $a,b$. Thus, assuming $\E H_-^{3d}<\infty$, and arguing exactly the same way as in Section \ref{sec:univ_proof}, we obtain
$$
\Q f_{\tau,i}^2 =O( |W_n|), \quad i \in [m],
$$
with the constant depending on $a,b$. This yields $\zeta_2^{(2)}, \zeta_3^{(2)} \lesssim m^2 |W_n|^{-1/2}$. Arguing similarly for $\zeta_4^{(2)}$ using \eqref{eq:kappabd} and \eqref{eq:gbd}, assuming $\E H_-^{3d}<\infty$, one obtains $\zeta_4^{(2)} \lesssim m^3 |W_n|^{-1/2}$. 

Finally, let \(\Sigma= \Sigma(\zz) =(k(z_i,z_j))_{i,j \in [m]}\), which is a symmetric positive semidefinite matrix. Then, by Theorem \ref{thm:var}, we immediately have
\begin{equation}\label{eq:zeta1_mult}
\zeta_1
=
\sum_{i,j=1}^m
\left|
k(z_i,z_j)
-
\Cov\big(\hat G_{n,*}(z_i),\hat G_{n,*}(z_j)\big)
\right|
=
\sum_{i,j=1}^m
\left|
k(z_i,z_j)
-
\frac{\Cov\big(G_n(z_i),G_n(z_j)\big)}{|W_n|}
\right| = O(|W_n|^{-1/d}).
\end{equation}

\noindent \underline{\em Volume-weighted case.}
By our assumption, $\E H_-^{m_0d}<\infty$ with $m_0 >5.5$. Also, we take $p=12$ and consequently $\tau=3/20$. Lemma \ref{lem:fmombdV} then yields $
\Q f_{\tau,i}^2 \lesssim |W_n|
$. Similarly, applying the bound in \eqref{eq:gbdV} and the subsequent estimate corresponding to $\kappa^{(z,\ms V)}$ (noting that the bounds also hold with the exponent $4\tau$), we obtain $\Q \left((\kappa^{(z_i,\ms V)} + g^{(z_i,\ms V)})^{4\tau} S\right) \lesssim |W_n|$. Thus, we have $\zeta_2^{(2)}, \zeta_3^{(2)} \lesssim m^2 |W_n|^{-1/2}$ while $\zeta_4^{(2)} \lesssim m^3 |W_n|^{-1/2}$.

Finally, let $\Sigma^{\ms V}=\Sigma^{\ms V}(\zz)
=(k^{\ms V}(z_i,z_j))_{i,j\in[m]}$.
Arguing as in \eqref{eq:zeta1_mult} and applying
Theorem \ref{thm:var}, we obtain
\begin{equation*}
\zeta_1
=
\sum_{i,j=1}^m
\left|
k^{\ms V}(z_i,z_j)
-
\Cov\big(\hat G^{\ms V}_{n,*}(z_i),\hat G^{\ms V}_{n,*}(z_j)\big)
\right|
=
\sum_{i,j=1}^m
\left|
k^{\ms V}(z_i,z_j)
-
\frac{\Cov\big(G^{\ms V}_n(z_i),G^{\ms V}_n(z_j)\big)}{|W_n|}
\right| = O(|W_n|^{-1/d}).
\end{equation*}
Since $m$ is fixed, this completes the proof.
\end{proof}

\section{Proofs for Section~\ref{sec:tight}}
\label{apx:func}
This appendix proves the localization estimates used in the cumulant
arguments of Section~\ref{sec:tight}. We first establish the covariance
bound in Lemma~\ref{covBoundLem} and then prove the corresponding inputs
for the weighted block scores. These estimates control the dependence
between products of scores associated with separated groups of cubes.

\begin{lemma}[Deterministic localization for extremality]
\label{lem:det-local-extreme}
Let \(J=[a,b]\subset\mathbb R\), put
\(
 R_0:=\frac{\sqrt d}{2},
	\)
	and
	\(
 K_J:=[-\tfrac12,\tfrac12]^d\times J,
\)
and define
\[
 A^{(b)}:=
 \left\{(v',h')\in\X:
 (\|v'\|-R_0)_+^2+h'\le b\right\}.
\]
For every locally finite configuration \(\mu\) and every
\(y=(v,h)\in\mu|_{K_J}\), membership of \(y\) in
\(\EE(\mu)\) is determined by \(\mu|_{A^{(b)}}\).
\end{lemma}

\begin{proof}
Let \(y=(v,h)\in K_J\). If \(x'=(v',h')\) satisfies
\(L(x',v)<L(y,v)=h\), then
\(
 \|v'-v\|^2+h'<h\le b.
\)
Since \(\|v\|\le R_0\), we have
\(
 (\|v'\|-R_0)_+\le\|v'-v\|.
\)
Consequently,
\(
 (\|v'\|-R_0)_+^2+h'<b,
\)
so \(x'\in A^{(b)}\). Thus, every point capable of destroying the
extremality of \(y\) belongs to \(A^{(b)}\).
\end{proof}

\begin{lemma}[Overlap bound for deterministic extremality regions]
\label{lem:det-ext-overlap}
Assume \(\E[H_-^{3d}]<\infty\). Let \(A_i^{(b)}\) and
\(A_j^{(b)}\) be translates of \(A^{(b)}\) to the centers of unit cubes
\(z_i,z_j\). Then
\[
 \Q(A_i^{(b)}\cap A_j^{(b)})
 \in O(\|z_i-z_j\|^{-5d}).
\]
The constant may depend on \(J\), but not on \(i,j\).
\end{lemma}

\begin{proof}
By translation invariance, it is enough to consider
\(A^{(b)}\cap(A^{(b)}+(z,0))\). Put
\(
 r:=\|z\|,\) \( R_0:=\frac{\sqrt d}{2},\) \(
 \tau:=(r/2-R_0)_+,\) and \( U:=(b-H)_+.
\)
If \((v,h)\) belongs to the intersection, then
\(\max\{\|v\|,\|v-z\|\}\ge r/2\), and therefore
\(h\le b-\tau^2\), equivalently \(U\ge\tau^2\). For fixed \(h\), the
spatial section of the intersection is contained in a ball of radius
\(R_0+\sqrt{(b-h)_+}\). Hence,
\[
 \Q(A^{(b)}\cap(A^{(b)}+(z,0)))
 \le C\left(
 \P(U\ge\tau^2)
 +\E[U^{d/2}\one\{U\ge\tau^2\}]
 \right).
\]
Since
\(
 \E[U^{3d}]
 \le C_J\bigl(1+\E[H_-^{3d}]\bigr)<\infty,
\)
Markov's inequality gives, for \(\tau\ge1\) that
\(
 \P(U\ge\tau^2)\le C\tau^{-6d},
	\)
	and
	\(
 \E[U^{d/2}\one\{U\ge\tau^2\}]
 \le \tau^{-5d}\E[U^{3d}].
\)
The bounded range of \(r\) is absorbed into the constant.
\end{proof}

\begin{proof}[Proof of Lemma~\ref{covBoundLem}]
For the multi-indices in the lemma, put
\[
 \dist(\alpha,\beta):=
 \operatorname{dist}\bigl(
 \{z_{i_a}:1\le a\le r\},
 \{z_{j_b}:1\le b\le s\}
 \bigr).
\]
Let
\[
 A_i^{(b)}:=A^{(b)}+(z_i,0),
 \qquad
 A_\alpha:=\bigcup_{a=1}^rA_{i_a}^{(b)},
 \qquad
 A_\beta:=\bigcup_{b=1}^sA_{j_b}^{(b)}.
\]
Since \(Q_i\times J\subseteq A_i^{(b)}\),
Lemma~\ref{lem:det-local-extreme} shows that \(\Delta_i\) is measurable
with respect to \(\sigma(\PP|_{A_i^{(b)}})\). Thus, \(X_\alpha\) and
\(X_\beta\) are measurable with respect to the corresponding union regions.

Let \(\PP'\) be an independent copy of \(\PP\), and define
\(
 \PP^\beta
 :=
 \PP|_{(A_\alpha\cap A_\beta)^{\complement}}
 +\PP'|_{A_\alpha\cap A_\beta}.
\)
Let \(X_\beta'\) be the same functional as \(X_\beta\), evaluated at
\(\PP^\beta\). Then, \(X_\beta'\stackrel d=X_\beta\), and
\(X_\beta'\) is independent of \(X_\alpha\). Therefore,
\(
 \Cov(X_\alpha,X_\beta)
 =\E[X_\alpha(X_\beta-X_\beta')].
\)
For every \(u<\infty\), we have
\(
 \sup_{n,I,i}\|\Delta_i(I)\|_u<\infty,
\)
because
\(
 0\le\Delta_i(I)
 \le\PP((Q_i\cap W_n)\times I),
\)
and the Poisson variable on the right has mean at most \(\mathbb F(J)\le1\).
Consequently, Hölder's inequality gives
\(
 \|X_\alpha\|_p\le
 \prod_{a=1}^r\|\Delta_{i_a}\|_{pr}\le C_p,
 \) and \(
 \|X_\beta\|_p+\|X_\beta'\|_p\le C_p.
\)
Since the two point processes differ only on \(A_\alpha\cap A_\beta\), we have
\(
 \P(X_\beta\ne X_\beta')
 \le2\Q(A_\alpha\cap A_\beta).
\)
Moreover,
\[
\begin{aligned}
 \Q(A_\alpha\cap A_\beta)
 \le\sum_{a=1}^r\sum_{b=1}^s
 \Q(A_{i_a}^{(b)}\cap A_{j_b}^{(b)})
 \le C(1+\dist(\alpha,\beta))^{-5d}.
\end{aligned}
\]
Hölder's inequality  therefore yields
\(
 |\Cov(X_\alpha,X_\beta)|
 \in O_p((1+\dist(\alpha,\beta))^{-5d(1-2/p)}).
\)
Choose
\(
 p\ge{10d}/{2d-1}.
\)
Then, \(5d(1-2/p)\ge3d+1\), including when \(d=1\), and hence
\(
 |\Cov(X_\alpha,X_\beta)|
 \le C(1+\dist(\alpha,\beta))^{-3d-1}.
\)
\end{proof}

We next estimate intersections of cube envelopes, which contain the regions
determining the cell volumes in each cube. The estimate separates bounded
spatial separation from the contribution of sufficiently low marks. It
provides the geometric bound used in the weighted cumulant argument.

\bel[Measure of intersections of cube envelopes]
\label{lem:Q_intersection_cube_envelopes}
Assume \(\E[H_-^{d/2}]<\infty\). Then, there are
\(C,c_1,c_2\in(0,\infty)\), depending on \(d,\mathbb F,J\), such that for
\(z,z'\in\mathbb Z^d\), \(T,T'\in\mathbb R\), \(M:=T\vee T'\), and
\(r:=\|z-z'\|\),
\[
\begin{aligned}
 \Q\bigl(
 \widehat D_{z,J}^V(T)\cap\widehat D_{z',J}^V(T')
 \bigr)
 &\le
 C\bigl(1+(M-a)_+^{d/2}\bigr)
 \one\left\{
 r\le c_1\bigl(1+\sqrt{(M-a)_+}\bigr)
 \right\}\\
 &\quad+
 C\int_{-\infty}^{M-c_2r^2}
 \left(
 1+(M-a)_+^{d/2}+(1+u_-)^{d/2}
 \right)\mathbb F(du).
\end{aligned}
\]
\enl

\begin{proof}
If \((v,u)\) lies in both envelopes, then
\(
 r
 \le
 \sqrt d
 +\sqrt{(T-a)_+}+\sqrt{(T'-a)_+}
 +\sqrt{(T-u)_+}+\sqrt{(T'-u)_+}.
\)
The section of either envelope at mark \(u\) has volume at most
\(
 O\Big(
 (M-a)_+^{d/2}+(1+u_-)^{d/2}
 \Big).
\)
The case of bounded spatial separation follows by integration over \(u\le M\). In the
complementary case, the preceding distance inequality implies
\(u\le M-r^2/16\) after increasing \(c_1\), which gives the term arising from low marks.
\end{proof}

We also need a tail bound for the maximum cover time among contributing points in a cube.
\bel[Tail bound for the maximum cover time in a cube]
\label{lem:cube_time_tail_under_E}
Let \(I\subseteq J\), and assume \(\E[H_-^p]<\infty\).
For \(z\in\mathbb Z^d\), define
\[
 \widetilde T_z
 :=
 \sup\left\{
 T_x:
 x\in\PP|_{Q_z\times I},\
 x\in\EE(\PP)
 \right\}.
\]
Then
\(
 \sup_{z\in\mathbb Z^d}
 \P(\widetilde T_z-a\ge t)
 \le Ct^{-p},
\), \(t\ge1\).
\enl

\begin{proof}
By the union bound and Mecke's formula,
\(
 \P(\widetilde T_z-a\ge t)
 \le
 \int_I
 \P\bigl(
 (0,h)\in\EE(\PP+\delta_{(0,h)}),\
 T_{(0,h)}-a\ge t
 \bigr)\mathbb F(dh).
\)
Since \(h\in J\), we have
	\((T_{(0,h)}-h)_+\ge t-(b-a)\) whenever
\(
 T_{(0,h)}-a\ge t.
\)
For sufficiently large \(t\), Markov's inequality and
Corollary~\ref{lem:cover-time-moment-integrability} yield \(Ct^{-p}\), uniformly in
\(h\in J\). The bounded range of \(t\) is absorbed into the constant.
\end{proof}

The next lemma concerns a generator inserted through the Mecke formula. It is distinct from Lemma~\ref{lem:cube_time_tail_under_E}, which concerns Poisson points already present in a cube.

\begin{lemma}[Palm moment for an inserted generator]
\label{lem:weighted-inserted-palm-moment}
For every \(\gamma>0\),
\[
 \sup_{x=(v,h):\,h\in J}
 \E\left[
 \one\{x\in\EE(\PP+\delta_x)\}
 \bigl(1+((b\vee T(x;\PP+\delta_x))-a)_+\bigr)^\gamma
 \right]<\infty.
\]
\end{lemma}

\begin{proof}
On the extremality event, \(T(x;\PP+\delta_x)\ge h\), and, since
\(h\in[a,b]\),
\[
 1+((b\vee T(x;\PP+\delta_x))-a)_+
 \le C_J\bigl(1+(T(x;\PP+\delta_x)-h)_+\bigr).
\]
The assertion now follows from
Corollary~\ref{lem:cover-time-moment-integrability}, applied with upper level
\(b\). Spatial stationarity makes the bound independent of \(v\).
\end{proof}

The following lemma controls the part of a cube envelope that extends beyond a fixed spatial neighborhood while retaining the weight inside the expectation.

\begin{lemma}[Tail of a cube envelope outside a fixed neighborhood]
\label{lem:weighted-envelope-tail}
Let $\gamma > 3d$ and assume \(\E[H_-^{3d}]<\infty\). Let \(Z\ge0\) and \(S\ge a\) be random
variables defined on an auxiliary probability space. If
\[
 \E\left[Z\bigl(1+(S-a)_+\bigr)^\gamma\right]<\infty,
\]
then, uniformly in \(z\in\mathbb Z^d\),
\[
 \E\left[
 Z\,\Q\left(
 \widehat D_{z,J}^V(S)
 \cap\{(v,h):\operatorname{dist}(v,Q_z)>r\}
 \right)
 \right]
 \le
 C(1+r)^{-5d}
 \E\left[Z\bigl(1+(S-a)_+\bigr)^\gamma\right].
\]
\end{lemma}

\begin{proof}
Put \(U:=S-a\) and \(R_0:=\sqrt d/2\), and let \(H\sim \mathbb F\) be independent
of \((Z,S)\). For fixed values of \(S\) and \(H\), the spatial section of
\(\widehat D_{z,J}^V(S)\) at mark \(H\) is contained in the ball with center
\(z\) and radius \(R_0+\sqrt U+\sqrt{(S-H)_+}\). Its volume is therefore at
most
\(
 C\bigl(1+U^{d/2}+H_-^{d/2}\bigr),
\)
where the constant depends only on \(J\) and \(d\). If this section contains a
point at distance more than \(r\) from \(Q_z\), then
\(r-R_0<\sqrt U+\sqrt{(S-H)_+}\). For all sufficiently large \(r\), this
implies
\(
 U\ge \frac{r^2}{16}
 \;\text{ or }\;
 H_-\ge c r^2
\)
for a constant \(c>0\). Indeed, if \(U<r^2/16\), then
\(\sqrt{(S-H)_+}>r/2\); since \(S=a+U<a+r^2/16\), this forces
\(H<a-3r^2/16\), and hence \(H_-\ge cr^2\) for large \(r\).
Representing the \(\Q\)-measure by integration over the independent mark
\(H\), the expectation in the statement is at most a constant times
\(
 \E\left[Z(1+U^{d/2}+H_-^{d/2})
 \one\left\{U\ge\frac{r^2}{16}\text{ or }H_-\ge cr^2\right\}\right].
\)
For the event involving \(U\), independence of \(H\) and Markov's inequality
give
\[
\begin{aligned}
 \E\big[Z(1+U^{d/2}+H_-^{d/2})
 \one\big\{U\ge\frac{r^2}{16}\big\}\big]
 \le
 C r^{-2(\gamma-d/2)}
 \E[Z(1+U)^\gamma]\E[1+H_-^{d/2}]
 \le C r^{-5d}\E[Z(1+U)^\gamma],
\end{aligned}
\]
because \(\gamma>3d\). For the event involving \(H_-\), independence gives
\[
\begin{aligned}
 \E\left[Z(1+U^{d/2}+H_-^{d/2})\one\{H_-\ge cr^2\}\right]
 &\le
 C\E[Z(1+U)^\gamma]
 \left(\P(H_-\ge cr^2)
 +\E[H_-^{d/2}\one\{H_-\ge cr^2\}]\right)\\
 &\le C r^{-5d}\E[Z(1+U)^\gamma].
\end{aligned}
\]
The last line follows from \(\E[H_-^{3d}]<\infty\): the two terms have orders
\(r^{-6d}\) and \(r^{-5d}\), respectively. 
\end{proof}
For a locally finite configuration \(\mu\), define
\[
\Delta_i^V(I;\mu):=
\sum_{x=(v,h)\in\mu}
\one\{v\in Q_i\cap W_n,\ h\in I\}A(x;\mu).
\]

The following deterministic lemma combines the local characterization of extremality with the region that determines the cell volume.
\begin{lemma}[Comparison of configurations that agree locally]
\label{lem:weighted-deterministic-comparison}
Let \(\mu\) and \(\nu\) be locally finite configurations that agree on
\(S_r(\mathcal A)\), where \(|\mathcal A|\le3\), and assume that each
configuration has finitely many points in every downward paraboloid
\(D(w,t)\). Assume that neither
configuration has a point in
\(\bigcup_{i\in\mathcal A}A_i^{\mathrm{ext}}\setminus S_r(\mathcal A)\).
Assume also that every point
\(x\in\operatorname{supp}(\mu)\cap\operatorname{supp}(\nu)\cap Q_{\mathcal A}^I\) that is extreme in either
configuration has a bounded cell, and hence a finite cover time, in that
configuration. Finally, for each such extreme point, assume that
\[
 \mu\big(D^V(x,T(x;\mu))\setminus S_r(\mathcal A)\big)=0
\]
when \(x\) is extreme in \(\mu\), and impose the analogous condition for
\(\nu\). Then
\(
 \Delta_i^V(I;\mu)=\Delta_i^V(I;\nu),\qquad i\in\mathcal A.
\)
\end{lemma}

\begin{proof}
Put \(S:=S_r(\mathcal A)\). Since \(Q_{\mathcal A}^I\subseteq S\), the two
configurations have the same candidate generators in the blocks under
consideration. Fix one of them, say \(x=(v,h)\in Q_i^I\). If \(x\) is not extreme in \(\mu\), Lemma~\ref{lem:det-local-extreme} identifies a point
of \(\mu|_{A_i^{\mathrm{ext}}}\) that beats \(x\) at \(v\). By assumption,
this point lies in \(S\), and therefore also belongs to \(\nu\). The converse
argument is identical, so \(x\) has the same extremality status in both
configurations.

Suppose now that \(x\) is extreme in both configurations, and write
\(C_\mu:=C(x;\mu)\), \(C_\nu:=C(x;\nu)\), and
\(T_\mu:=T(x;\mu)\), \(T_\nu:=T(x;\nu)\). These cells are bounded by
assumption. We prove \(C_\mu\subseteq C_\nu\). If this inclusion fails, choose
\(u\in C_\mu\setminus C_\nu\). Since extremality implies
\(v\in C_\mu\cap C_\nu\), the segment from \(v\) to \(u\) has a first exit
point \(w\) from the closed convex set \(C_\nu\). At that point, one of the
Laguerre halfspaces defining \(C_\nu\) becomes active. Thus, there is a
competitor \(y\in\operatorname{supp}(\nu)\setminus\{x\}\) such that
\(
 L(y,w)=L(x,w)\le T_\nu,
\)
and immediately beyond \(w\) toward \(u\), one has \(L(y,\cdot)<L(x,\cdot)\).
If \(y\in S\), then \(y\in\mu\). Points of the segment immediately
beyond \(w\) belong to \(C_\mu\) by convexity, but they are beaten by
\(y\), a contradiction. Hence, \(y\notin S\). Writing \(y=(w_y,s_y)\), the equality at
\(w\) and the bound \(L(x,w)\le T_\nu\) give
\[
 s_y\le T_\nu,
	\text{ and }
 \|w_y-v\|\le
 \sqrt{(T_\nu-s_y)_+}+\sqrt{(T_\nu-h)_+}.
\]
Therefore, \(y\in D^V(x,T_\nu)\setminus S\), contrary to the hypothesis for
\(\nu\). This proves \(C_\mu\subseteq C_\nu\). Interchanging \(\mu\) and
\(\nu\) gives equality. Hence, the cell volumes and all block contributions
coincide.
\end{proof}

We next bound \(\P(\mathsf B_{\mathcal A}(r;\PP))\), the probability that localization fails.
\begin{proof}[Proof of Lemma~\ref{lem:weighted-bad-event}]
Put \(S:=S_r(\mathcal A)\). We first consider the deterministic extremality
regions. If a point belongs to
\(A_i^{\mathrm{ext}}\setminus S\), then its spatial coordinate is at distance
more than \(r\) from \(Q_i\). The definition of
\(A_i^{\mathrm{ext}}\) therefore forces its mark to be at most
\(b-c r^2\), after changing \(c>0\) by a constant depending only on \(d\).
As in the proof of Lemma~\ref{lem:det-ext-overlap}, we have
\(
 \Q(A_i^{\mathrm{ext}}\setminus S)
 \le C\int_{-\infty}^{b-cr^2}
 \bigl(1+(b-h)_+^{d/2}\bigr)\mathbb F(dh)
 \le C(1+r)^{-5d}.
\)
The last estimate uses \(\E[H_-^{3d}]<\infty\). Since
\(|\mathcal A|\le3\), the first part of the event on which localization fails has probability of the same order.

For the volume part, the union bound by the number of ordered pairs of Poisson
points, followed by the twofold Mecke formula, gives
\[
\begin{aligned}
 &\P\left(
 \exists x\in\EE(\PP)\cap Q_{\mathcal A}^I:
 \PP(D^V(x,T(x;\PP))\setminus S)>0
 \right)\\
 &\quad\le
 \sum_{i\in\mathcal A}\int_{Q_i^I}
 \E\left[
 \one\{x\in\EE(\PP+\delta_x)\}
 \Q\left(
 \widehat D_{i,J}^V(b\vee T(x;\PP+\delta_x))\setminus S
 \right)
 \right]\Q(dx).
\end{aligned}
\]
Indeed, after inserting both \(x\) and a possible exterior point \(y\),
insertion of \(y\) cannot create extremality of \(x\), and it can only shrink
the cell of \(x\). By the radius representation of the cover time,
\(
 T(x;\PP+\delta_x+\delta_y)\le T(x;\PP+\delta_x).
\)
Lemma~\ref{lem:env} then places the corresponding region that determines the score inside the
envelope in the last display.
For fixed \(x\in Q_i^I\), apply
Lemma~\ref{lem:weighted-envelope-tail} with
\(
 Z:=\one\{x\in\EE(\PP+\delta_x)\},
 \) and \(
 S_x:=b\vee T(x;\PP+\delta_x).
\)
The required weighted moment is supplied by
Lemma~\ref{lem:weighted-inserted-palm-moment}. Since a point outside
\(S_r(\mathcal A)\) is at distance more than \(r\) from \(Q_i\), the lemma
gives, uniformly in \(x\in Q_i^I\), that
\(
 \E\left[
 Z\,\Q\left(\widehat D_{i,J}^V(S_x)\setminus S\right)
 \right]
 \le C(1+r)^{-5d}.
\)
Finally, \(\Q(Q_i^I)\le\mathbb F(I)\le1\). Combining the extremality and volume estimates proves
the claim.
\end{proof}

We can now prove the approximation by conditional expectation.
\begin{proof}[Proof of Lemma~\ref{lem:weighted-spatial-truncation}]
Put \(S:=S_r(\mathcal A)\) and
\(\mathcal F:=\sigma(\PP|_S)\). Let \(\PP'\) be an independent copy of
\(\PP\), and define
\(
 \PP^*:=\PP|_S+\PP'|_{S^{\complement}}.
\)
Conditionally on \(\mathcal F\), the restrictions of \(\PP\) and \(\PP^*\)
to \(S^{\complement}\) are independent Poisson processes with the same
intensity. Thus, if \(X_{\mathcal A}^*\) denotes the same centered block
product evaluated at \(\PP^*\), then
\(\E[X_{\mathcal A}^*\mid\mathcal F]
 =\E[X_{\mathcal A}\mid\mathcal F]=X_{\mathcal A}^{(r)}\). Conditioning
first on \(\PP\) and applying Jensen's inequality gives
\(
 \|X_{\mathcal A}-X_{\mathcal A}^{(r)}\|_{a_0}
 \le \|X_{\mathcal A}-X_{\mathcal A}^*\|_{a_0}.
\)
The processes \(\PP\) and \(\PP^*\) agree on \(S\), and both have bounded
Laguerre cells almost surely because \(\PP^*\) has the same law as \(\PP\).
On the complement of \(\mathsf B_{\mathcal A}(r;\PP)\cup
\mathsf B_{\mathcal A}(r;\PP^*)\), the remaining hypotheses of
Lemma~\ref{lem:weighted-deterministic-comparison} hold, and hence
\(X_{\mathcal A}=X_{\mathcal A}^*\). Put
\(
 E_r:=\mathsf B_{\mathcal A}(r;\PP)
 \cup\mathsf B_{\mathcal A}(r;\PP^*).
\)
For the fixed exponents \(1<a_0<p\), H\"older's inequality gives
\(
 \|X_{\mathcal A}-X_{\mathcal A}^*\|_{a_0}
 \le
 \|X_{\mathcal A}-X_{\mathcal A}^*\|_p
 \P(E_r)^{1/a_0-1/p}
 \le C\P(E_r)^{1/a_0-1/p}.
\)
The last inequality follows from
Lemma~\ref{lem:weighted-block-product-moments} and the fact that
\(\PP^*\) has the same law as \(\PP\). Lemma~\ref{lem:weighted-bad-event}
therefore yields
\(
 \|X_{\mathcal A}-X_{\mathcal A}^{(r)}\|_{a_0}
 \le C(1+r)^{-5d(1/a_0-1/p)}
	\in O(r^{-3d-1}).
\)
\end{proof}

\end{document}